\documentclass[11pt,leqno]{article}
\usepackage[margin=1.0in]{geometry}
\usepackage{lmodern}
\usepackage{amsmath,amsthm,amssymb,mathtools,mathrsfs,tikz,graphicx}
\usepackage{aliascnt}
\usepackage{xcolor}
\usepackage{enumitem}
\usepackage{microtype}
\usepackage{tocloft}
\usepackage{hyperref}
\usepackage[nameinlink,capitalise,noabbrev]{cleveref}
\usetikzlibrary{arrows.meta}
\hypersetup{hidelinks}
\allowdisplaybreaks
\newtheorem{theorem}{Theorem}[section]
\newaliascnt{lemma}{theorem}
\newtheorem{lemma}[lemma]{Lemma}
\aliascntresetthe{lemma}
\newaliascnt{proposition}{theorem}
\newtheorem{proposition}[proposition]{Proposition}
\aliascntresetthe{proposition}
\newaliascnt{corollary}{theorem}
\newtheorem{corollary}[corollary]{Corollary}
\aliascntresetthe{corollary}
\theoremstyle{definition}
\newaliascnt{definition}{theorem}
\newtheorem{definition}[definition]{Definition}
\aliascntresetthe{definition}
\newtheorem*{example}{Example}
\newtheorem*{terminology}{Definition}
\theoremstyle{remark}
\newaliascnt{remark}{theorem}
\newtheorem{remark}[remark]{Remark}
\aliascntresetthe{remark}

\newcommand{\wh}[1]{\widehat{#1}}
\newcommand{\cl}[1]{\overline{#1}}
\newcommand{\Z}{\mathbb Z}
\newcommand{\Q}{\mathbb Q}
\newcommand{\R}{\mathbb R}
\newcommand{\F}{\mathbb F}
\newcommand{\HH}{\mathbb H}
\newcommand{\PSL}{\operatorname{PSL}}
\newcommand{\Out}{\operatorname{Out}}
\newcommand{\Aut}{\operatorname{Aut}}
\newcommand{\Inn}{\operatorname{Inn}}
\newcommand{\Stab}{\operatorname{Stab}}
\newcommand{\Fix}{\operatorname{Fix}}
\newcommand{\ind}{\operatorname{ind}}
\newcommand{\cd}{\operatorname{cd}}
\newcommand{\HNN}{\operatorname{HNN}}
\newcommand{\im}{\operatorname{im}}
\newcommand{\Hom}{\operatorname{Hom}}
\newcommand{\Ncl}[2]{\langle\!\langle #1\rangle\!\rangle_{#2}}
\newcommand{\conj}[1]{\operatorname{conj}_{#1}}
\newcommand{\acts}{\curvearrowright}
\newcommand{\normal}{\trianglelefteq}
\newcommand{\normalf}{\trianglelefteq_{\!f}}
\newcommand{\normalo}{\trianglelefteq_{\!o}}
\newcommand{\cross}{\pitchfork}
\newcommand{\normalcore}[2]{\operatorname{core}_{#1}(#2)}
\newcommand{\Icell}{\mathcal I}
\newcommand{\Dgraph}{\mathcal D}
\newcommand{\Daction}{\operatorname{Diag}}
\newcommand{\id}{\operatorname{id}}
\newcommand{\C}{\mathbb C}
\newcommand{\hZ}{\widehat{\Z}}
\newcommand{\tors}{\operatorname{tors}}
\newcommand{\tr}{\operatorname{tr}}
\newcommand{\lcm}{\operatorname{lcm}}
\newcommand{\GL}{\operatorname{GL}}
\newcommand{\Int}{\operatorname{int}}
\numberwithin{equation}{section}

\title{Cubical residual complexes and profinite rigidity of real hyperbolic lattices}

\author{Yong Hou}
\date{}
\begin{document}
\maketitle

\begin{abstract}
We study profinite rigidity of real hyperbolic lattices, with the main application in dimension
three. The basic tool is a profinite version of the dual cube complex of a pair of filling curves on
a surface, which we call a profinite cubical residual complex. We prove that an equivariant
isomorphism between the completions of two such complexes is induced, after conjugation by a single
element, by an isomorphism of the discrete groups and complexes, and that the same holds for free
cocompact cellular actions of finitely generated residually finite groups on connected, locally
finite, finite dimensional regular CW complexes. On a closed orientable surface of genus at least
two, a theory of divisible fillings determines the completed cut tree of a nonseparating simple
closed curve. Combined with the homological intersection criterion of Boggi and Zalesskii, this
determines the completed dual square complex of a nonseparating filling pair, and hence realizes
every isomorphism of profinite surface groups which preserves the two marked cyclic subgroups. In
dimension three we combine these results with virtual fibering, profinite invariants of
hyperbolic $3$-manifolds and the goodness of $3$-manifold groups. Massey products and virtual
domination give cohomological integrality, and pseudo-Anosov dynamics, periodic orbit traces and
cross sections of suspension flows give a correspondence of prime periodic orbits together with the
fiber curves needed for the realization. Profinite rigidity of closed and of cusped lattices
follows, the cusped case through cyclic orbifold fillings, and we obtain
$\Out(\Gamma)\cong\Out(\wh\Gamma)$ for every lattice of $\PSL_2(\C)$.
\end{abstract}

\tableofcontents


\section{Introduction}
For $G$ a finitely generated residually finite group, its profinite completion $\wh G$ determines
all finite quotients of $G$ \cite{RibesZalesskii2010}. However, it does not in general determine the
dense copy of $G$ inside $\wh G$; in fact, a completed action does not necessarily determine a
discrete action of $G$, and one may ask when it does. The purpose of this paper is to answer this
question. We show how to determine discrete actions from equivariant isomorphisms of profinite
completions, and we apply this to the profinite rigidity of hyperbolic lattices.

Profinite rigidity of real hyperbolic lattices is a well known problem with both algebraic and
geometric origins \cite{Grothendieck1970,WiltonZalesskii2017}. Our approach is to build spaces on
which a profinite action can be followed. We use Sageev's cube complexes \cite{Sageev1995}, and we
generalize the construction of Guirardel's core \cite{Guirardel2005}, the core of a product of two
trees, which Guirardel studied in connection with irreducible outer automorphisms of free groups
acting on $\R$-trees, to the dual square complex of a divisible filling pair of curves. With the
resulting profinite cubical residual complexes
(\cref{def:cubical-occurrence,def:cubical-action-completion}) one can track a profinite action
through the completion (\cref{thm:core-main}) and then determine the original G-action from it
(\cref{thm:discrete-star}). This last step works in every dimension
(\cref{thm:all-dimensional-occurrence}), and we expect it to be useful for profinite rigidity in
higher dimensions.

In dimension two we work with nonseparating filling pairs of curves satisfying divisibility
conditions. Let $S$ be a closed orientable surface of genus at least two, let $\Pi=\pi_1(S)$, and
let $A=\langle a\rangle$ be the cyclic subgroup represented by a nonseparating simple closed curve
$\alpha$. Cutting $S$ along $\alpha$ gives an efficient HNN splitting of $\Pi$ whose vertex group
$H$ is the fundamental group of the cut surface \cite{SerreTrees2003,Ribes2017}. Imposing the
relation $a^{\ell_n}=1$, where $\ell_n=\lcm(1,\ldots,n)$, produces an HNN extension with a finite
edge group and a cocompact Fuchsian vertex group; we call these quotients divisible fillings. A
fixed point theorem for profinite Poincar\'e duality groups \cite[Theorem~1.10 and
Corollary~1.11]{WiltonZalesskii2019} identifies the completed vertex group of each filling, and if
$P_{\alpha,\ell_n}$ denotes its inverse image in $\wh\Pi$, then
$\cl H=\bigcap_{n\geq3}P_{\alpha,\ell_n}$. The conjugators appearing at the different levels can be
chosen compatibly, because the relevant sets form a nested family of nonempty compact sets, and
profinite malnormality \cite[Theorem~3.3]{WiltonZalesskii2017} detects adjacency in the profinite
Bass--Serre tree. In this way the divisible fillings determine the full profinite cut tree of
$\alpha$ (\cref{thm:finite-cone-main}). For an ordered nonseparating filling pair $(\alpha,\beta)$,
the dual square complex $C(\alpha,\beta)$ is the surface version of the Guirardel core in
$T_\alpha\times T_\beta$. We complete its vertex, edge, square, endpoint occurrence and side
occurrence spaces, together with all attaching maps. Crossing of edges is detected in finite covers
by the homological criterion of Boggi and Zalesskii \cite{BoggiZalesskii2017}, and since this
criterion is insensitive to replacing the two cyclic generators by independent profinite unit
powers, the two marked procyclic conjugacy classes determine an equivariant isomorphism of the
completed cores (\cref{thm:core-main}).

The original G-action is then determined from the completed occurrence structure. The point is that
once the profinite action sends some original vertex to an original vertex, the attaching maps
determine its original star; by connectedness this propagates through the whole complex and also
determines the dense acting group (\cref{thm:discrete-star}). By induction over the skeleta the same
conclusion holds for free cocompact actions on connected, locally finite, finite dimensional regular
CW complexes (\cref{thm:all-dimensional-occurrence}).

$3$-dimensional hyperbolic geometry is one of the most developed theories in geometry, and a
tremendous amount of work has been done on the profinite rigidity question in dimension $3$
\cite{WiltonZalesskii2017,WiltonZalesskii2019,Liu2023}. What is so special about dimension $3$ is
the fact that all geometric structures are determined by surfaces, as is already evident from the
basic example of a mapping torus. It is this special property that distinguishes dimension $3$ from
higher dimensions, and the theory of hyperbolic $3$-manifolds is especially successful because of
it. In fact, we have a pretty much complete understanding of $3$-manifold groups through
Perelman--Thurston geometrization \cite{AschenbrennerFriedlWilton2015}, Agol's virtual fibration
theorem \cite{Agol2008,Agol2013,Wise2021}, JSJ decompositions \cite{JacoShalen1979} and the surface
subgroup theorem \cite{KahnMarkovic2012}. It is safe to say that all the ingredients exist to
establish profinite rigidity in this dimension, except the realization part. Hence, by combining
these high-powered results with our realization theorems, we are able to complete the proof of
profinite rigidity of lattices in $\PSL_2(\C)$ (\cref{thm:lattice-rigidity}). Given an isomorphism
between the profinite completions of two finite volume hyperbolic $3$-manifold groups, we first
produce surface curves to which the realization theorems apply. Liu's regularity theorem
\cite[Theorems~1.2--1.3]{Liu2023}, a triple Massey product evaluated on an explicit integral
Heisenberg cocycle, and the virtual domination theorem of Liu and Sun \cite{LiuSun2018} give
cohomological integrality (\cref{thm:closed-integral}): the cubic and quartic evaluations involve
the same orientation unit but different powers of the scalar $\mu$, which forces $\mu=\pm1$.
Comparing twisted torsion with fixed point traces \cite{Liu2023,Jiang1996} and using separation in
finite quotients \cite{HamiltonWiltonZalesskii2013} together with compactness, we obtain a
correspondence between the prime periodic orbits of the two suspension flows. Two independent
fibered classes place a matched orbit in corresponding surface fibers: we perturb a class in the
closure of a fibered cone \cite{Agol2008,Thurston1986}, choose an orbit on which it vanishes
\cite{Fried1982,FathiShub1979}, and pass to characteristic simple elevations \cite{Scott1978}. Large
pseudo-Anosov iterates then give two nonseparating filling pairs \cite{MasurMinsky1999}, and the
realization theorems apply (\cref{cor:closed-rigidity}).

In the cusped case, each peripheral map is a profinite unit times an integral map, so the closed
normal subgroups defining divisible peripheral fillings are preserved, because
$\cl{\langle e^{k\mu}\rangle}=\cl{\langle e^k\rangle}$ for $\mu\in\hZ^\times$. The closed case,
Mostow--Prasad rigidity and simultaneous drilling \cite{BoileauPorti2001,Mostow1968,Prasad1973} give
homeomorphisms of the unfilled compact cores. Every fixed finite quotient factors through the
fillings of all sufficiently large index, only finitely many discrete outer classes occur, and the
conjugator sets are compact, so we determined the original cusped isomorphism
(\cref{cor:cusped-rigidity}). Passing to finite extensions proves \cref{thm:lattice-rigidity},
including $\Out(\Gamma)\cong\Out(\wh\Gamma)$ for Kleinian lattices, and cusp integrality follows.

Background on profinite groups, graphs of groups and their trees can be found in
\cite{SerreTrees2003,RibesZalesskii2010,Ribes2017}; for regular CW complexes, PL maps and hyperbolic
geometry see \cite{Hatcher2002,RourkeSanderson1982,Thurston1986}.

The main results are collected in \cref{sec:main-results}.
Sections~\ref{sec:preliminaries}--\ref{sec:graph-occurrence} contain the profinite preliminaries and
the realization theorems for cubical and graph residual complexes.
Sections~\ref{sec:closed-sign}--\ref{sec:cusps} prove the cohomological, dynamical and filling
results used for lattices. In \cref{sec:nonlattice} we treat marked nonlattice groups and compact
cores.

\subsection*{Definitions and notations}


\begin{terminology}
\emph{original} $G$-object:
Let $G$ be a residually finite group and let $Z$ be a left $G$-set with finitely many $G$-orbits.
Its profinite completion is $\wh Z=\varprojlim_{N\normalf G}N\backslash Z$, with the natural map
$\iota_Z:Z\longrightarrow\wh Z$, $z\longmapsto(Nz)_{N\normalf G}$. A point $\wh z\in\wh Z$ is called
\emph{original} if $\wh z\in\iota_Z(Z)$; equivalently, there is $z\in Z$ whose image $Nz$ is the
$N$-th coordinate of $\wh z$ for every $N\normalf G$. This terminology is always relative to the
$G$-set $Z$ in question. Whenever $\iota_Z$ is injective, we identify $Z$ with its original image.
Taking $Z=G$ with the left multiplication action, the original elements of $\wh G$ are the elements
of its canonical subgroup $G$. For a completed complex, the definition applies separately to every
cell and occurrence space defined later, and the \emph{original complex} with its \emph{original
$G$-action} is the discrete complex with its discrete action. The same definition applies to
completions formed over any cofinal family of finite index normal subgroups.\end{terminology}
We assume all profinite homomorphisms are continuous. If $A\subseteq G$, we write $\cl A$ for the
closure of its image in $\wh G$. For $f:G\to H$, we write $\wh f:\wh G\to\wh H$ for the induced
continuous homomorphism. For $X\subseteq G$, we set
$\Ncl{X}{G}=\bigcap_{\substack{N\normal G\\X\subseteq N}}N$, the normal closure of $X$ in $G$.

\begin{terminology}
We say that an isomorphism $\Phi:\wh G\to\wh H$ is \emph{discretely induced up to profinite inner
automorphism} if
\[
\Phi=\Inn(u)\circ\wh f
\]
for an isomorphism $f:G\to H$ of the discrete groups and some $u\in\wh H$.
\end{terminology}


\begin{definition}\label{def:cut-tree}
A simple closed curve is the image of an embedding of a circle. It is nonseparating if its
complement in the surface is connected. Let $\alpha$ be such a curve on $S$. We write
\[
\Pi=\langle H,t\mid tA_-t^{-1}=A_+\rangle,
\qquad H=\pi_1(S_\alpha),
\]
where $A_\pm$ are the two cyclic boundary subgroups, $A=A_-$, and $t$ is the stable letter. We set
$S_\alpha=S\setminus\Int N(\alpha)$, where $N(\alpha)$ is a closed annular neighborhood of $\alpha$.
This compact cut surface is homotopy equivalent to $S\setminus\alpha$. Based paths identify its
fundamental group with the subgroup $H\leq\Pi$. We denote the original Bass--Serre tree by
$T_\alpha$. Efficiency of the splitting, defined below, identifies $\wh H$ and $\wh A$ with their
closures in $\wh\Pi$, so the standard profinite tree, again denoted $\wh T_\alpha$, has vertex and
edge spaces
\[
V(\wh T_\alpha)=\wh\Pi/\cl H,
\qquad E(\wh T_\alpha)=\wh\Pi/\cl A.
\]
The tree is the one associated to the completed HNN splitting before imposing any finite cyclic
curve power relation. We use \emph{full profinite cut tree} only in this sense. See
\cite{SerreTrees2003} and \cite{Ribes2017}.
\end{definition}

\begin{definition}
Let $(\mathcal G,Y)$ be a finite graph of groups with residually finite fundamental group $\Gamma$.
The splitting is \emph{efficient} when every vertex and edge group is closed in the profinite
topology of $\Gamma$ and the topology induced on each of them is its full profinite topology. A
graph of profinite groups is \emph{proper} if every vertex group maps injectively to its profinite
fundamental group. See \cite{Ribes2017}.
\end{definition}

\begin{definition}
A cyclic subgroup is maximal cyclic if it is contained in no larger cyclic subgroup. Given a
nonseparating simple closed curve, let $a$ be a based representative that traverses it once, and set
$A=\langle a\rangle\leq\Pi$. This subgroup is maximal cyclic. The marked pair is denoted by
$(\wh\Pi,[\cl A])$. An isomorphism of marked pairs is an isomorphism $\theta:\wh\Pi_1\to\wh\Pi_2$
such that $\theta(\cl A_1)$ is conjugate to $\cl A_2$.
\end{definition}

\begin{definition}\label{def:cubical-occurrence}
Let $C$ be a regular square ($2$-dimensional cube) complex. We choose for each cell a characteristic
map that is a homeomorphism onto its closure. We write $C^k$ for the set of open $k$-cells. For an
open edge $e$ and an open square $q$, we fix the characteristic maps given by the cubical structure,
\[
\chi_e:[-1,1]\longrightarrow\cl e,
\qquad
\chi_q:[-1,1]^2\longrightarrow\cl q.
\]
For $\epsilon\in\{-1,+1\}$, the pair $(e,\epsilon)$ is an endpoint occurrence of $e$. For
$j\in\{1,2\}$ and $\epsilon\in\{-1,+1\}$, set
\[
F_{j,\epsilon}=\{(x_1,x_2)\in[-1,1]^2:x_j=\epsilon\}.
\]
The triple $(q,j,\epsilon)$ is the corresponding side occurrence of $q$. We define
\[
I_{10}(C)=\{(e,\epsilon):e\in C^1,\ \epsilon\in\{-1,+1\}\},
\]
\[
I_{21}(C)=\{(q,j,\epsilon):q\in C^2,\ j\in\{1,2\},\
\epsilon\in\{-1,+1\}\}.
\]
The structure maps are
\[
\begin{aligned}
p_{10}(e,\epsilon)&=e,
&q_{10}(e,\epsilon)&=\chi_e(\epsilon),\\
p_{21}(q,j,\epsilon)&=q,
&q_{21}(q,j,\epsilon)&=e,
\end{aligned}
\]
where $e$ is the unique open edge satisfying
\[
\chi_q\bigl(F^o_{j,\epsilon}\bigr)=e,
\]
where $F^o_{1,\epsilon}=\{\epsilon\}\times(-1,1),F^o_{2,\epsilon}=(-1,1)\times\{\epsilon\}.$
The \emph{cubical residual complex} of $C$ is
\[
\Icell(C)=
\bigl(C^0,C^1,C^2,I_{10}(C),I_{21}(C),
p_{10},q_{10},p_{21},q_{21}\bigr).
\]
Distinct occurrence orbits remain distinct even when their incident cell orbits coincide.

Let $C'$ be another regular square complex. A \emph{morphism} $F:\Icell(C)\to\Icell(C')$ is a tuple
of five maps
\[
F=(F_0,F_1,F_2,F_{10},F_{21})
\]
commuting with the four structure maps:
\[
\begin{aligned}
F_1p_{10}&=p'_{10}F_{10},
&F_0q_{10}&=q'_{10}F_{10},\\
F_2p_{21}&=p'_{21}F_{21},
&F_1q_{21}&=q'_{21}F_{21}.
\end{aligned}
\]
It is an isomorphism when its five component maps are bijective. In the profinite setting the
component maps are required to be continuous, and a continuous component-wise bijection is a
homeomorphism because the component spaces are compact and Hausdorff.
\end{definition}
A \emph{finite cubical residual diagram} is a diagram of finite sets
\[
\mathrm I=
\bigl(
C^0,C^1,C^2,I_{10},I_{21},
p_{10},q_{10},p_{21},q_{21}
\bigr),
\]
with structure maps
\[
p_{10}:I_{10}\longrightarrow C^1,
\qquad
q_{10}:I_{10}\longrightarrow C^0,
\]
\[
p_{21}:I_{21}\longrightarrow C^2,
\qquad
q_{21}:I_{21}\longrightarrow C^1,
\]
which is isomorphic to a component-wise quotient
$H\backslash\mathcal I(C)$, where $C$ is a regular square
complex and $H$ acts cellularly on $C$ with finitely many
orbits on each cell and occurrence set.
The quotient structure maps are induced by those of
$\mathcal I(C)$.
An isomorphism here consists of bijections on the five sets
commuting with the four structure maps.
Taking $H=\{1\}$ includes the cubical residual complex of
every finite regular square complex.
\begin{definition}
\label{def:profinite-cubical-occurrence}
Let $\Lambda$ be a nonempty directed partially ordered set. Let $\lambda\in\Lambda$, let
\[
\mathrm I_\lambda
=
\bigl(
C^0_\lambda,C^1_\lambda,C^2_\lambda,
I_{10,\lambda},I_{21,\lambda},
p_{10,\lambda},q_{10,\lambda},
p_{21,\lambda},q_{21,\lambda}
\bigr)
\]
be a finite cubical residual diagram. We allow both the residual complexes of
\cref{def:cubical-occurrence} and the finite diagrams obtained by taking component-wise quotients of
the cell and occurrence spaces of a regular square complex by a cellular group action. Such a
quotient diagram is not necessarily the residual complex of a regular quotient square complex. A
morphism of these diagrams is a tuple of five maps commuting with the four structure maps, as in
\cref{def:cubical-occurrence}. The structure maps have domains and ranges
\[
\begin{aligned}
p_{10,\lambda}&:I_{10,\lambda}\longrightarrow C^1_\lambda,
&q_{10,\lambda}&:I_{10,\lambda}\longrightarrow C^0_\lambda,\\
p_{21,\lambda}&:I_{21,\lambda}\longrightarrow C^2_\lambda,
&q_{21,\lambda}&:I_{21,\lambda}\longrightarrow C^1_\lambda.
\end{aligned}
\]

For $\lambda\leq\mu$, let $r_{\mu\lambda}:\mathrm I_\mu\to\mathrm I_\lambda$ be a morphism as
defined above. For each component $X\in\{C^0,C^1,C^2,I_{10},I_{21}\}$, write
\[
r_{\mu\lambda}^{X}:X_\mu\longrightarrow X_\lambda
\]
for its corresponding component map. We require
\[
r_{\lambda\lambda}^{X}=\id_{X_\lambda},
\qquad
r_{\nu\lambda}^{X}
=r_{\mu\lambda}^{X}\circ r_{\nu\mu}^{X}
\quad(\lambda\leq\mu\leq\nu).
\]
The requirement that $r_{\mu\lambda}$ be a morphism means that
\[
\begin{aligned}
r_{\mu\lambda}^{C^1}\circ p_{10,\mu}
&=p_{10,\lambda}\circ r_{\mu\lambda}^{I_{10}},\\
r_{\mu\lambda}^{C^0}\circ q_{10,\mu}
&=q_{10,\lambda}\circ r_{\mu\lambda}^{I_{10}},\\
r_{\mu\lambda}^{C^2}\circ p_{21,\mu}
&=p_{21,\lambda}\circ r_{\mu\lambda}^{I_{21}},\\
r_{\mu\lambda}^{C^1}\circ q_{21,\mu}
&=q_{21,\lambda}\circ r_{\mu\lambda}^{I_{21}}.
\end{aligned}
\]

We give each finite component set the discrete topology. Then for each of the five components we
define:
\[
\wh X
=
\varprojlim_{\lambda\in\Lambda}X_\lambda
=
\left\{
(x_\lambda)_\lambda\in\prod_{\lambda\in\Lambda}X_\lambda:
r_{\mu\lambda}^{X}(x_\mu)=x_\lambda
\text{ whenever }\lambda\leq\mu
\right\},
\]
with the subspace topology inherited from the product. This is a closed subspace of a product of
finite discrete spaces, hence a profinite space. Denote its coordinate projections by
\[
\pi_\lambda^{X}:\wh X\longrightarrow X_\lambda,
\qquad
(x_\eta)_\eta\longmapsto x_\lambda.
\]
The nonempty sets
\[
(\pi_\lambda^{X})^{-1}(\{x\}),
\qquad \lambda\in\Lambda,\quad x\in X_\lambda,
\]
form a basis of both closed and open sets. Define the four structure maps by
\[
\begin{aligned}
\wh p_{10}((\iota_\lambda)_\lambda)
&=(p_{10,\lambda}(\iota_\lambda))_\lambda,\\
\wh q_{10}((\iota_\lambda)_\lambda)
&=(q_{10,\lambda}(\iota_\lambda))_\lambda,\\
\wh p_{21}((\sigma_\lambda)_\lambda)
&=(p_{21,\lambda}(\sigma_\lambda))_\lambda,\\
\wh q_{21}((\sigma_\lambda)_\lambda)
&=(q_{21,\lambda}(\sigma_\lambda))_\lambda,
\end{aligned}
\]
where $(\iota_\lambda)_\lambda\in\wh I_{10}$ and $(\sigma_\lambda)_\lambda\in\wh I_{21}$. The
commutation identities above ensure that these sequences belong to the inverse limits in question.
For example,
\[
\begin{aligned}
r_{\mu\lambda}^{C^1}
\bigl(p_{10,\mu}(\iota_\mu)\bigr)
&=p_{10,\lambda}
\bigl(r_{\mu\lambda}^{I_{10}}(\iota_\mu)\bigr)\\
&=p_{10,\lambda}(\iota_\lambda).
\end{aligned}
\]
In addition, each resulting map $\wh s:\wh X\to\wh Y$ is continuous and satisfies
\[
\pi_\lambda^{Y}\circ\wh s
=s_\lambda\circ\pi_\lambda^{X}
\qquad(\lambda\in\Lambda).
\]

Then the resulting diagram of profinite spaces
\[
\wh{\mathrm I}
=
\bigl(
\wh C^0,\wh C^1,\wh C^2,
\wh I_{10},\wh I_{21},
\wh p_{10},\wh q_{10},
\wh p_{21},\wh q_{21}
\bigr)
=
\varprojlim_{\lambda\in\Lambda}\mathrm I_\lambda
\]
is called the \emph{profinite cubical residual complex} of this inverse system.
\end{definition}
\begin{definition}
\label{def:cubical-action-completion}
Let a residually finite group $\Gamma$ act cellularly on a regular square ($2$-cube) complex $C$, with finitely
many $\Gamma$-orbits in each cell and occurrence space. Let $\mathscr N$ be a cofinal family of
finite index normal subgroups, ordered by reverse inclusion. Then for $N'\leq N$ and any component
space $Z$, the \emph{attaching map} is given by
\[
r^Z_{N',N}:N'\backslash Z\longrightarrow N\backslash Z,
\qquad N'z\longmapsto Nz.
\]
This is well defined and surjective. Equivariance of the four structure maps shows that the five
attaching maps form a morphism $N'\backslash\Icell(C)\longrightarrow N\backslash\Icell(C)$, and we
obtain the inverse system
\[
\wh{\Icell}_{\mathscr N}(C)
=\varprojlim_{N\in\mathscr N}N\backslash\Icell(C).
\]
If the action is free and cocompact, then every component space $Z$ is a finite disjoint union of
free $\Gamma$-orbits and
\[
Z=\bigsqcup_{j=1}^{r_Z}\Gamma z_j,
\qquad
\wh Z=\bigsqcup_{j=1}^{r_Z}\wh\Gamma z_j.
\]
In addition, if an orbit representative has finite stabilizer $H_j$, its completed orbit is
$\wh\Gamma/H_j$.

Completion is carried out separately for the occurrence spaces and for the cell spaces, so in
general there are no canonical identifications
\[
N\backslash I_{10}(C)\cong(N\backslash C^1)\times\{-1,+1\},
\]
or
\[
N\backslash I_{21}(C)\cong(N\backslash C^2)\times
\{1,2\}\times\{-1,+1\},
\]
because an element of $N$ may reverse endpoint positions or permute side positions.
\end{definition}

\begin{example}
Let $C$ be the standard unit square tiling of $\R^2$. We write
\[
v_{m,n}=(m,n),\qquad
a_{m,n}=[m,m+1]\times\{n\},\qquad
b_{m,n}=\{m\}\times[n,n+1],
\]
\[
q_{m,n}=[m,m+1]\times[n,n+1].
\]
The translation group $\Gamma=\Z^2$ acts freely and cocompactly. The quotient cell sets contain one
vertex $v$, two edges $a,b$, and one square $q$, but
\[
\Gamma\backslash I_{10}(C)=\{a^-,a^+,b^-,b^+\},
\]
where
\[
p_{10}(a^\pm)=a,\qquad p_{10}(b^\pm)=b,
\qquad q_{10}(a^\pm)=q_{10}(b^\pm)=v.
\]
For the square side occurrences,
\[
\Gamma\backslash I_{21}(C)
=\{q_{1,-},q_{1,+},q_{2,-},q_{2,+}\},
\]
with
\[
p_{21}(q_{j,\epsilon})=q,
\qquad
q_{21}(q_{1,-})=q_{21}(q_{1,+})=b,
\qquad
q_{21}(q_{2,-})=q_{21}(q_{2,+})=a.
\]
The two vertical side positions are distinct in the quotient occurrence set although both map to
$b$. Likewise, the two horizontal positions are distinct although both map to $a$. Replacing the
occurrence sets by the relations
\[
\{(a,v),(b,v)\},\qquad \{(q,a),(q,b)\}
\]
would remove these multiplicities. We therefore take every inverse limit on the occurrence sets as
well as on the cell sets.
\end{example}

\begin{definition}\label{def:core}
Let $S$ be a connected closed orientable surface of genus $\ge 2$ and let $\Pi=\pi_1(S)$. Let
$(\alpha,\beta)$ be an ordered pair of nontrivial simple closed curves such that both
$S\setminus\alpha$ and $S\setminus\beta$ are connected. We choose representatives such that
\[
|\alpha\cap\beta|=i(\alpha,\beta),
\]
where $i(\alpha,\beta)$ is the minimum number of transverse intersections among representatives of
their free homotopy classes, called their geometric intersection number. Assume further that every
component of $S\setminus(\alpha\cup\beta)$ is an open disk. Such an ordered pair is called a
\emph{nonseparating filling pair}.

We choose a hyperbolic metric on $S$ and replace $\alpha$ and $\beta$ by their geodesic
representatives. Let $\mathscr A$ and $\mathscr B$ be the $\Pi$-invariant families of their geodesic
lifts in $\widetilde S\cong\HH^2$. Elements of the same family are disjoint, an element of
$\mathscr A$ meets an element of $\mathscr B$ in at most one point, and $\mathscr A\cup\mathscr B$
is locally finite.

Following Sageev's construction of the cube complex dual to a wallspace
\cite{Sageev1995,HruskaWise2014}, we define the \emph{dual square complex} $C(\alpha,\beta)$ by the
following cells and attaching maps:
\begin{enumerate}[label=\textup{(\alph*)}]
\item For every component $R$ of
$\widetilde S\setminus\bigcup_{L\in\mathscr A\cup\mathscr B}L$, introduce one vertex
$v_R$.
\item A \emph{wall segment} is a connected component of
\[
L\setminus\bigcup_{M\in(\mathscr A\cup\mathscr B)\setminus\{L\}}M,
\qquad L\in\mathscr A\cup\mathscr B.
\]
For every wall segment $s$, introduce an abstract open edge $e_s$. Its two endpoint
positions are attached to the vertices corresponding to the two complementary components
adjacent to $s$.
\item Let $x=L_\alpha\cap L_\beta$ be a crossing, with $L_\alpha\in\mathscr A$ and
$L_\beta\in\mathscr B$. The four half branches at $x$ lie in four wall segments
$s_1,s_2,s_3,s_4$, listed in their cyclic order about $x$. Introduce a model square
$[-1,1]^2$ and attach its four side positions, in cyclic order, to
$e_{s_1},e_{s_2},e_{s_3},e_{s_4}$. Its four vertices are attached to the four
complementary components incident to $x$.
\end{enumerate}


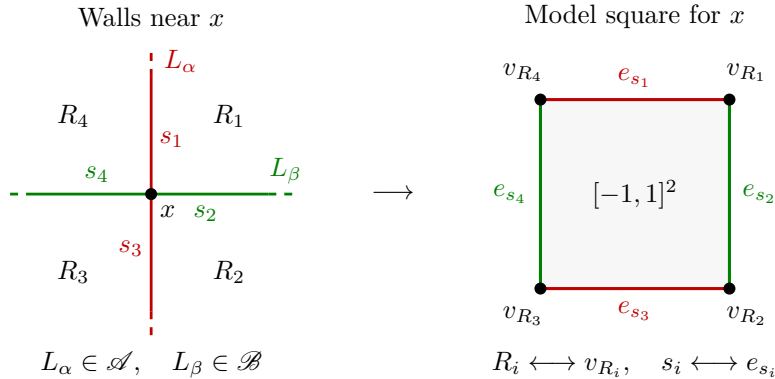
\begin{figure}[!ht]
\centering
\begin{tikzpicture}[x=1cm,y=1cm,font=\small,
awall/.style={red!75!black,line width=1.05pt},
bwall/.style={green!50!black,line width=1.05pt},
dot/.style={circle,fill=black,inner sep=1.6pt}]


\node at (0,2.35) {Walls near $x$};
\draw[awall] (0,-1.55)--(0,1.55);
\draw[awall,dashed] (0,1.55)--(0,1.9);
\draw[awall,dashed] (0,-1.55)--(0,-1.9);
\draw[bwall] (-1.55,0)--(1.55,0);
\draw[bwall,dashed] (1.55,0)--(1.9,0);
\draw[bwall,dashed] (-1.55,0)--(-1.9,0);
\node[dot] at (0,0) {};
\node at (0.22,-0.23) {$x$};
\node[red!75!black] at (0.27,0.72) {$s_1$};
\node[green!50!black] at (0.72,-0.26) {$s_2$};
\node[red!75!black] at (-0.27,-0.72) {$s_3$};
\node[green!50!black] at (-0.72,0.26) {$s_4$};
\node[red!75!black] at (0.39,1.76) {$L_\alpha$};
\node[green!50!black] at (1.78,0.31) {$L_\beta$};
\node at (1.03,1.03) {$R_1$};
\node at (1.03,-1.03) {$R_2$};
\node at (-1.03,-1.03) {$R_3$};
\node at (-1.03,1.03) {$R_4$};
\node at (0,-2.25) {$L_\alpha\in\mathscr A,\quad L_\beta\in\mathscr B$};
\node at (3.2,0) {$\longrightarrow$};


\begin{scope}[shift={(6.4,0)}]
\node at (0,2.35) {Model square for $x$};
\fill[black!3] (-1.25,-1.25) rectangle (1.25,1.25);
\draw[awall,line width=1.15pt] (-1.25,1.25)--(1.25,1.25);
\draw[awall,line width=1.15pt] (-1.25,-1.25)--(1.25,-1.25);
\draw[bwall,line width=1.15pt] (1.25,-1.25)--(1.25,1.25);
\draw[bwall,line width=1.15pt] (-1.25,-1.25)--(-1.25,1.25);
\foreach \p in {(-1.25,1.25),(1.25,1.25),(1.25,-1.25),(-1.25,-1.25)}
\node[dot] at \p {};
\node at (0,0) {$[-1,1]^2$};
\node[red!75!black] at (0,1.54) {$e_{s_1}$};
\node[green!50!black] at (1.65,0) {$e_{s_2}$};
\node[red!75!black] at (0,-1.54) {$e_{s_3}$};
\node[green!50!black] at (-1.65,0) {$e_{s_4}$};
\node at (1.48,1.61) {$v_{R_1}$};
\node at (1.48,-1.61) {$v_{R_2}$};
\node at (-1.48,-1.61) {$v_{R_3}$};
\node at (-1.48,1.61) {$v_{R_4}$};
\node at (0,-2.25) {$R_i\longleftrightarrow v_{R_i},\quad s_i\longleftrightarrow e_{s_i}$};
\end{scope}
\end{tikzpicture}
\caption{A crossing $x$ and its dual square. On the left, the four wall segments
$s_1,s_2,s_3,s_4$ at $x$ in cyclic order, and the four adjacent complementary components
$R_1,\ldots,R_4$ (with $R_0=R_4$). On the right, the model square attached at $x$: its
$i$th side is glued to $e_{s_i}$, whose endpoints are $v_{R_{i-1}}$ and $v_{R_i}$.}
\label{fig:dual-square-construction}
\end{figure}


The graph $\alpha\cup\beta$ has finitely many vertices and edges, and each complementary component
is a disk bounded by finitely many curve segments. Every component of the lifted complement is a
lift of one of these disks, so only finitely many wall segments meet it. A wall segment has two
crossing endpoints and belongs to two dual squares (\cref{fig:dual-square-construction}), so
$C(\alpha,\beta)$ is a locally finite regular square complex. The deck action preserves regions,
wall segments, crossings, and face positions, so it is cellular. The construction is the two
dimensional case of Sageev's cubulation of a wallspace \cite{Sageev1995,HruskaWise2014}, and for
surfaces its image in the product of the two dual trees is the Guirardel core \cite{Guirardel2005}.


\begin{figure}[htbp]
\centering
\resizebox{\linewidth}{!}{
\begin{tikzpicture}[
x=1cm,y=1cm,
font=\small,
alphaedge/.style={
draw=red!75!black,line width=1.3pt},
betaedge/.style={
draw=green!50!black,line width=1.3pt},
vertex/.style={
circle,fill=black,inner sep=1.5pt},
maparrow/.style={
-{Stealth[length=2.6mm]},line width=.8pt},
paneltitle/.style={
font=\normalsize,align=center}]
\node[paneltitle] at (2.1,5.0) {Local wall picture};
\fill[black!2] (.35,.6) rectangle (3.85,4.05);

\draw[alphaedge] (2.1,.6)--(2.1,4.05);
\draw[betaedge] (.35,2.3)--(3.85,2.3);
\node[vertex] at (2.1,2.3) {};
\node[anchor=north east,fill=white,inner sep=1pt]
at (2.02,2.22) {$x$};

\node[red!75!black,anchor=south west]
at (2.15,3.7) {$L_\alpha$};
\node[green!50!black,anchor=south east]
at (3.83,2.36) {$L_\beta$};

\node at (1.16,3.15) {$R_{-+}$};
\node at (3.04,3.15) {$R_{++}$};
\node at (1.16,1.45) {$R_{--}$};
\node at (3.04,1.45) {$R_{+-}$};
\node at (2.1,.05) {$x=L_\alpha\cap L_\beta$};

\draw[maparrow] (4.25,2.3)--(5.65,2.3)
node[midway,above=5pt] {dual};

\node[paneltitle] at (7.9,5.0)
{Dual square in $C(\alpha,\beta)$};

\fill[black!2] (6.35,.75) rectangle (9.45,3.85);
\draw[alphaedge] (6.35,.75)--(9.45,.75);
\draw[alphaedge] (6.35,3.85)--(9.45,3.85);
\draw[betaedge] (6.35,.75)--(6.35,3.85);
\draw[betaedge] (9.45,.75)--(9.45,3.85);

\node[vertex,
label={[label distance=3pt]below:$v_{R_{--}}$}]
at (6.35,.75) {};
\node[vertex,
label={[label distance=3pt]below:$v_{R_{+-}}$}]
at (9.45,.75) {};
\node[vertex,
label={[label distance=3pt]above:$v_{R_{-+}}$}]
at (6.35,3.85) {};
\node[vertex,
label={[label distance=3pt]above:$v_{R_{++}}$}]
at (9.45,3.85) {};

\node[font=\large] at (7.9,2.3) {$\cl{q_x}$};

\draw[maparrow] (10.15,2.3)--(12.2,2.3)
node[midway,above=5pt,font=\large] {$j$};

\node[paneltitle] at (15.65,5.0)
{Product cell in $T_\alpha\times T_\beta$};

\fill[black!2] (13.2,.75) rectangle (18.1,3.85);
\draw[alphaedge] (13.2,.75)--(18.1,.75);
\draw[alphaedge] (13.2,3.85)--(18.1,3.85);
\draw[betaedge] (13.2,.75)--(13.2,3.85);
\draw[betaedge] (18.1,.75)--(18.1,3.85);

\node[vertex,
label={[label distance=4pt]below:
$(U_\alpha^-,U_\beta^-)$}]
at (13.2,.75) {};
\node[vertex,
label={[label distance=4pt]below:
$(U_\alpha^+,U_\beta^-)$}]
at (18.1,.75) {};
\node[vertex,
label={[label distance=4pt]above:
$(U_\alpha^-,U_\beta^+)$}]
at (13.2,3.85) {};
\node[vertex,
label={[label distance=4pt]above:
$(U_\alpha^+,U_\beta^+)$}]
at (18.1,3.85) {};

\node at (15.65,2.3)
{$E_\alpha(L_\alpha)\times E_\beta(L_\beta)$};

\node[red!75!black,fill=white,inner sep=2pt]
at (15.65,.75)
{$E_\alpha(L_\alpha)\times\{U_\beta^-\}$};

\node[green!50!black,rotate=90,fill=white,inner sep=2pt]
at (18.1,2.3)
{$\{U_\alpha^+\}\times E_\beta(L_\beta)$};

\node[font=\normalsize] at (9.4,-.6)
{$j(v_{R_{\varepsilon\delta}})=(U_\alpha^\varepsilon,U_\beta^\delta),
\qquad \varepsilon,\delta\in\{-,+\}$};

\draw[alphaedge] (3.15,-1.35)--(3.8,-1.35);
\node[anchor=west] at (3.95,-1.35)
{$\alpha$-wall / edge dual to an $\alpha$-wall};

\draw[betaedge] (10.85,-1.35)--(11.5,-1.35);
\node[anchor=west] at (11.65,-1.35)
{$\beta$-wall / edge dual to a $\beta$-wall};

\end{tikzpicture}
}
\caption{The map $j$ on the closed dual square $\cl{q_x}$ at $x=L_\alpha\cap L_\beta$,
where $E_\alpha(L_\alpha)$ and $E_\beta(L_\beta)$ are the closed tree edges dual to the
two walls, with endpoints $U_\alpha^\pm$ and $U_\beta^\pm$. The signs in
$R_{\varepsilon\delta}$ retain its $\alpha$- and $\beta$-components.
Horizontal dual edges vary the $T_\alpha$-coordinate, and vertical dual edges vary the
$T_\beta$-coordinate. Red indicates the $\alpha$-family and green the $\beta$-family;
the left panel is only a local topological picture; the walls do not in general meet at
right angles.}
\label{fig:map-j}
\end{figure}
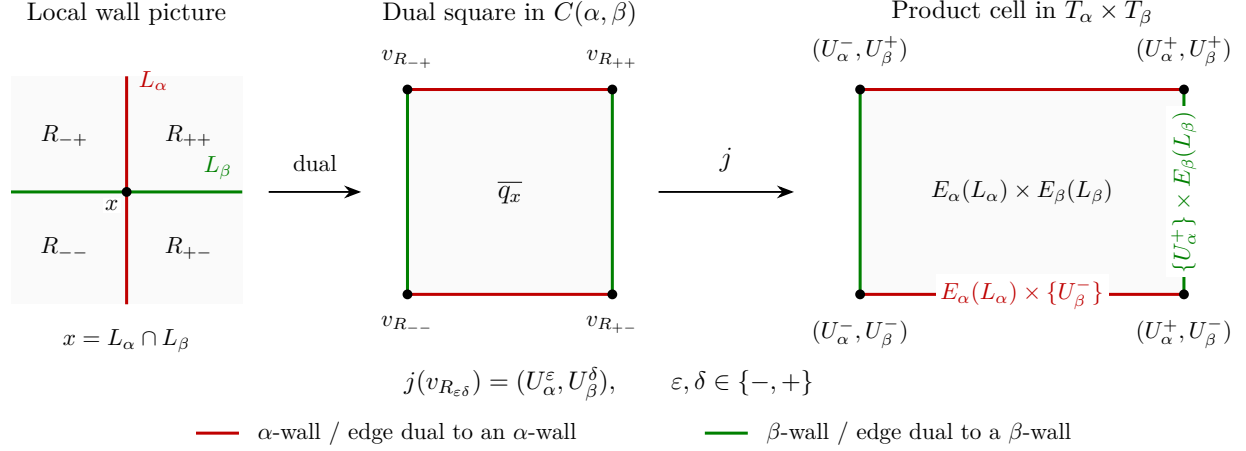


Let $T_\alpha$ and $T_\beta$ be the trees dual to $\mathscr A$ and $\mathscr B$. For a complementary
component $R$, let $U_\alpha(R)$ and $U_\beta(R)$ be the components of
$\widetilde S\setminus\mathscr A$ and $\widetilde S\setminus\mathscr B$ containing $R$. For
$L\in\mathscr A$, let $E_\alpha(L)\subset T_\alpha$ denote the closed edge dual to $L$. Define
$E_\beta(L)\subset T_\beta$ similarly for $L\in\mathscr B$. We first set
\[
j(v_R)=\bigl(U_\alpha(R),U_\beta(R)\bigr).
\]

Suppose that $s\subset L\in\mathscr A$ is an open wall segment, and let $R_-,R_+$ be its two
adjacent regions. The segment $s$ is connected and avoids $\mathscr B$, so it lies in a unique
component $U_\beta(s)$ of $\widetilde S\setminus\mathscr B$. We have
\[
U_\beta(R_-)=U_\beta(R_+)=U_\beta(s),
\]
while $U_\alpha(R_-)$ and $U_\alpha(R_+)$ are the two endpoints of $E_\alpha(L)$. Extend $j$ over
the closed dual edge $\cl{e_s}$ by the unique affine homeomorphism
\[
j|_{\cl{e_s}}:
\cl{e_s}\xrightarrow{\cong}
E_\alpha(L)\times\{U_\beta(s)\}
\]
whose endpoint values are the given images of $v_{R_-}$ and $v_{R_+}$.

Similarly, if $s\subset L\in\mathscr B$, let $U_\alpha(s)$ be the component of
$\widetilde S\setminus\mathscr A$ containing $s$, and define
\[
j|_{\cl{e_s}}:
\cl{e_s}\xrightarrow{\cong}
\{U_\alpha(s)\}\times E_\beta(L)
\]
affinely, with the given endpoint values.

Now let $q_x$ be the open dual square corresponding to a crossing $x=L_\alpha\cap L_\beta$. The four
regions adjacent to $x$ have coordinate pairs equal to the four vertices of
\[
E_\alpha(L_\alpha)\times E_\beta(L_\beta).
\]
Their cyclic order agrees with the boundary order of this product square, and the vertex assignment
extends to a unique cubical isomorphism
\[
j|_{\cl{q_x}}:
\cl{q_x}\xrightarrow{\cong}
E_\alpha(L_\alpha)\times E_\beta(L_\beta).
\]
On each side, this is the affine map already defined on the corresponding dual edge, so the cell
maps agree on their common faces and define a continuous cellular map
\[
j:C(\alpha,\beta)\longrightarrow T_\alpha\times T_\beta.
\]

The action of $\Pi$ preserves both wall families and their complementary components, hence for every
$g\in\Pi$ we have
\[
\begin{aligned}
j(gv_R)
&=j(v_{gR})\\
&=\bigl(U_\alpha(gR),U_\beta(gR)\bigr)\\
&=\bigl(gU_\alpha(R),gU_\beta(R)\bigr)\\
&=g\,j(v_R).
\end{aligned}
\]
The action on the model cells is cubical, so uniqueness of the edge and square maps above yields
\[
j(gz)=g\,j(z)
\qquad
(g\in\Pi,\ z\in C(\alpha,\beta)).
\]
Therefore $j$ is $\Pi$-equivariant, and the square at $L_\alpha\cap L_\beta$ maps to the product of
the two dual tree edges.

Let $\mathscr N$ be a downward directed family of finite index normal subgroups of $\Pi$: for
$N_1,N_2\in\mathscr N$ there is $N_3\in\mathscr N$ with $N_3\leq N_1\cap N_2$. Order $\mathscr N$ by
reverse inclusion. It is \emph{cofinal} if
\[
\forall K\normalf\Pi\quad\exists N\in\mathscr N\quad N\leq K.
\]
We assume henceforth that $\mathscr N$ is cofinal among the finite index normal subgroups of $\Pi$.
For $N'\leq N$, the covering projection $N'\backslash C\to N\backslash C$ induces a morphism
$N'\backslash\Icell(C)\longrightarrow N\backslash\Icell(C)$ on all five component spaces. The
\emph{profinite cubical residual completion with respect to $\mathscr N$} is the component-wise
inverse limit
\[
\wh{\Icell}_{\mathscr N}(C)
=\varprojlim_{N\in\mathscr N}N\backslash\Icell(C).
\]
Thus, for example:
\[
\wh C_{\mathscr N}^{\,2}=\varprojlim_{N\in\mathscr N}N\backslash C^2,
\qquad
\wh I_{21,\mathscr N}(C)=\varprojlim_{N\in\mathscr N}N\backslash I_{21}(C),
\]
and the four structure maps are the inverse limits of the corresponding quotient maps. Let
$\mathscr K$ be the directed family of all finite index normal subgroups of $\Pi$. Restriction of
coordinates defines an isomorphism
\[
\varprojlim_{K\in\mathscr K}K\backslash\Icell(C)
\longrightarrow
\varprojlim_{N\in\mathscr N}N\backslash\Icell(C).
\]
Now for the inverse map, take $(x_N)_{N\in\mathscr N}$ a compatible point in the completion. Given
$K\in\mathscr K$, choose $N\in\mathscr N$ with $N\leq K$ and take the image of $x_N$ in
$K\backslash\Icell(C)$. This coordinate does not depend on the choice of $N$, because the system is
directed \cite{RibesZalesskii2010}.

In fact, we will show that the action of $\Pi$ on each cell and occurrence set is free and has
finitely many orbits, see \cref{prop:original-core}. Since $\Pi$ is residually finite and
$\mathscr N$ is cofinal, the coordinate maps give injections with dense image
\[
C^d\hookrightarrow\wh C_{\mathscr N}^{\,d}\quad(0\leq d\leq2),
\qquad
I_{10}(C)\hookrightarrow\wh I_{10,\mathscr N}(C),
\qquad
I_{21}(C)\hookrightarrow\wh I_{21,\mathscr N}(C).
\]
We call a point in one of these dense images, that is, a point coming from the discrete complex, an
\emph{original cell} or an \emph{original occurrence}.

Choose base lifts $\widetilde\alpha\in\mathscr A$ and $\widetilde\beta\in\mathscr B$, with
stabilizers $A,B\leq\Pi$, and set
\[
\Omega=\{g\in\Pi:\widetilde\alpha\cap g\widetilde\beta\neq\varnothing\}.
\]
This set is $A$--$B$ bi-invariant. If $e=x\cl A\in E(\wh T_\alpha)$ and
$f=y\cl B\in E(\wh T_\beta)$, we define
\[
e\cross f
\quad\Longleftrightarrow\quad
x^{-1}y\in\cl\Omega\subseteq\wh\Pi.
\]
Then the condition does not depend on the coset representatives, because replacing $x$ by $xa$ and
$y$ by $yb$, with $a\in\cl A$ and $b\in\cl B$, replaces $x^{-1}y$ by $a^{-1}x^{-1}yb$, while
$\cl\Omega$ is $\cl A$--$\cl B$ bi-invariant. We call $\cross$ the \emph{crossing relation}. By
\cref{lem:finite-cover-crossing}, two edges cross exactly when the corresponding elevations meet in
every regular cover from a cofinal family.

We next define the profinite cubical residual complex of the product of the two profinite trees.
For $\xi\in\{\alpha,\beta\}$, let $V_\xi$ and $E_\xi$ be the vertex and edge spaces of $\wh T_\xi$,
and write $\partial_{\xi,0},\partial_{\xi,1}:E_\xi\to V_\xi$ for the two endpoint maps. Set
\[
J_\xi=E_\xi\times\{0,1\},
\qquad
p_\xi(e,r)=e,
\qquad
q_\xi(e,r)=\partial_{\xi,r}(e).
\]
As $\{0,1\}$ is discrete, $J_\xi$ is profinite and both maps are continuous. The remaining component
spaces are
\[
P^0=V_\alpha\times V_\beta,
\qquad
P^1=(E_\alpha\times V_\beta)\sqcup(V_\alpha\times E_\beta),
\qquad
P^2=E_\alpha\times E_\beta,
\]
\[
I_{10}(P)=(J_\alpha\times V_\beta)\sqcup(V_\alpha\times J_\beta),
\qquad
I_{21}(P)=(J_\alpha\times E_\beta)\sqcup(E_\alpha\times J_\beta).
\]
On $J_\alpha\times V_\beta$, we set
\[
p_{10}(\iota,v)=(p_\alpha(\iota),v),
\qquad q_{10}(\iota,v)=(q_\alpha(\iota),v).
\]
On $V_\alpha\times J_\beta$, we set
\[
p_{10}(v,\iota)=(v,p_\beta(\iota)),
\qquad q_{10}(v,\iota)=(v,q_\beta(\iota)).
\]
On $J_\alpha\times E_\beta$, we set
\[
p_{21}(\iota,e)=(p_\alpha(\iota),e),
\qquad q_{21}(\iota,e)=(q_\alpha(\iota),e).
\]
On $E_\alpha\times J_\beta$, we set
\[
p_{21}(e,\iota)=(e,p_\beta(\iota)),
\qquad q_{21}(e,\iota)=(e,q_\beta(\iota)).
\]
All component spaces are profinite, and all four maps are continuous. We write
$\Icell(\wh T_\alpha\times\wh T_\beta)$ for this five component profinite cubical residual
complex.

The \emph{completed coordinate map} is the unique continuous extension
\[
\wh j:\wh{\Icell}_{\mathscr N}(C)
\longrightarrow\Icell(\wh T_\alpha\times\wh T_\beta)
\]
of the five component maps induced by $j$. Uniqueness follows from density of the original component
spaces and the Hausdorff property of the codomain spaces. \Cref{prop:completion-embedding} proves
that these component maps form an embedding and that the completed square space is
\[
\{e\times f:e\in E(\wh T_\alpha),\ f\in E(\wh T_\beta),\ e\cross f\}.
\]
\end{definition}

\begin{example}
Let $x=\widetilde\alpha\cap g\widetilde\beta$ be a crossing, let $q_x$ be its dual square, and let
$e_{\widetilde\alpha}$ and $e_{g\widetilde\beta}$ be the corresponding tree edges. We fix a nested
cofinal sequence $\Pi\geq N_1\geq N_2\geq\cdots$ of finite index normal subgroups and form the
completion with respect to this sequence. Such a sequence exists because a finitely generated group
has only finitely many subgroups of each finite index: enumerate the finite index normal subgroups
as $L_1,L_2,\ldots$ and set $N_m=L_1\cap\cdots\cap L_m$. The original square determined by $q_x$ is
$\wh q_x=\bigl(N_mq_x\bigr)_{m\geq1}\in\wh C^{\,2}$. Its four side positions give four elements
\[
\bigl(N_m(q_x,j,\epsilon)\bigr)_{m\geq1}
\in\wh I_{21}(C),
\qquad j\in\{1,2\},\ \epsilon\in\{-1,+1\},
\]
and the inverse limit map $q_{21}$ sends them to the four boundary edge elements of $\wh q_x$.
Moreover,
\[
e_{\widetilde\alpha}\cross e_{g\widetilde\beta}
\quad\Longleftrightarrow\quad g\in\Omega,
\qquad
\wh j(\wh q_x)=e_{\widetilde\alpha}\times e_{g\widetilde\beta}.
\]

Let $\gamma\in\wh\Pi$. For each $m$, choose $g_m\in\Pi$ representing the image of $\gamma$ in
$\Pi/N_m$. Then compatibility of the quotient coordinates implies that
$\gamma\wh q_x=\bigl(N_mg_mq_x\bigr)_{m\geq1}$ is a well defined element of $\wh C^{\,2}$,
independent of the representatives $g_m$, and by equivariance
$\wh j(\gamma\wh q_x)=\gamma\bigl(e_{\widetilde\alpha}\times e_{g\widetilde\beta}\bigr)$. If
$\gamma\notin\Pi$, then $\gamma\wh q_x$ is not original: the completed square orbit is a free
$\wh\Pi$-orbit, so an equality $\gamma\wh q_x=\delta\wh q_x$ with $\delta\in\Pi$ would imply
$\gamma=\delta$. Hence left multiplication by an element of $\wh\Pi\setminus\Pi$ preserves the
completed cubical residual complex but moves every square in this orbit outside the original
subset. This is why we need normalization in \cref{thm:discrete-star}.
\end{example}

\begin{definition}\label{def:characteristic-cover}
Given $G$ finitely generated, a connected finite sheeted characteristic cover of a connected
aspherical space with fundamental group $G$ is the cover corresponding to a finite index
characteristic subgroup of $G$. For $m\geq1$, we set $K_m(G)=\bigcap_{[G:L]\leq m}L$. If
$N\normalf G$ and $[G:N]=d$, then $K_d(G)\leq N$. We therefore obtain a cofinal family
$(K_m(G))_{m\geq1}$ among the finite index normal subgroups of $G$, which we call the \emph{standard
family of characteristic subgroups}.
\end{definition}

\begin{definition}\label{def:divisible-fillings}
For $n\geq1$, we set $\ell_n=\lcm(1,2,\ldots,n)$. Then
\[
\ell_n\mid\ell_{n+1},
\qquad
n\geq d\Longrightarrow d\mid\ell_n,
\qquad
\ell_n\longrightarrow\infty.
\]
Indeed, $\ell_{n+1}=\lcm(\ell_n,n+1)$, the integer $d$ occurs among $1,\ldots,n$ whenever $n\geq d$,
and $\ell_n\geq n$. For a marked element $a$, the quotient obtained by imposing $a^{\ell_n}=1$ is
called the $n$th \emph{divisible filling}. In the simultaneous cusped construction, the orders
$r_i\ell_n$, with fixed positive integers $r_i$, are called the corresponding \emph{divisible
filling orders}.
\end{definition}

\begin{definition}\label{def:period-prime-group-primitivity}
Let $(\phi^t)_{t\in\R}$ be a flow without stationary points on a manifold $M$. A point $x\in M$ is
periodic if $\phi^T(x)=x$ for some $T>0$. Its \emph{least period} is
\[
T_x=\min\{T>0:\phi^T(x)=x\}.
\]
For a homeomorphism $f:S\to S$, a point $x\in S$ has \emph{least return period} $m\geq1$ if
$f^m(x)=x$ and $f^j(x)\neq x$ for $1\leq j<m$. A positive periodic parametrization of the flow orbit
through $x$ is the map
\[
\gamma_x:\R/T_x\Z\longrightarrow M,
\qquad \gamma_x([s]_{T_x})=\phi^s(x).
\]
A positively oriented closed suspension trajectory is \emph{prime} if it is not the $r$-fold
traversal of another positively oriented closed suspension trajectory for any integer $r\geq2$. For
a positive periodic parametrization $\gamma:\R/T\Z\to M$, its $r$-fold traversal is the map
\[
\gamma^{[r]}:\R/(rT)\Z\longrightarrow M,
\qquad
\gamma^{[r]}([s]_{rT})=\gamma([s]_T).
\]
Let $K$ be a group and let $b\in K$. The element $b$ is a \emph{proper power in $K$} if $b=g^r$ for
some $g\in K$ and some integer $r\geq2$. Thus $b$ is not a proper power in $K$ if no such pair
$(g,r)$ exists.

The group matters here, since roots may appear after enlarging the group: if
$G=\langle a\rangle\cong\Z$ and $H=\langle a^2\rangle\leq G$, then $a^2$ is not a proper power in
$H$ but is a square in $G$. Later the same element will lie in a fiber subgroup, in a finite index
subgroup and in a mapping torus group, and we always say in which group a statement about proper
powers is meant. Dynamical primeness is a different condition; for pseudo-Anosov suspension orbits,
the implication we need is \cref{lem:orbit-roots}.
\end{definition}

\begin{definition}\label{def:fibered-class}
Let $M$ be a connected manifold. An integral class $\chi\in H^1(M,\Z)=\Hom(\pi_1M,\Z)$ is
\emph{primitive} if $\chi(\pi_1M)=\Z$. A primitive integral class $\chi$ is \emph{fibered} if it is
induced by a fibration $M\to S^1$ with connected fiber. Its fiber group is $\Pi_\chi=\ker\chi$. See
\cite[Section~3]{Thurston1986} for fibered classes in dimension three. The class $\chi$ is
primitive, hence surjective, and there is an exact sequence
\[
1\longrightarrow \Pi_\chi\longrightarrow \pi_1M
\overset{\chi}{\longrightarrow}\Z\longrightarrow1.
\]
We choose an element $t\in\pi_1M$ such that $\chi(t)=1\in\Z$. The kernel $\Pi_\chi$ is normal, so
conjugation by $t$ restricts to an automorphism
\[
a_\chi:\Pi_\chi\longrightarrow\Pi_\chi,
\qquad a_\chi(x)=txt^{-1},
\]
and the choice of $t$ splits the sequence:
\[
\pi_1M\cong\Pi_\chi\rtimes_{a_\chi}\langle t\rangle,
\qquad \chi(t)=1.
\]
If $t'$ is another element with $\chi(t')=1$, then $h=t't^{-1}\in\Pi_\chi$ and
\[
\conj{t'}|_{\Pi_\chi}=\Inn(h)\circ\conj{t}|_{\Pi_\chi}.
\]
Consequently the outer class $[a_\chi]\in\Out(\Pi_\chi)$ is independent of the chosen element $t$
with $\chi(t)=1$ and is the algebraic monodromy of the oriented fibered class $\chi$. Choosing $t$
with $\chi(t)=-1$ instead would replace this outer class by its inverse. The condition $\chi(t)=1$
fixes the orientation determined by $\chi$.
\end{definition}

\begin{definition}\label{def:full-cellular-occurrence}
Let $X$ be a finite dimensional regular CW complex of dimension $d$ with a characteristic
homeomorphism $\chi_\sigma$ for each open $k$-cell $\sigma$. A \emph{codimension one face position}
of $\sigma$ is an open $(k-1)$-cell $\iota$ of the pulled back boundary decomposition. Its image
\[
\tau=\chi_\sigma(\iota)\in X^{k-1}
\]
is the corresponding open codimension one face of $\sigma$. An \emph{occurrence of $\tau$ as a face
of $\sigma$} is a face position $\iota$ with $\chi_\sigma(\iota)=\tau$, or equivalently the
restriction $\chi_\sigma|_\iota:\iota\xrightarrow{\cong}\tau$. Thus an occurrence contains both the
incident face and its position in the model boundary.

For $0\leq k\leq d$, write $X^k$ for the set of open $k$-cells. If $1\leq k\leq d$, set
\[
I_{k,k-1}(X)
=\left\{(\sigma,\tau,\iota):
\begin{array}{l}
\sigma\in X^k,\ \tau\in X^{k-1},                                          \\
\iota\subset S^{k-1}\text{ is a codimension one face position of }\sigma, \\
\chi_\sigma(\iota)=\tau
\end{array}
\right\}.
\]
We keep the two projections $p_{k,k-1}(\sigma,\tau,\iota)=\sigma$ and
$q_{k,k-1}(\sigma,\tau,\iota)=\tau$. The \emph{full cellular residual complex} $\Icell_\bullet(X)$
consists of all cell spaces $X^k$, all codimension one occurrence spaces $I_{k,k-1}(X)$, and all
maps $p_{k,k-1},q_{k,k-1}$. If $X$ is a square complex, this structure contains the complex of
\cref{def:cubical-occurrence}, with endpoint and side occurrence spaces $I_{1,0}$ and $I_{2,1}$.

A cellular automorphism $g$ acts on face positions by
\[
g(\sigma,\tau,\iota)
=\bigl(g\sigma,g\tau,\iota'\bigr),
\qquad
\iota'=(\chi_{g\sigma})^{-1}
\bigl(g\chi_\sigma(\iota)\bigr).
\]
The set $\iota'$ is the unique face position in the chosen model of $g\sigma$ that maps to $g\tau$.
If a subgroup $N$ acts cellularly on $X$, define the quotient cellular residual complex
component-wise by
\[
\Icell_\bullet(N\backslash X)
:=N\backslash\Icell_\bullet(X).
\]
Two face position orbits remain distinct unless an element of $N$ identifies them, even when their
image cells in $N\backslash X$ coincide.

Suppose a finitely generated residually finite group $\Gamma$ acts freely, cellularly, and
cocompactly on a connected, locally finite, finite dimensional regular CW complex $X$. For a cofinal
family $\mathscr N$ of finite index normal subgroups of $\Gamma$, we define
\[
\wh{\Icell}_\bullet(X)
=\varprojlim_{N\in\mathscr N}\Icell_\bullet(N\backslash X),
\]
where the inverse limit is taken simultaneously on every cell space, every occurrence space, and
every structure map. An isomorphism of completed full cellular residual complexes consists of
homeomorphisms on these profinite spaces that commute with all structure maps and admit an inverse
with the same property.
\end{definition}


\section{Main theorems}\label{sec:main-results}
We first state the cut tree theorem, then the realization theorems for cubical residual complexes,
and finally the applications: profinite rigidity of lattices in $\PSL_2(\C)$, and realization
statements in arbitrary dimension relative to the family of characteristic subgroups.

We begin by fixing notation. For $i\in\{1,2\}$, let $S_i$ be a closed orientable surface of genus
$g_i\geq2$, set $\Pi_i=\pi_1(S_i)$, and let $\alpha_i$ be a nonseparating simple closed curve on
$S_i$. Let $H_i$ be the fundamental group of the surface obtained by cutting $S_i$ along $\alpha_i$.
Choose generators $a_{i,-},a_{i,+}\in H_i$ for the two boundary subgroups and a stable letter $t_i$
such that
\[
\Pi_i=
\left\langle H_i,t_i\ \middle|\
t_i a_{i,-}t_i^{-1}=a_{i,+}
\right\rangle,
\qquad
a_i=a_{i,-},\quad A_i=\langle a_i\rangle.
\]
Suppose that
\[
\theta:\wh\Pi_1\xrightarrow{\cong}\wh\Pi_2,
\qquad
\theta(\cl A_1)=x\cl A_2x^{-1}
\quad\text{for some }x\in\wh\Pi_2.
\]

For $n\geq3$, set
\[
K_{\alpha_i,n}=\Ncl{a_i^n}{\Pi_i},
\qquad
\Gamma_{\alpha_i,n}=\Pi_i/K_{\alpha_i,n},
\]
and let $\rho_{\alpha_i,n}:\wh\Pi_i\twoheadrightarrow \wh\Gamma_{\alpha_i,n}$ be the continuous
epimorphism induced by the quotient. Also set
\[
V_{\alpha_i,n}
=
H_i/\Ncl{a_{i,-}^n,a_{i,+}^n}{H_i}.
\]
Then $V_{\alpha_i,n}$ is the cocompact orientable Fuchsian orbifold group of genus $g_i-1$ with two
cone points of order $n$, while $\Gamma_{\alpha_i,n}$ has the HNN splitting
\[
\Gamma_{\alpha_i,n}
\cong
\left\langle V_{\alpha_i,n},t_i\ \middle|\
t_i\cl a_{i,-}t_i^{-1}=\cl a_{i,+}
\right\rangle,
\]
where $\cl a_{i,-}$ and $\cl a_{i,+}$ have order $n$. This splitting is efficient, so the inclusion
$V_{\alpha_i,n}\hookrightarrow\Gamma_{\alpha_i,n}$ induces an embedding
\[
\wh V_{\alpha_i,n}\hookrightarrow\wh\Gamma_{\alpha_i,n}.
\]
We set
\[
W_{\alpha_i,n}
:=
\im\!\left(
\wh V_{\alpha_i,n}
\hookrightarrow
\wh\Gamma_{\alpha_i,n}
\right)
=
\cl{V_{\alpha_i,n}}^{\,\wh\Gamma_{\alpha_i,n}}
\]
and $P_{\alpha_i,n} := \rho_{\alpha_i,n}^{-1}(W_{\alpha_i,n})$; here $W_{\alpha_i,n}$ is the closed
stabilizer of the vertex represented by $V_{\alpha_i,n}$ in the standard profinite Bass--Serre tree.
The HNN edge group of $\Gamma_{\alpha_i,n}$ is cyclic of order $n$, with images
\[
C_{i,n}^{-}
=
\langle\cl a_{i,-}\rangle,
\qquad
C_{i,n}^{+}
=
\langle\cl a_{i,+}\rangle.
\]
\begin{theorem}\label{thm:finite-cone-main}
With the notation above, let $A_i=\langle a_i\rangle$ be the maximal cyclic subgroup of $\Pi_i$
represented by $\alpha_i$, and let
\[
H_i=\pi_1\bigl(S_i\setminus\Int N(\alpha_i)\bigr)
\]
be the vertex group of the cut surface, where $N(\alpha_i)$ is a closed annular neighborhood.
Suppose
\[
\theta:\wh\Pi_1\xrightarrow{\cong}\wh\Pi_2,
\qquad
\theta(\cl A_1)=x\cl A_2x^{-1}
\quad\text{for some }x\in\wh\Pi_2.
\]
For $n\geq3$, let $P_{\alpha_i,n}\leq\wh\Pi_i$ be the inverse image of the completed Fuchsian vertex
group under
\[
\wh\Pi_i\longrightarrow
\wh{\Pi_i/\Ncl{a_i^n}{\Pi_i}}.
\]
Then
\[
\cl H_i=\bigcap_{n\geq3}P_{\alpha_i,\ell_n}.
\]
There is $g\in\wh\Pi_2$ such that
\begin{equation}\label{eq:cut-surface-conjugacy-main}
\theta(\cl H_1)=g\cl H_2g^{-1}.
\end{equation}
The map on vertex cosets
\[
q\cl H_1\longmapsto\theta(q)g\cl H_2
\]
extends uniquely to a $\theta$-equivariant isomorphism of the full profinite cut trees
\[
\wh T_{\alpha_1}\xrightarrow{\cong}\wh T_{\alpha_2}.
\]
It preserves the endpoint occurrence maps after a single global reversal of the two endpoint labels
in $\wh T_{\alpha_2}$ when the chosen edge orientations disagree. The induced tree isomorphism is
independent of the element $g$ satisfying \eqref{eq:cut-surface-conjugacy-main}, and these tree
isomorphisms are compatible with composition of profinite isomorphisms.
\end{theorem}
Now choose on each $S_i$ a nonseparating filling pair $(\alpha_i,\beta_i)$, and write
$C_i=C(\alpha_i,\beta_i)$ for the dual square complex of \cref{def:core}. After lifting the geodesic
representatives to $\widetilde S_i\cong\HH^2$, its vertices are the complementary regions, its edges
are dual to wall segments between consecutive crossings, and its squares are dual to crossings of an
$\alpha_i$ lift with a $\beta_i$ lift. We call $C_i$ the \emph{original filling core} only to
distinguish this discrete $\Pi_i$ complex from its cell-wise profinite completion $\wh C_i$.

Let $T_{\alpha_i}$ and $T_{\beta_i}$ be the original cut trees and let $\wh T_{\alpha_i}$ and
$\wh T_{\beta_i}$ be their profinite completions. The coordinate map
\[
j_i:C_i\hookrightarrow
T_{\alpha_i}\times T_{\beta_i},
\]
which sends a complementary region to the pair of components of the complements of the two lift
families that contain it, extends uniquely to a continuous $\wh\Pi_i$-equivariant cellular embedding
\[
\wh j_i:\wh C_i\hookrightarrow
\wh T_{\alpha_i}\times\wh T_{\beta_i},
\]
see Sections~\ref{sec:cut-tree} and~\ref{sec:core}. Its square space is the closed set of products
$e_\alpha\times e_\beta$ such that $e_\alpha\cross e_\beta$, where $\cross$ is the profinite
crossing relation of \cref{def:core}. With this notation, our second main theorem reads as follows.

\begin{theorem}
\label{thm:core-main}
Let $S_i$, $\Pi_i$ and $(\alpha_i,\beta_i)$ be as above, let $A_i,B_i\leq\Pi_i$ be the maximal
cyclic subgroups represented by $\alpha_i$ and $\beta_i$, and put
\[
C_i=C(\alpha_i,\beta_i).
\]
Suppose that
\[
\theta:\wh\Pi_1\xrightarrow{\cong}\wh\Pi_2
\]
is a continuous isomorphism satisfying
\[
\theta(\cl A_1)=x\cl A_2x^{-1},
\qquad
\theta(\cl B_1)=y\cl B_2y^{-1}
\]
for some $x,y\in\wh\Pi_2$.

Let
\[
F_\alpha:\wh T_{\alpha_1}
\xrightarrow{\cong}\wh T_{\alpha_2},
\qquad
F_\beta:\wh T_{\beta_1}
\xrightarrow{\cong}\wh T_{\beta_2}
\]
be the $\theta$-equivariant cut tree isomorphisms of \cref{thm:finite-cone-main}. Then, for every
$e\in E(\wh T_{\alpha_1})$ and $f\in E(\wh T_{\beta_1})$,
\[
e\cross f
\quad\Longleftrightarrow\quad
F_\alpha(e)\cross F_\beta(f).
\]

Under the completed coordinate embeddings $\wh j_i$, the product $F_\alpha\times F_\beta$ and its
induced maps on occurrences restrict to an isomorphism
\[
F:\wh{\Icell}(C_1)
\xrightarrow{\cong}\wh{\Icell}(C_2).
\]

\end{theorem}

For arbitrary free cocompact actions on regular square complexes, the corresponding realization
statement is \cref{thm:discrete-star}. Our incidence structure retains face positions and their
multiplicities, but not a parametrized attaching map up to isotopy.

\begin{theorem}
\label{thm:all-dimensional-occurrence}
For $i=1,2$, let a finitely generated residually finite group $\Gamma_i$ act freely, cellularly, and
cocompactly on a connected, locally finite, finite dimensional regular CW complex $X_i$. Suppose
\[
\eta:\wh\Gamma_1\xrightarrow{\cong}\wh\Gamma_2.
\]
We form the completed full cellular residual complexes $\wh{\Icell}_\bullet(X_i)$ of
\cref{def:full-cellular-occurrence} over the standard families of characteristic subgroups
$(K_m(\Gamma_i))_{m\geq1}$. By \cref{lem:char-functor}, $\eta$ identifies the corresponding inverse
systems of group quotients. Suppose further that
\[
F:\wh{\Icell}_\bullet(X_1)
\xrightarrow{\cong}
\wh{\Icell}_\bullet(X_2)
\]
is an $\eta$-equivariant isomorphism of completed full cellular residual complexes. Then there are
an element $u\in\wh\Gamma_2$ and an isomorphism $h:\Gamma_1\to\Gamma_2$ such that
\[
\eta=\Inn(u)\circ\wh h,
\qquad
F(\Icell(X_1))=u\Icell(X_2)
\]
on every original cell and occurrence space. Equivalently, after replacing $(\eta,F)$ by
$(\Inn(u^{-1})\circ\eta,u^{-1}F)$, the normalized map sends the full original cellular residual
complex of $X_1$ onto that of $X_2$.
\end{theorem}

For a directed multigraph $X$, we use
\[
\Dgraph(X)=\bigl(V(X),A(X),s_X,t_X\bigr)
\]
for the vertex set, the directed edge set, and the two endpoint maps, see
\cref{sec:graph-occurrence}. If both orientations of a geometric edge occur, we also keep track of
reversal. We form quotients on every orbit set and every structure map, so loops and parallel edges
remain distinct. See \cref{def:graph-diagram,def:graph-distinguished,def:graph-completion} for the
full definitions.

\begin{theorem}\label{thm:graph-realization}
For $i=1,2$, let a finitely generated residually finite group $G_i$ act cocompactly on a connected
locally finite directed multigraph $X_i$ with finite vertex stabilizers. Let
$\mathcal S_i=(Z_{i,\lambda})_{\lambda\in\Lambda}$ be finite labelled families of invariant
subgraphs, possibly empty, and let $\mathfrak P$ be a subgroup of the symmetric group
$\mathfrak S_\Lambda$ on $\Lambda$, the group of permitted label permutations. Fix an isomorphism
\[
\Phi:\wh G_1\xrightarrow{\cong}\wh G_2,
\]
and let $\Phi_m:G_1/K_m(G_1)\to G_2/K_m(G_2)$ be the induced isomorphism.

Assume either of the following two conditions.
\begin{enumerate}[label=\textup{(\alph*)}]
\item There is an increasing unbounded sequence $(m_j)$ and, for every $j$, a
$\Phi_{m_j}$-equivariant isomorphism
\[
F_j:\Dgraph(K_{m_j}(G_1)\backslash X_1)
\xrightarrow{\cong}
\Dgraph(K_{m_j}(G_2)\backslash X_2)
\]
mapping the quotient of $Z_{1,\lambda}$ onto the quotient of $Z_{2,\sigma_j(\lambda)}$
for some $\sigma_j\in\mathfrak P$. The maps $F_j$ and the permutations $\sigma_j$ are
not required to be compatible with one another.

\item There are $\sigma\in\mathfrak P$ and a $\Phi$-equivariant isomorphism
\[
F:\wh{\Dgraph}_{G_1}(X_1)
\xrightarrow{\cong}
\wh{\Dgraph}_{G_2}(X_2)
\]
such that
\[
F\bigl(\wh{\Dgraph}_{G_1}(Z_{1,\lambda})\bigr)
=\wh{\Dgraph}_{G_2}(Z_{2,\sigma(\lambda)})
\qquad(\lambda\in\Lambda).
\]
\end{enumerate}

Under condition \textup{(a)}, we may select the maps compatibly at the chosen levels, all with the
same permutation $\sigma\in\mathfrak P$. Their inverse limit is an isomorphism $F$ as in
\textup{(b)}. Under either condition there are $u\in\wh G_2$ and an isomorphism
$h:G_1\xrightarrow{\cong}G_2$ such that
\begin{equation}\label{eq:graph-discrete-realization}
\Phi=\Inn(u)\circ\wh h.
\end{equation}
Moreover, $L_{u^{-1}}\circ F$, where $L_{u^{-1}}$ denotes left multiplication by $u^{-1}$ on the
completed vertex and directed edge spaces, restricts to an $h$-equivariant directed graph
isomorphism
\[
\Dgraph(X_1)\xrightarrow{\cong}\Dgraph(X_2)
\]
mapping $Z_{1,\lambda}$ onto $Z_{2,\sigma(\lambda)}$ for every $\lambda\in\Lambda$.
\end{theorem}

\begin{theorem}
\label{thm:graph-pl-realization}
For $i=1,2$, let $C_i$ be connected finite simplicial complexes, let $p_i:X_i\to C_i$ be their
simplicial universal coverings, and identify $G_i=\pi_1(C_i)$ with the deck groups. Assume $G_i$ are
residually finite and let $\Phi:\wh G_1\xrightarrow{\cong}\wh G_2$. Suppose that for an increasing
unbounded sequence $(m_j)$ there are $\Phi_{m_j}$-equivariant isomorphisms
\begin{equation}\label{eq:graph-simplex-finite}
\Icell_\bullet\bigl(K_{m_j}(G_1)\backslash X_1\bigr)
\xrightarrow{\cong}
\Icell_\bullet\bigl(K_{m_j}(G_2)\backslash X_2\bigr).
\end{equation}
Although these isomorphisms need have nothing to do with one another, compatible maps can be
selected at all chosen levels, and if $F$ is the resulting inverse limit isomorphism, then there are
$u\in\wh G_2$ and $h:G_1\xrightarrow{\cong}G_2$ such that $\Phi=\Inn(u)\circ\wh h$ and $u^{-1}F$
maps every original simplex and codimension one occurrence onto its original counterpart. The
resulting map is an $h$-equivariant simplicial isomorphism $X_1\to X_2$ and descends to a
simplicial, hence PL, homeomorphism $C_1\to C_2$ inducing $h$ up to basepoint conjugacy.

If finite families of invariant subcomplexes are distinguished and the finite maps preserve them by
permutations in a fixed finite permutation group, the selected maps all involve the same permutation
and the descended homeomorphism preserves the corresponding subcomplexes.
\end{theorem}

At any one level in \eqref{eq:graph-simplex-finite}, quotienting the equivariant isomorphism by
$G_i/K_{m_j}(G_i)$ produces an isomorphism of the full residual complexes of the base
triangulations $C_1$ and $C_2$. Their face relations determine the simplices and hence give a
simplicial isomorphism $C_1\to C_2$. The cofinal family is used to obtain a realization whose
induced isomorphism on profinite completions agrees with the given $\Phi$ up to inner conjugation.

\subsection*{The two curve criterion}
\begin{corollary}\label{cor:two-curve}
Let $S_i$ be closed orientable surfaces of genus at least two, $\Pi_i=\pi_1(S_i)$, and let
$(\alpha_i,\beta_i)$ be nonseparating filling pairs with associated maximal cyclic subgroups
$A_i,B_i\leq\Pi_i$, for $i=1,2$. If
\[
\theta:\wh\Pi_1\xrightarrow{\cong}\wh\Pi_2,
\quad
\theta(\cl A_1)=x\cl A_2x^{-1},
\quad
\theta(\cl B_1)=y\cl B_2y^{-1},
\]
then $\theta$ is discretely induced up to profinite inner automorphism.
\end{corollary}

\begin{theorem}
\label{thm:no-exotic}
Let $S$ be a closed orientable surface of genus at least two, $\Pi=\pi_1(S)$, and let
$(\alpha,\beta)$ be an ordered nonseparating filling pair with maximal cyclic subgroups
$A,B\leq\Pi$. For every continuous automorphism $\theta\in\Aut(\wh\Pi)$ satisfying
\[
\theta(\cl A)=x\cl A x^{-1},
\qquad
\theta(\cl B)=y\cl B y^{-1}
\qquad\text{for some }x,y\in\wh\Pi,
\]
there exist $h\in\Aut(\Pi)$ and $u\in\wh\Pi$ such that $\theta=\Inn(u)\circ\wh h$. The two conjugacy
conditions are independent of the representative $\theta$ of its outer automorphism class.
\end{theorem}
\begin{remark}
The surface markings are conjugacy classes of closed procyclic subgroups, and no generator is
chosen. With this marking, the given profinite isomorphism is realized after conjugation by an
element of $\wh\Pi_2$. The example in \cref{rem:translation-necessary} shows that this conjugation
cannot be omitted.
\end{remark}

\subsection*{Profinite rigidity for $3$-dimensional hyperbolic lattices}
In Sections~\ref{sec:closed-sign}--\ref{sec:cusps}, we construct the surface markings required for
the lattice theorem.

\begin{theorem}\label{thm:lattice-rigidity}
Let $\Gamma,\Delta<\PSL_2(\C)$ be lattices, with torsion allowed. Every continuous isomorphism
$\Phi:\wh\Gamma\to\wh\Delta$ has the form
\[
\Phi=\Inn(u)\circ\wh f,
\qquad u\in\wh\Delta,\qquad
f:\Gamma\xrightarrow{\cong}\Delta.
\]
The isomorphism $f$ is induced, up to a discrete inner automorphism, by an isometry of the
corresponding finite volume hyperbolic orbifolds, possibly reversing orientation. Consequently
lattices in $\PSL_2(\C)$ with isomorphic profinite completions are isomorphic, and completion
induces
\[
\Out(\Gamma)\xrightarrow{\cong}\Out(\wh\Gamma).
\]
\end{theorem}

In the closed case, we combine the integral sign theorem, the prime orbit correspondence, and the
two fibration configuration of
\cref{thm:closed-integral,thm:orbit-correspondence,thm:two-fibers,cor:configuration}. The one orbit
theorem, \cref{thm:one-orbit}, then invokes the two curve criterion. In the cusped case we use the
unit invariant filling argument of \cref{cor:cusped-rigidity}. Peripheral integrality follows from
\cref{cor:all-integrality}.


\subsection{Applications to nonlattices}
The graph theorem does not require a geometric action on hyperbolic space, and
\cref{cor:graph-kleinian} therefore applies to every finitely generated torsion-free Kleinian group
through a compact core graph that is surjective on fundamental groups. Directed Cayley graphs allow
torsion in \cref{cor:graph-cayley}. In both applications, the finite quotient maps are not assumed
to be compatible with each other at the outset.

The boundary subgroup conclusion of \cref{thm:graph-boundary-realization} allows compressible
boundary and determines the images of the boundary fundamental groups. If we keep the full cellular
residual complex, \cref{cor:graph-compact-core-pl,cor:graph-core-pairs} instead give PL
homeomorphisms of the given compact cores and core pairs. For incompressible boundary with the full
induced profinite topology, \cref{cor:graph-boundary-filling} applies the divisible filling and
filling core theorems to marked boundary filling pairs.
\section{Profinite preliminaries}\label{sec:preliminaries}
We collect some elementary facts about profinite groups in the form in which they will be used. Most
of them can be found in the references cited; a few need minor modifications, and for these we give
proofs.

Throughout, $G$ is finitely generated and $K_m(G)=\bigcap_{[G:L]\leq m}L$ is the standard family of
characteristic subgroups of \cref{def:characteristic-cover}.

\begin{lemma}\label{lem:char-functor}
Let $G,H$ be finitely generated groups. Every continuous isomorphism
$\Phi:\wh G\xrightarrow{\cong}\wh H$ satisfies
\[
\Phi\bigl(\cl{K_m(G)}\bigr)=\cl{K_m(H)}
\qquad(m\geq1).
\]
\end{lemma}
\begin{proof}
Since there are only finitely many subgroups of index at most $m$, and closure commutes with finite
intersections of finite-index subgroups, \cite[Lemma~3.2.1 and Proposition~3.2.2]{RibesZalesskii2010} give
\[
\cl{K_m(G)}
=\bigcap_{\substack{U\leq\wh G\text{ open}\\
[\wh G:U]\leq m}}U.
\]
The same identity holds for $H$. As $\Phi$ preserves open subgroups and their indices, it maps the
first intersection onto the second.
\end{proof}
\begin{lemma}\label{lem:virtual-retract-topology}
Let $G$ be residually finite, let $H\leq G$ have finite index, and let $A\leq H$ admit a retraction
$r:H\to A$. Then $A$ is closed in the profinite topology of $G$, and the topology induced from $G$
on $A$ is the full profinite topology of $A$.
\end{lemma}
\begin{proof}
Closedness of a virtual retract in a residually finite group is
\cite[Lemma~2.2]{Minasyan2021VirtualRetractions}. We prove that the induced topology is the full
profinite topology.

We fix $N\normalf A$. The subgroup $r^{-1}(N)$ has finite index in $H$ and satisfies
$r^{-1}(N)\cap A=N$. Let
\[
D=\normalcore{G}{r^{-1}(N)}=\bigcap_{g\in G}g r^{-1}(N)g^{-1}.
\]
Only finitely many conjugates occur because $r^{-1}(N)$ has finite index, so $D\normalf G$.
Moreover, $D\cap A\leq r^{-1}(N)\cap A=N$, so every basic neighborhood $N$ of $1$ in the full
profinite topology of $A$ contains the intersection with $A$ of an open normal subgroup of $G$. The
reverse inclusion follows because the intersection with $A$ of a finite index subgroup of $G$ has
finite index in $A$. Therefore the induced and full profinite topologies on $A$ coincide.
\end{proof}

\begin{lemma}\label{lem:completion-quotient}
If $N\normal G$, then
\[
\wh{G/N}\cong \wh G/\cl{N},
\]
where $\cl N$ is also the closed normal subgroup of $\wh G$ generated by $N$.
\end{lemma}
\begin{proof}
This is the quotient form of the right exactness of profinite completion \cite{RibesZalesskii2010}.
\end{proof}

\begin{lemma}\label{lem:compact-conjugators}
Let $G$ and $H$ be finitely generated residually finite groups. Let $f:G\to H$ be an isomorphism and
$\Phi:\wh G\to\wh H$ an isomorphism. If the maps induced by $\Phi$ and $\wh f$ on every standard
characteristic quotient differ by an inner automorphism, then $\Phi=\Inn(u)\circ\wh f$ for some
$u\in\wh H$.
\end{lemma}
\begin{proof}
For each $m$, let $q_m:\wh H\to H/K_m(H)$ be the standard characteristic quotient and define
\[
C_m=\{u\in\wh H:q_m\Phi(x)=q_m(u\wh f(x)u^{-1})\text{ for every }x\in\wh G\}.
\]
A finite topological generating set of $\wh G$ suffices to test the equality. So $C_m$ is an
intersection of finitely many inverse images of the diagonal in the finite group $H/K_m(H)$, so it
is closed and open. The conditions make $C_m$ nonempty.

Equality modulo $K_{m+1}(H)$ implies equality modulo $K_m(H)$, because $K_{m+1}(H)\leq K_m(H)$. So
$C_{m+1}\subseteq C_m$. By compactness of $\wh H$ there is an element $u\in\bigcap_m C_m$. For every
$x\in\wh G$, the two elements $\Phi(x)$ and $u\wh f(x)u^{-1}$ have the same image in every standard
characteristic quotient. The family is cofinal and has trivial intersection, so the two elements are
equal. Hence $\Phi=\Inn(u)\circ\wh f$.
\end{proof}

\begin{proposition}
\label{prop:centerless-extension}
For $i=1,2$, let $G_i$ be finitely generated residually finite groups and let
\[
1\longrightarrow H_i\longrightarrow G_i\longrightarrow Q_i\longrightarrow1
\]
be exact sequences with $Q_i$ finite. Assume that the topology induced on $H_i$ by the profinite
topology of $G_i$ is the full profinite topology of $H_i$, so that completion yields an exact
sequence
\[
1\longrightarrow\wh H_i\longrightarrow\wh G_i\longrightarrow Q_i\longrightarrow1.
\]
Assume further that
\[
Z(H_i)=Z(\wh H_i)=1,
\qquad
\Out(H_i)\longrightarrow\Out(\wh H_i)\text{ is injective}.
\]
Let $\Phi:\wh G_1\to\wh G_2$ be an isomorphism such that $\Phi(\wh H_1)=\wh H_2$. Suppose that there
are an isomorphism $f_H:H_1\to H_2$ and an element $a\in\wh H_2$ such that
\[
\Phi|_{\wh H_1}=\Inn(a)\circ\wh f_H.
\]
Then there are an isomorphism $f:G_1\to G_2$ and an element $u\in\wh G_2$ such that
\[
\Phi=\Inn(u)\circ\wh f.
\]
\end{proposition}
\begin{proof}
We use the classification of extensions with centerless kernel \cite{Brown1982}.

Compose $\Phi$ with $\Inn(a^{-1})$. This normalization does not change the induced finite quotient
map and makes
\[
\Phi|_{\wh H_1}=\wh f_H.
\]
Let $\theta:Q_1\to Q_2$ be the isomorphism induced by $\Phi$. For $q\in Q_1$, choose lifts
$x\in G_1$ and $y\in G_2$ of $q$ and $\theta(q)$. Since $\Phi(x)$ and $y$ have the same image in
$Q_2$, there is $c_q\in\wh H_2$ with
\[
\Phi(x)=y c_q.
\]
For every $h\in H_1$, applying $\Phi$ to $xhx^{-1}$ we find
\[
\wh f_H(xhx^{-1})
=y c_q\,\wh f_H(h)\,c_q^{-1}y^{-1}.
\]
Consequently the automorphism of $\wh H_2$ obtained by sending conjugation by $x$ through $\wh f_H$
differs from conjugation by $y$ by an inner automorphism. Injectivity of $\Out(H_2)\to\Out(\wh H_2)$
therefore shows that the two discrete outer actions agree. Thus $f_H$ and $\theta$ identify the
abstract kernels
\[
Q_1\longrightarrow\Out(H_1),
\qquad
Q_2\longrightarrow\Out(H_2).
\]

For a centerless group $H$ and a homomorphism $\omega:Q\to\Out(H)$, called an abstract kernel, we
set
\[
P_\omega=\{(q,\alpha)\in Q\times\Aut(H):[\alpha]=\omega(q)\},
\]
where $[\alpha]$ is the class of $\alpha$ modulo inner automorphisms. If $1\to H\to G\to Q\to1$
induces $\omega$ by conjugation, then $g\longmapsto\bigl(\bar g,\conj g|_H\bigr)$ is an isomorphism
$G\to P_\omega$, where $\bar g$ is the image of $g$ in $Q$. Injectivity follows from $Z(H)=1$, and
surjectivity follows because any two automorphisms representing the same outer class differ by
conjugation by an element of $H$. This construction goes back to Eilenberg and Mac Lane
\cite{EilenbergMacLane1947}. Equivalently, the obstruction and ambiguity groups have coefficients in
$Z(H)$ and vanish when the center is trivial. See also \cite[Chapter~IV, Section~6]{Brown1982}.
Applied to the identified abstract kernels, this construction produces an isomorphism
\[
f:G_1\xrightarrow{\cong}G_2
\]
inducing $f_H$ and $\theta$.

We set
\[
\Psi=\wh f^{-1}\circ\Phi\in\Aut(\wh G_1).
\]
The map $\Psi$ fixes $\wh H_1$ pointwise and induces the identity on $Q_1$. For $x\in\wh G_1$, we
define
\[
a_x=\Psi(x)x^{-1}\in\wh H_1.
\]
For every $k\in\wh H_1$ one has
\[
a_x(xkx^{-1})a_x^{-1}
=\Psi(x)\,k\,\Psi(x)^{-1}
=\Psi(xkx^{-1})
=xkx^{-1}.
\]
The element $a_x$ centralizes $\wh H_1$, because $x\wh H_1x^{-1}=\wh H_1$, and $Z(\wh H_1)=1$ then
forces $a_x=1$ for every $x$. Hence $\Psi=\id$ and the normalized map is $\wh f$. After restoring
the initial conjugation we have $\Phi=\Inn(u)\circ\wh f$ for some $u\in\wh G_2$.
\end{proof}

We next collect the separability, centerlessness, and injectivity results for outer automorphism
groups used in the realization arguments.

\begin{theorem}[Niblo-Wilton]
\label{thm:surface-double-cosets}
Let $\Pi$ be the fundamental group of a connected compact surface with negative Euler
characteristic, and let $P,Q\leq\Pi$ be finitely generated. Then:
\begin{enumerate}[label=\textup{(\roman*)}]
\item every double coset $PgQ$ is separable in $\Pi$,
\item in the profinite completion,
\[
\cl P\cap\cl Q=\cl{P\cap Q}.
\]
\end{enumerate}
\end{theorem}
\begin{proof}
Part~\textup{(i)} is Niblo's theorem \cite{Niblo1992}. For orientable compact surfaces, we use its
precise formulation in \cite[Theorem~3.3]{Wilton2010}, which yields separability of $P(gQg^{-1})$,
and hence of its right translate $PgQ$. For completeness, if the surface is nonorientable, let
$\Pi_0\normal\Pi$ be the subgroup of its orientable double cover. We set $P_0=P\cap\Pi_0$ and
$Q_0=Q\cap\Pi_0$. Choose finite coset decompositions $P=\bigcup_i p_iP_0$ and $Q=\bigcup_j Q_0q_j$.
Then
\[
PgQ=\bigcup_{i,j}p_i\bigl(P_0(gQ_0g^{-1})\bigr)gq_j.
\]
The two factors inside the parentheses are finitely generated subgroups of $\Pi_0$. Their product is
closed in $\Pi_0$ by the orientable case. Because $\Pi_0$ is closed in $\Pi$ and inherits its full
profinite topology, each of these translates is closed in $\Pi$. The finite union is therefore
closed. Part~\textup{(ii)} follows from Minasyan's criterion for intersections of subgroup closures
for groups with separable double cosets \cite[Corollary~1.2]{Minasyan2023}.
\end{proof}

\begin{theorem}\label{thm:normalizer-comparisons}
Let $G$ be either the fundamental group of a connected closed orientable surface of genus at least
two or the fundamental group of an orientable torsion-free finite volume hyperbolic $3$-manifold.
Then we have:
\begin{enumerate}[label=\textup{(\roman*)}]
\item The group $G$ is hereditarily conjugacy separable.
\item For every $g\in G$,
\[
C_{\wh G}(g)=\cl{C_G(g)},
\]
\item In the case of surface groups, for every pair of finitely generated subgroups
$P,Q\leq G$, every double coset $PgQ$ is separable and
\[
\cl P\cap\cl Q=\cl{P\cap Q}.
\]
In the finite volume hyperbolic $3$-manifold case, the same two conclusions hold for
every pair of finitely generated abelian subgroups $P,Q\leq G$.
\item If $g\in G$ is loxodromic and is not a proper power in $G$, then $\langle g\rangle$ is
a virtual retract, is closed, and inherits its full profinite topology from $G$.
\end{enumerate}
\end{theorem}
\begin{proof}
For surface groups, hereditary conjugacy separability follows from Martino
\cite[Theorem~3.7]{Martino2007}, since every subgroup of finite index is again a surface group. For
finite volume hyperbolic $3$-manifold groups it follows from Hamilton--Wilton--Zalesskii
\cite[Theorem~1.3]{HamiltonWiltonZalesskii2013}. Item~\textup{(ii)} is Minasyan
\cite[Corollary~12.3]{Minasyan2012}. The statement about double cosets in surface groups is
\cref{thm:surface-double-cosets}(i), and the $3$-manifold statement is Hamilton--Wilton--Zalesskii
\cite[Theorem~1.4]{HamiltonWiltonZalesskii2013}. In the surface case, the identity for intersections
of subgroup closures follows from \cref{thm:surface-double-cosets}. In the $3$-manifold case, let
$P,Q$ be finitely generated abelian and let $L\leq G$ have finite index with $P\cap Q\leq L$. Then
$P\cap L$ is finitely generated abelian, so $(P\cap L)Q$ is separable by the same theorem about
double cosets. By the pairwise criterion \cite[Theorem~1.1]{Minasyan2023},
$\cl P\cap\cl Q=\cl{P\cap Q}$. In the Kleinian case item~\textup{(iv)} follows from the alternative
between a virtual retract and a virtual fiber subgroup \cite[(H.9) and
(L.15)]{AschenbrennerFriedlWilton2015}. The virtual fiber alternative cannot hold for
$\langle g\rangle$. Indeed, a fiber in a finite cover of a finite volume hyperbolic $3$-manifold has
negative Euler characteristic, and every finite index subgroup of its fundamental group is
nonabelian. Such a group cannot be commensurable with an infinite cyclic group. In the surface case
item~\textup{(iv)} follows from quasiconvex virtual retraction. Closedness and fullness then follow
from \cref{lem:virtual-retract-topology}.
\end{proof}
\begin{lemma}\label{lem:normalizer-rigidity}
Let $G$ be either the fundamental group of a connected closed orientable surface of genus at least
two or the fundamental group of an orientable torsion-free finite volume hyperbolic $3$-manifold,
regarded as a dense subgroup of $\wh G$. Then
\[
N_{\wh G}(G)=G,
\qquad
Z(\wh G)=1,
\qquad
\Out(G)\longrightarrow\Out(\wh G)
\text{ is injective}.
\]
The same conclusions hold for every finite index subgroup of $G$.
\end{lemma}
\begin{proof}
These facts are elementary, but we did not find a reference for exactly this form, so we give the
proof. All closures in this proof are taken in $\wh G$. We use
\cref{thm:normalizer-comparisons}(i)--(iii).

First choose loxodromic elements $x,y\in G$, neither of which is a proper power, whose axes are
distinct. Such elements exist because $G$ is a non-elementary torsion-free discrete group of
orientation preserving isometries of $\HH^2$ or $\HH^3$. In either case, the centralizer in $G$ of a loxodromic element is infinite cyclic. Choosing generators of these centralizers, we therefore
have
\[
C_G(x)=\langle x\rangle,
\qquad
C_G(y)=\langle y\rangle.
\]
Moreover,
\[
\langle x\rangle\cap\langle y\rangle=\{1\}.
\]
Indeed, an equality $x^r=y^s\neq1$ would force $x$ and $y$ to have the same axis, since a nonzero
power of a loxodromic element has the same axis as that element.

By \cref{thm:normalizer-comparisons}(ii),
\[
C_{\wh G}(x)=\cl{\langle x\rangle},
\qquad
C_{\wh G}(y)=\cl{\langle y\rangle}.
\]
By part~\textup{(iii)} of the same theorem,
\[
\cl{\langle x\rangle}
\cap\cl{\langle y\rangle}
=\cl{\langle x\rangle\cap\langle y\rangle}
=\{1\}.
\]
It also yields separability of the double coset $\langle x\rangle\langle y\rangle$. Consequently,
\[
G\cap
\bigl(
\cl{\langle x\rangle}
\cl{\langle y\rangle}
\bigr)
=\langle x\rangle\langle y\rangle,
\]
since the product of the closures of $\langle x\rangle$ and $\langle y\rangle$ is the closure of the
product $\langle x\rangle\langle y\rangle$.

Now let $u\in N_{\wh G}(G)$. Then $uxu^{-1}\in G$. This element and $x$ are conjugate in $\wh G$. By
conjugacy separability, \cref{thm:normalizer-comparisons}(i), they are conjugate in $G$. Choose
$a\in G$ such that $uxu^{-1}=axa^{-1}$, and put $h=a^{-1}u$. Then $h$ normalizes $G$ and centralizes
$x$, so $h\in\cl{\langle x\rangle}$. Since $h$ normalizes $G$, the element $hyh^{-1}$ belongs to
$G$. Applying conjugacy separability again, choose $b\in G$ such that $hyh^{-1}=byb^{-1}$. It
follows that
\[
b^{-1}h\in C_{\wh G}(y)
=\cl{\langle y\rangle}.
\]
Therefore
\[
b=h(b^{-1}h)^{-1}
\in
G\cap
\bigl(
\cl{\langle x\rangle}
\cl{\langle y\rangle}
\bigr)
=\langle x\rangle\langle y\rangle.
\]
Write
\[
b=x^r y^s
\qquad(r,s\in\Z).
\]
As $h\in\cl{\langle x\rangle}$, we have $x^{-r}h\in\cl{\langle x\rangle}$. On the other hand,
\[
x^{-r}h
=y^s(b^{-1}h)
\in\cl{\langle y\rangle}.
\]
The intersection of these two closures is trivial, so $h=x^r\in G$ and $u=ah\in G$. The reverse
inclusion is immediate, so $N_{\wh G}(G)=G$. If $z\in Z(\wh G)$, then $z$ centralizes both $x$ and
$y$. Hence
\[
z\in
C_{\wh G}(x)\cap C_{\wh G}(y)
=
\cl{\langle x\rangle}
\cap\cl{\langle y\rangle}
=\{1\}.
\]
This proves $Z(\wh G)=1$.

Every automorphism of $G$ extends uniquely to a continuous automorphism of $\wh G$, and completion
sends inner automorphisms to inner automorphisms. It therefore induces a homomorphism on outer
automorphism groups. Suppose that $\alpha\in\Aut(G)$ satisfies
\[
\wh\alpha=\Inn(u)
\qquad\text{for some }u\in\wh G.
\]
Because $\alpha(G)=G$, this equality implies $uGu^{-1}=G$. By the normalizer equality already
proved, $u\in G$. Restricting to $G$ yields $\alpha=\Inn(u)$, so the homomorphism
\[
\Out(G)\longrightarrow\Out(\wh G)
\]
is injective.

Finally, let $L\leq G$ have finite index. The corresponding connected finite cover is again a closed
orientable surface of genus at least two, or an orientable finite volume hyperbolic $3$-manifold,
respectively. The argument above, applied to $L$, shows that
\[
N_{\wh L}(L)=L,
\qquad
Z(\wh L)=1,
\qquad
\Out(L)\longrightarrow\Out(\wh L)
\text{ is injective}.
\]
\end{proof}
\begin{theorem}
\label{thm:torsion-completion-sources}
Every orientable finite volume hyperbolic $3$-manifold group has a finite index subgroup that is the
fundamental group of a compact special cube complex. If a torsion-free group has such a finite index
subgroup, then its profinite completion is torsion-free.
\end{theorem}
\begin{proof}
For closed hyperbolic $3$-manifolds, virtual compact specialness follows from \cite[Theorems~1.1
and~9.3]{Agol2013}. For cusped finite volume hyperbolic $3$-manifolds, it is Wise's theorem in the
form stated in \cite[Theorem~9.2]{WiltonZalesskii2017}. The assertion that a torsion-free virtually
compact special group has torsion-free profinite completion is
\cite[Proposition~3.2]{WiltonZalesskii2017}.
\end{proof}

\begin{theorem}\label{thm:quasiconvex-subgroup-conjugacy}
Let $G$ be a hyperbolic virtually compact special group and let $H,K\leq G$ be infinite quasiconvex
subgroups. Regard $G$ as a dense subgroup of $\wh G$. If
\[
\cl K=z\cl H z^{-1}
\qquad\text{for some }z\in\wh G,
\]
then there is $g\in G$ such that $K=gHg^{-1}$.
\end{theorem}
\begin{proof}
The infinite quasiconvex subgroup case of Chagas--Zalesskii \cite{ChagasZalesskii2016} states that
two nonconjugate infinite quasiconvex subgroups have nonconjugate images in some finite quotient of
$G$. Suppose that $H$ and $K$ were not conjugate in $G$, and choose such a quotient $q:G\to Q$. Its
continuous extension $\wh q:\wh G\to Q$ satisfies
\[
q(K)=\wh q(z)q(H)\wh q(z)^{-1},
\]
because $Q$ is finite and $\wh q(\cl H)=q(H)$ and $\wh q(\cl K)=q(K)$. This contradicts the choice
of $q$ and proves the claim.
\end{proof}


\section{Cut trees by divisible fillings}\label{sec:cut-tree}
Let $S$ be a closed orientable surface of genus $g\geq2$, and let $\alpha\subset S$ be a
nonseparating simple closed curve. Set $\Pi=\pi_1(S)$. Cutting along $\alpha$ produces
$F=S_{g-1,2}$, the compact orientable surface of genus $g-1$ with two boundary components, and the
free group $H=\pi_1(F)$. Choose generators $a_-,a_+\in H$ for the two maximal cyclic boundary
subgroups, with orientations agreeing after gluing. Thus, with the usual based paths, the boundary
orientations of the cut surface are represented by $a_-$ and $a_+^{-1}$. We choose a stable letter
$t$ such that
\begin{equation}\label{eq:surface-HNN}
\Pi=\langle H,t\mid ta_-t^{-1}=a_+\rangle.
\end{equation}
We write $A=A_-=\langle a\rangle$, with $a=a_-$, and $A_+=tAt^{-1}$. For $n\geq3$, we set
\[
K_{\alpha,n}=\Ncl{a^n}{\Pi},
\qquad
N_{\alpha,n}=\cl{K_{\alpha,n}}\normal\wh\Pi,
\qquad
\Gamma_{\alpha,n}=\Pi/K_{\alpha,n}.
\]
By \cref{lem:completion-quotient}, $\wh\Gamma_{\alpha,n}\cong\wh\Pi/N_{\alpha,n}$.

\begin{terminology}
A profinite graph is a profinite space $T=V(T)\sqcup E(T)$ with closed vertex set $V(T)$ and
continuous endpoint maps $d_0,d_1:T\to V(T)$ restricting to the identity on $V(T)$. We set
$E^*(T)=T/V(T)$, with distinguished point $*$. It is a profinite tree if
\[
0\longrightarrow\hZ[[E^*(T),*]]
\xrightarrow{\partial}\hZ[[V(T)]]\xrightarrow{\epsilon}\hZ\longrightarrow0
\]
is exact, where $\partial(e)=d_1(e)-d_0(e)$ and $\epsilon(v)=1$. The double brackets denote inverse
limits of free modules on finite quotient sets, with the distinguished point omitted in the pointed
case. See \cite{WiltonZalesskii2019}.
\end{terminology}

\begin{theorem}[Ribes]
\label{thm:finite-edge-completion}
Let $(\mathcal G,Y)$ be a finite connected efficient graph of groups with fundamental group
$\Gamma$. Fix a maximal subtree $Y_0\subset Y$ and use the corresponding vertex embeddings and
stable letters in the discrete and profinite presentations. The continuous homomorphism from the
profinite fundamental group of the graph of completed groups to $\wh\Gamma$ that extends the
completed vertex embeddings and sends each stable letter to the image of the corresponding discrete
stable letter is an isomorphism. If every edge group is finite, the completed graph of groups is
proper and its standard profinite Bass--Serre tree has the coset description: vertex and edge
stabilizers are the corresponding conjugates.
\end{theorem}
For the proof see \cite{Ribes2017}.

The HNN splitting \eqref{eq:surface-HNN} obtained by cutting the surface is efficient. Indeed, $H$
and $A_\pm$ are finitely generated quasiconvex subgroups of the hyperbolic surface group $\Pi$.
Surface groups are virtually compact special, so the virtual retraction theorem for quasiconvex
subgroups makes these subgroups retracts of finite index subgroups of $\Pi$
\cite{HaglundWise2008,Wise2021}. By \cref{lem:virtual-retract-topology}, they are closed in the
profinite topology and inherit their full profinite topologies. The standard completion theorem for
finite efficient graphs of groups then identifies the profinite fundamental group of the completed
HNN splitting with $\wh\Pi$ \cite{Ribes2017}. In particular, the full profinite cut tree
$\wh T_\alpha$ used below is well defined and has vertex and edge spaces $\wh\Pi/\cl H$ and
$\wh\Pi/\cl A$, respectively.

\begin{terminology}
Let $p$ be a prime, $n\geq0$ an integer, and $P$ a profinite group. We write
$\Z_p[[P]]=\varprojlim_{r,U}(\Z/p^r\Z)[P/U]$, where $r\geq1$ and $U$ ranges over open normal
subgroups of $P$. The group $P$ is a profinite $PD^n$ group at $p$ if the trivial module $\Z_p$ has
a resolution by finitely generated projective profinite $\Z_p[[P]]$-modules, $\cd_p(P)=n$, and
\[
H^j(P,\Z_p[[P]])=0\quad(j\neq n),
\qquad H^n(P,\Z_p[[P]])\cong\Z_p
\]
as abelian groups. Cohomology here is continuous cohomology with profinite coefficients. This is the
definition of \cite{WiltonZalesskii2019}.
\end{terminology}

\begin{theorem}\label{thm:fuchsian-fixed}
Let $V$ be a non-elementary cocompact Fuchsian group. Every continuous action of $\wh V$ on a
profinite tree, in the sense of \cite[Definition~1.2]{WiltonZalesskii2019}, with finite edge
stabilizers has a fixed vertex.

In addition, in the standard tree of a proper finite graph of profinite groups with finite edge
groups, the intersection of two distinct vertex stabilizers is finite. More precisely, it is
contained in the stabilizer of every edge in the smallest profinite subtree joining the two
vertices.
\end{theorem}
\begin{proof}
By Selberg's lemma, we choose a torsion-free normal subgroup $\Sigma\normalf V$. As $V$ is cocompact
Fuchsian, there is a closed orientable hyperbolic surface $S_\Sigma$ with $\Sigma=\pi_1(S_\Sigma)$.
A finite index subgroup is closed and inherits its full profinite topology from the  group,
so the closure of $\Sigma$ in $\wh V$ is an open normal subgroup canonically identified with
$\wh\Sigma$.

The group $\Sigma$ is a $PD^2$ group, that is, a Poincar\'e duality group of dimension $2$
\cite{Brown1982}. It is good in the sense of Serre because it is the fundamental group of the
compact $3$-manifold $S_\Sigma\times I$ \cite{AschenbrennerFriedlWilton2015}. By
\cite[Theorem~1.10]{WiltonZalesskii2019}, $\wh\Sigma$ is a profinite $PD^2$ group at every prime. In
particular, $\cd_p(\wh\Sigma)=2$ for every prime $p$.

The group $\wh\Sigma$ is torsion-free, for if $\wh\Sigma$ contained an element of finite order, then
it would contain a closed subgroup $C_p$ of prime order for some $p$. By monotonicity of
$p$-cohomological dimension for closed subgroups we have
\[
\cd_p(C_p)
\leq
\cd_p(\wh\Sigma)=2,
\]
which contradicts $\cd_p(C_p)=\infty$ \cite{RibesZalesskii2010}. Hence $\wh\Sigma$ has no nontrivial
finite subgroup.

We now restrict the given action to $\wh\Sigma$. For every edge $e$ of the profinite tree,
\[
\Stab_{\wh\Sigma}(e)
=
\wh\Sigma\cap\Stab_{\wh V}(e)
=
1,
\]
because the $\wh V$-stabilizer of $e$ is finite and $\wh\Sigma$ is torsion-free. Thus every edge
stabilizer of the restricted action has cohomological dimension $0=2-2$. Corollary~1.11 of
Wilton--Zalesskii \cite{WiltonZalesskii2019} applies to this good $PD^2$ group and shows that
\[
T^{\wh\Sigma}\neq\varnothing.
\]
It follows that $\wh\Sigma$ has a fixed vertex. Alternatively, one can apply
\cite[Theorem~1.7]{WiltonZalesskii2019}, since the restricted edge stabilizers have cohomological
dimension $0<1$.

In fact, this fixed vertex is unique. Suppose that $\wh\Sigma$ fixed two distinct vertices. Then for
each $h\in\wh\Sigma$, by \cite[Theorem~4.1.5]{Ribes2017} its fixed point set is a closed profinite
subtree, so by \cite[Proposition~2.4.9]{Ribes2017} it contains the smallest profinite subtree
joining those vertices. Hence every element of $\wh\Sigma$ fixes this subtree pointwise. The subtree
contains an edge, so $\wh\Sigma$ would be contained in an edge stabilizer, which is trivial for the
restricted action. This is impossible because the residually finite infinite group $\Sigma$ embeds
densely in $\wh\Sigma$. Thus $T^{\wh\Sigma}$ is a singleton. As $\wh\Sigma\normal\wh V$, the group
$\wh V$ preserves this singleton and therefore fixes its unique vertex.

Now let $T$ be the standard tree of a proper finite graph of
profinite groups with finite edge groups, and let $v\neq w$ be
distinct vertices. Let $U$ be the smallest profinite subtree
containing $v$ and $w$. Since $U$ contains two distinct vertices,
it contains an edge $e$. By the same two results from
\cite{Ribes2017}, every element fixing both $v$ and $w$ fixes $U$
pointwise. Hence, $\Stab(v)\cap\Stab(w)\leq\Stab(e).$
The edge stabilizer is finite, which proves the claim.
\end{proof}
Our first lemma is a routine consequence of the orbifold presentation and of profinite Bass--Serre
theory; we prove the HNN identification and efficiency, after which the completion and tree
statements follow from \cite[Proposition~6.5.3 and Corollary~6.3.6]{Ribes2017}.
\begin{lemma}\label{lem:cone-HNN}
Let $n\geq3$ and
\[
\Gamma_{\alpha,n}=\HNN(V_{\alpha,n},C_n^-,C_n^+),
\]
where
\[
V_{\alpha,n}=H/\Ncl{a_-^n,a_+^n}{H}
\]
is the cocompact orientable Fuchsian orbifold group of signature $(g-1,n,n)$, meaning the orbifold
has underlying surface a closed surface of genus $g-1$ and two cone points of order $n$, and
$C_n^\pm\cong\Z/n\Z$ are generated by the images of $a_\pm$. This HNN splitting is efficient and
proper after completion. We write $T_{\alpha,n}$ for its standard profinite tree.
\end{lemma}
\begin{proof}
Quotienting \eqref{eq:surface-HNN} by the normal closure of $a_-^n$ also imposes $a_+^n=1$, and we
obtain
\[
\Gamma_{\alpha,n}
\cong
\langle V_{\alpha,n},t\mid tc_-t^{-1}=c_+\rangle.
\]
The group $V_{\alpha,n}$ has presentation
\[
\left\langle x_1,y_1,\ldots,x_{g-1},y_{g-1},c_-,c_+
\ \middle|\
\prod_{j=1}^{g-1}[x_j,y_j]c_-c_+^{-1}=1,
\ c_-^n=c_+^n=1
\right\rangle.
\]
Since the defining relations show that the orders of $c_-$ and $c_+$ divide $n$, and the assignments
\[
c_-\longmapsto 1,\qquad c_+\longmapsto1,
\qquad x_j,y_j\longmapsto0
\]
define a homomorphism $V_{\alpha,n}\to\Z/n\Z$, both $c_-$ and $c_+$ have order at least $n$, and
hence $n$. Moreover,
\[
\chi_{\mathrm{orb}}=2-2(g-1)-2\left(1-\frac1n\right)
=2-2g+\frac2n<0.
\]
Replacing $c_+$ by $c_+^{-1}$ in the vertex presentation turns it into the usual oriented orbifold
presentation of signature $(g-1,n,n)$. Thus $V_{\alpha,n}$ is a non-elementary cocompact Fuchsian
group. Abelianizing the full HNN presentation, we find
\[
(\Gamma_{\alpha,n})_{\mathrm{ab}}
\cong \Z^{2g-1}\oplus\Z/n\Z.
\]
This agrees with imposing $n[a]=0$ on the primitive class $[a]\in H_1(S,\Z)$.

Next we show efficiency. We first prove that the vertex group inherits its full profinite topology.
For this argument, write
\[
V=V_{\alpha,n},
\qquad
\Gamma=\Gamma_{\alpha,n}.
\]
Let $N\normalf V$ be arbitrary. The set
\[
D=(C_n^-\cup C_n^+)\setminus\{1\},\qquad  C_n^\pm=\langle c_\pm\rangle
\]
is finite. By residual finiteness of $V$, for each $d\in D$ there exists $L_d\normalf V$ such that
$d\notin L_d$. Set
\[
M=N\cap\bigcap_{d\in D}L_d.
\]
This is a finite index normal subgroup of $V$, and
\[
M\leq N,
\qquad
M\cap C_n^-=M\cap C_n^+=\{1\}.
\]
Consequently, the quotient homomorphism
\[
\pi_M:V\longrightarrow B_M:=V/M
\]
is injective on both $C_n^-$ and $C_n^+$.

In particular, $\pi_M(c_-)$ and $\pi_M(c_+)$ both have order $n$. So the rule
\[
\pi_M(c_-^k)\longmapsto\pi_M(c_+^k)
\]
defines an isomorphism between their cyclic subgroups. We may therefore form the HNN extension
\[
\Gamma_M
=
\left\langle B_M,s
\ \middle|\
s\pi_M(c_-)s^{-1}=\pi_M(c_+)
\right\rangle.
\]
By the normal form theorem for HNN extensions, $B_M$ embeds in $\Gamma_M$. The assignments
\[
v\longmapsto\pi_M(v)\quad(v\in V),
\qquad
t\longmapsto s
\]
respect the defining relation of $\Gamma$ and therefore define a surjective homomorphism
$q_M:\Gamma\longrightarrow\Gamma_M$.

The group $\Gamma_M$ is the fundamental group of a finite graph of finite groups, hence virtually
free and residually finite. For each $b\in B_M\setminus\{1\}$ choose a homomorphism
$\theta_b:\Gamma_M\longrightarrow F_b$ to a finite group such that $\theta_b(b)\neq1$. Taking the
product we get a homomorphism $\theta:\Gamma_M\longrightarrow \prod_{b\in B_M\setminus\{1\}}F_b$.
Then $\theta|_{B_M}$ is injective, since every nonidentity element of $B_M$ has nonidentity image in
at least one coordinate.

Set $K=\ker(\theta\circ q_M)$. The group on the right of $\theta$ is finite, so $K\normalf\Gamma$.
For $v\in V$, injectivity of $\theta|_{B_M}$ shows that $v\in K$ if and only if $v\in M$, that is,
$K\cap V=M\leq N$.

The neighborhoods of the identity in the topology induced from the profinite topology of $\Gamma$
have a basis consisting of the subgroups $K'\cap V$, where $K'\normalf\Gamma$. Each such
intersection is a finite index normal subgroup of $V$. Conversely, the construction above shows that
every $N\normalf V$ contains such an intersection. These two neighborhood bases are therefore
cofinal, so the induced topology on $V$ equals its full profinite topology. We next prove that $V$
is closed in the profinite topology of $\Gamma$. Let $w\in\Gamma\setminus V$. By the HNN normal form
theorem, we can write
\[
w=v_0t^{\varepsilon_1}v_1\cdots t^{\varepsilon_r}v_r,
\qquad r\geq1,
\quad v_i\in V,
\quad \varepsilon_i\in\{1,-1\},
\]
with no pinch. Thus, for $1\leq i<r$,
\[
\begin{aligned}
v_i&\notin C_n^- &&\text{if }(\varepsilon_i,\varepsilon_{i+1})=(1,-1),\\
v_i&\notin C_n^+ &&\text{if }(\varepsilon_i,\varepsilon_{i+1})=(-1,1).
\end{aligned}
\]
Choose $M\normalf V$ meeting both $C_n^-$ and $C_n^+$ trivially and satisfying
\[
\begin{aligned}
\pi_M(v_i)&\notin\pi_M(C_n^-) &&\text{if }(\varepsilon_i,\varepsilon_{i+1})=(1,-1),\\
\pi_M(v_i)&\notin\pi_M(C_n^+) &&\text{if }(\varepsilon_i,\varepsilon_{i+1})=(-1,1).
\end{aligned}
\]
To obtain $M$, take all nonidentity elements of $C_n^-$ and $C_n^+$, together with the elements
$v_ic^{-1}$ for $c\in C_n^-$ when $(\varepsilon_i,\varepsilon_{i+1})=(1,-1)$, and for $c\in C_n^+$
when $(\varepsilon_i,\varepsilon_{i+1})=(-1,1)$. Since the edge groups are finite and the word has
no pinch, this is a finite set of nonidentity elements. As $V$ is residually finite, the product of
finitely many separating quotients is a finite quotient of $V$ in which each of these elements is
nontrivial, and its kernel is a finite index normal subgroup $M$ with these properties. The image
\[
q_M(w)=\pi_M(v_0)s^{\varepsilon_1}\pi_M(v_1)\cdots s^{\varepsilon_r}\pi_M(v_r)
\]
is still reduced and has positive stable letter length, so $q_M(w)\notin B_M$ by the HNN normal form
theorem. For every $b\in B_M$, since $\Gamma_M$ is residually finite there is a homomorphism
$\eta_b:\Gamma_M\to Q_b$ to a finite group with
\[
\eta_b\bigl(q_M(w)b^{-1}\bigr)\neq1.
\]
The product homomorphism
\[
\eta=(\eta_b)_{b\in B_M}:\Gamma_M\longrightarrow\prod_{b\in B_M}Q_b
\]
then satisfies $\eta(q_M(w))\notin\eta(B_M)$. As $q_M(V)=B_M$, the finite quotient defined by
$\eta\circ q_M$ separates $w$ from $V$. This proves that $V$ is closed in the profinite topology of
$\Gamma$.

The same argument separates every element of $\Gamma\setminus V$ from the identity, and if
$1\neq v\in V$, we choose $N\normalf V$ with $v\notin N$; the construction above of
$K\normalf\Gamma$ with $K\cap V\leq N$ then separates $v$ from the identity. So $\Gamma$ is
residually finite. In particular the finite edge groups are closed in $\Gamma$ and carry their full
profinite topologies, and together with the statement about $V$ proved above this is efficiency. By
\cref{thm:finite-edge-completion}, the canonical homomorphism from the profinite fundamental group
of the completed HNN splitting to $\wh\Gamma$ is an isomorphism, and the completed splitting is
proper, hence its standard profinite tree has the coset description given in the statement.
\end{proof}
The next lemma combines the behaviour of $\theta$ on the closed normal subgroups defining the
fillings with the fixed point theorem (\cref{thm:fuchsian-fixed}).
\begin{lemma}\label{lem:cone-recognition}
For $i=1,2$, use the notation above for $(S_i,\alpha_i)$. If
\[
\theta:(\wh\Pi_1,[\cl A_1])\xrightarrow{\cong}(\wh\Pi_2,[\cl A_2]),
\]
then, for every $n\geq3$, $\theta$ induces
\[
\theta_n:\wh\Gamma_{\alpha_1,n}\xrightarrow{\cong}\wh\Gamma_{\alpha_2,n},
\]
and $\theta_n$ maps the completed Fuchsian vertex group of $\wh\Gamma_{\alpha_1,n}$ onto a conjugate
of the completed Fuchsian vertex group of $\wh\Gamma_{\alpha_2,n}$.
\end{lemma}
\begin{proof}
After an inner conjugation we may assume $\theta(\cl A_1)=\cl A_2$. A continuous isomorphism of
procyclic groups sends a topological generator to a unit power of a topological generator, so
\[
\theta(a_1)=a_2^\mu,
\qquad \mu\in\wh\Z^{\times}.
\]
Multiplication by $\mu$ preserves the open subgroup $n\wh\Z$. So $a_2^{n\mu}$ and $a_2^n$
topologically generate the same subgroup. Taking closed normal closures, we get
$\theta(N_{\alpha_1,n})=N_{\alpha_2,n}$, and therefore the quotient isomorphism $\theta_n$.

Let $W_i$ be the completed Fuchsian vertex group and $v_i$ its vertex in $T_{\alpha_i,n}$. Then the
group $\theta_n(W_1)$ acts on $T_{\alpha_2,n}$, whose edge stabilizers are finite, so by
\cref{thm:fuchsian-fixed} there is a fixed vertex $v$ with $\theta_n(W_1)\leq\Stab(v)$. Applying
\cref{thm:fuchsian-fixed} to $\theta_n^{-1}(\Stab(v))$ acting on $T_{\alpha_1,n}$, we find a fixed
vertex $w$ with
\[
W_1\leq\theta_n^{-1}(\Stab(v))\leq\Stab(w).
\]
If $w\neq v_1$, the intersection of the distinct vertex stabilizers $W_1$ and $\Stab(w)$ would
contain the infinite group $W_1$, contradicting the finiteness of intersections of distinct vertex
stabilizers in \cref{thm:fuchsian-fixed}. It follows that $w=v_1$ and all inclusions are equalities.
So $\theta_n(W_1)$ is a vertex stabilizer, and therefore a conjugate of $W_2$.
\end{proof}

Let
\[
\rho_{\alpha,n}:\wh\Pi\longrightarrow\wh\Gamma_{\alpha,n},
\qquad
P_{\alpha,n}=\rho_{\alpha,n}^{-1}(\wh V_{\alpha,n}).
\]

\begin{lemma}\label{lem:preimage-divisibility}
For every $n\geq3$ we have
\[
P_{\alpha,n}=\cl H\,N_{\alpha,n}.
\]
If $n\mid m$, then
\begin{equation}\label{eq:divisibility}
N_{\alpha,m}\leq N_{\alpha,n},
\qquad
P_{\alpha,n}=P_{\alpha,m}N_{\alpha,n}.
\end{equation}
\end{lemma}
\begin{proof}
Let $q=\rho_{\alpha,n}:\wh\Pi\to\wh\Gamma_{\alpha,n}$, and let $V_n$ denote the completed Fuchsian
vertex group. The discrete image $q(H)$ is $V_{\alpha,n}$. Because $\cl H$ is compact and $q$ is
continuous, its image is closed and
\[
q(\cl H)=\cl{q(H)}=\cl{V_{\alpha,n}}=V_n,
\]
where the last equality uses \cref{lem:cone-HNN}. For any quotient $q:G\to G/N$ and subgroup
$L\leq G$ we have
\[
q^{-1}(q(L))=LN;
\]
indeed $q(x)\in q(L)$ if and only if $q(x)=q(\ell)$ for some $\ell\in L$, which is equivalent to
$\ell^{-1}x\in N$ and hence to $x\in LN$. Taking $L=\cl H$ and $N=N_{\alpha,n}$ proves
$P_{\alpha,n}=\cl HN_{\alpha,n}$.

Suppose $n\mid m$, say $m=rn$. Then $a^m=(a^n)^r$, so the closed normal subgroup generated by $a^m$
is contained in the one generated by $a^n$: $N_{\alpha,m}\leq N_{\alpha,n}$. Since $N_{\alpha,n}$ is
normal we have
\[
P_{\alpha,m}N_{\alpha,n}
=\cl HN_{\alpha,m}N_{\alpha,n}
=\cl HN_{\alpha,n}
=P_{\alpha,n}.
\]
The product $\cl HN_{\alpha,n}$ is a subgroup because the second factor is normal.
\end{proof}

\begin{theorem}[Scott]
\label{thm:scott-separability}
Let $\Pi$ be the fundamental group of a connected compact surface and let $H\leq\Pi$ be finitely
generated. Then $H$ is separable in $\Pi$ \cite{Scott1978}; equivalently, $H$ is closed in the
profinite topology of $\Pi$.
\end{theorem}

\begin{lemma}\label{lem:divisible-intersection}
\[
\bigcap_{n\geq3}P_{\alpha,\ell_n}=\cl H.
\]
\end{lemma}
\begin{proof}
From the formula $P_{\alpha,m}=\cl HN_{\alpha,m}$ we have $\cl H\subseteq P_{\alpha,m}$ for every
$m$, hence $\cl H\subseteq\bigcap_{n\geq3}P_{\alpha,\ell_n}$.

For the reverse inclusion, let $x\in\wh\Pi\setminus\cl H$. The subgroup $\cl H$ is closed in
$\wh\Pi$ by definition and is compact. The compact set $\cl H$ and the point $x$ can therefore be
separated in a finite quotient: there is an open normal subgroup $U\normalo\wh\Pi$ such that
\[
x\notin\cl H U.
\]
Let $d$ be the order of the image of $a$ in the finite group $\wh\Pi/U$. For every
$n\geq\max\{3,d\}$ one has $d\mid\ell_n$, and therefore
\[
a^{\ell_n}\in U,
\qquad
N_{\alpha,\ell_n}\leq U.
\]
The preimage formula of \cref{lem:preimage-divisibility} at the exponent $\ell_n$ reads
\[
P_{\alpha,\ell_n}
=\cl HN_{\alpha,\ell_n}
\leq\cl H U.
\]
Thus $x\notin P_{\alpha,\ell_n}$ for every sufficiently large $n$. Every point outside $\cl H$ is
excluded by the intersection, proving equality.
\end{proof}

\begin{proposition}\label{prop:vertex-recognition}
Under the conditions of \cref{lem:cone-recognition}, there is an element $g\in\wh\Pi_2$ such that
\[
\theta(\cl H_1)=g\cl H_2g^{-1}.
\]
\end{proposition}
\begin{proof}
For $n\geq3$, we set
\[
C_n=\{g\in\wh\Pi_2:\theta(P_{\alpha_1,\ell_n})=gP_{\alpha_2,\ell_n}g^{-1}\}.
\]
By \cref{lem:cone-recognition}, each $C_n$ is nonempty. It is a coset of the closed normalizer of
$P_{\alpha_2,\ell_n}$ and hence closed. Since $\ell_n\mid\ell_{n+1}$, if $g\in C_{n+1}$, then
\eqref{eq:divisibility}, normality of $N_{\alpha_2,\ell_n}$, and
$\theta(N_{\alpha_1,\ell_n})=N_{\alpha_2,\ell_n}$ give
\[
\theta(P_{\alpha_1,\ell_n})
=\theta(P_{\alpha_1,\ell_{n+1}}N_{\alpha_1,\ell_n})
=gP_{\alpha_2,\ell_{n+1}}g^{-1}N_{\alpha_2,\ell_n}
=gP_{\alpha_2,\ell_n}g^{-1}.
\]
Thus $C_{n+1}\subseteq C_n$. By compactness we get $g\in\bigcap_{n\geq3}C_n$. A homeomorphism
commutes with arbitrary intersections, so \cref{lem:divisible-intersection} yields the result.
\end{proof}

\begin{theorem}[Wilton--Zalesskii]\label{thm:profinite-malnormality}
Let $G$ be word-hyperbolic and suppose that $G$ is the fundamental group of a compact virtually
special cube complex. If $\{H_1,\ldots,H_r\}$ is a malnormal family of quasiconvex subgroups of $G$,
then $\{\cl H_1,\ldots,\cl H_r\}$ is a malnormal family in $\wh G$.
\end{theorem}
This is \cite[Theorem~3.3]{WiltonZalesskii2017}.

\begin{terminology}[\cite{WiltonZalesskii2017}]
For an integer $k\geq0$, an action of a profinite group $P$ on
a profinite tree $T$ is $k$-acylindrical if, for every
$1\neq g\in P$ and every two vertices $v,w$ fixed by $g$, $d_T(v,w)\leq k.$
Here $d_T(v,w)$ is the number of edges in a shortest finite
edge path joining $v$ and $w$, and is infinite if no such
path exists. This is
\cite[Definition~7.2]{WiltonZalesskii2017}.
\end{terminology}

Profinite malnormality \cite[Theorem~3.3]{WiltonZalesskii2017} and the acylindricity criterion for a
graph of profinite groups with one edge \cite[Lemma~7.3]{WiltonZalesskii2017} are the two
ingredients of the next lemma; the criterion for adjacency comes from the edge stabilizers being
infinite procyclic.
\begin{lemma}
\label{lem:full-acylindricity}
The action of $\wh\Pi$ on the full cut tree $\wh T_\alpha$ is $1$-acylindrical. Consequently, for
two distinct vertices $v,w\in V(\wh T_\alpha)$,
\[
\Stab_{\wh\Pi}(v)\cap\Stab_{\wh\Pi}(w)
\text{ is infinite}
\quad\Longleftrightarrow\quad
v\text{ and }w\text{ are adjacent}.
\]
In the adjacent case this intersection is the infinite procyclic stabilizer of their common edge.
\end{lemma}
\begin{proof}
The subgroup $A=\langle a\rangle$ is maximal cyclic in the torsion-free hyperbolic surface group
$\Pi$. Hence it is quasiconvex and malnormal. Surface groups are fundamental groups of compact
special cube complexes, so by \cref{thm:profinite-malnormality}
\begin{equation}\label{eq:completed-curve-malnormal}
\cl A\cap z\cl A z^{-1}=1
\qquad(z\in\wh\Pi\setminus\cl A).
\end{equation}
Since the original HNN splitting obtained by cutting the surface is efficient, its completed graph
of profinite groups is proper, has profinite fundamental group $\wh\Pi$, and has edge group $\cl A$.
By \eqref{eq:completed-curve-malnormal}, the edge group is malnormal in the profinite fundamental
group. Wilton--Zalesskii's standard profinite Bass--Serre acylindricity lemma
\cite[Lemma~7.3]{WiltonZalesskii2017} therefore shows that the action on $\wh T_\alpha$ is
$1$-acylindrical.

Let $v\neq w$. If they are adjacent, then their unique common edge $e$ satisfies
\[
\Stab(v)\cap\Stab(w)=\Stab(e).
\]
Indeed, an element fixing both endpoints fixes the smallest profinite subtree joining them, which
here is the single edge $e$ with its endpoints, and its stabilizer is a conjugate of $\cl A\cong\hZ$
and is infinite.

Conversely, suppose that $1\neq z\in\Stab(v)\cap\Stab(w)$. By \cite[Theorem~4.1.5]{Ribes2017}, the
nonempty fixed subgraph $U=(\wh T_\alpha)^z$ is a profinite subtree. The coset description of the
edges and \eqref{eq:completed-curve-malnormal} show that two distinct edges have trivially
intersecting stabilizers, so $U$ has at most one edge.

Exactness of the sequence defining a profinite tree implies that an edge-free profinite tree has
only one vertex. As $v,w\in V(U)$ are distinct, $U$ has one edge, say with endpoints $v_0,v_1$. Then
it has no further vertex, for an additional vertex $u$ could be separated from the finite set
$\{v_0,v_1\}$ by a continuous function
\[
f:V(U)\longrightarrow\F_2,
\qquad f(u)=1,
\qquad f(v_0)=f(v_1)=0.
\]
The induced continuous homomorphism $\hZ[[V(U)]]\to\F_2$ would vanish on the image of the
boundary mapping but not on the augmentation-zero element $u-v_0$, contradicting exactness.
Therefore $v,w$ are the endpoints of the unique edge of $U$, so they are adjacent.

\end{proof}

\begin{theorem}
\label{thm:tree-correspondence}
For $i=1,2$, let $S_i$ be a closed orientable surface of genus at least two, $\Pi_i=\pi_1(S_i)$, and
let $\alpha_i\subset S_i$ be a nonseparating simple closed curve with maximal cyclic subgroup
$A_i\leq\Pi_i$. If
\[
\theta:\wh\Pi_1\xrightarrow{\cong}\wh\Pi_2,
\qquad
\theta(\cl A_1)=x\cl A_2x^{-1}
\quad(x\in\wh\Pi_2),
\]
then $\theta$ induces a $\theta$-equivariant isomorphism of the associated full profinite cut trees
\[
\wh T_{\alpha_1}\xrightarrow{\cong}\wh T_{\alpha_2}.
\]
This isomorphism preserves endpoint occurrences. If the chosen edge orientations disagree, we
reverse the two endpoint labels throughout $\wh T_{\alpha_2}$.
\end{theorem}
\begin{proof}
By \cref{prop:vertex-recognition}, choose $g\in\wh\Pi_2$ with $\theta(\cl H_1)=g\cl H_2g^{-1}$. We
define
\[
F_V(q\cl H_1)=\theta(q)g\cl H_2.
\]
This is a continuous $\theta$-equivariant bijection of vertex spaces. By
\cref{lem:full-acylindricity}, adjacency of distinct vertices is characterized by infinite
intersection of stabilizers, so $F_V$ preserves adjacency in both directions.

Let $e_0$ be the edge represented by the identity coset $\cl A_1$ in
$E(\wh T_{\alpha_1})=\wh\Pi_1/\cl A_1$. Then $\Stab_{\wh\Pi_1}(e_0)=\cl A_1$. Let $v_0,v_1$ be its
endpoints. Their images under $F_V$ are adjacent and hence are the endpoints of a unique edge
$e'_0$. Because $F_V$ is a $\theta$-equivariant bijection,
$\Stab_{\wh\Pi_2}(F_V(v))=\theta\bigl(\Stab_{\wh\Pi_1}(v)\bigr)$ for every vertex $v$. Applying
\cref{lem:full-acylindricity} to both trees, we obtain
\[
\begin{aligned}
\Stab_{\wh\Pi_2}(e'_0)
&=
\Stab_{\wh\Pi_2}(F_V(v_0))
\cap
\Stab_{\wh\Pi_2}(F_V(v_1))\\
&=
\theta\bigl(
\Stab_{\wh\Pi_1}(v_0)
\cap
\Stab_{\wh\Pi_1}(v_1)
\bigr)\\
&=
\theta\bigl(\Stab_{\wh\Pi_1}(e_0)\bigr)
=
\theta(\cl A_1).
\end{aligned}
\]
Therefore
\[
F_E(q\cl A_1)=\theta(q)e_0'
\]
is well defined, continuous, bijective, and equivariant, and it preserves unordered endpoint pairs.

Now let $d_{i,0},d_{i,1}$ be the endpoint maps of the $i$th tree. After reversing the labels if
necessary, we have $d_{2,r}(e'_0)=F_V(d_{1,r}(e_0))$ for $r=0,1$. Each edge space is one group
orbit, so equivariance makes the reversal constant, and
\[
d_{2,r}\circ F_E=F_V\circ d_{1,r}\qquad(r=0,1).
\]
The same reversal permutes the endpoint positions in \cref{def:cubical-occurrence}, so it is
compatible with the cubical occurrence structure maps.
\end{proof}

\begin{proof}[Proof of \Cref{thm:finite-cone-main}]
Set $\theta'=\Inn(x^{-1})\circ\theta$, so that $\theta'(\cl A_1)=\cl A_2$. We prove the theorem
first for $\theta'$. At the end, left multiplication by $x$ on the second cut tree restores the
original equivariance by $\theta$ and replaces the conjugator obtained for $\theta'$ by its product
with $x$. To simplify notation, we write $\theta$ for $\theta'$ during the normalized argument. For
every integer $m\geq3$, \cref{lem:cone-recognition} provides an induced quotient isomorphism
\[
\theta_m:\wh\Gamma_{\alpha_1,m}\longrightarrow\wh\Gamma_{\alpha_2,m}
\]
which sends the completed Fuchsian vertex group to a conjugate. Pulling that vertex group back to
the original completions, we get $\theta(P_{\alpha_1,m})=g_mP_{\alpha_2,m}g_m^{-1}$ for some
$g_m\in\wh\Pi_2$.

We set $m=\ell_n$. By \cref{lem:divisible-intersection},
\[
\cl H_i=\bigcap_{n\geq3}P_{\alpha_i,\ell_n}.
\]
The conjugators $g_m$ cannot simply be chosen independently, because they may then fail to converge.
Instead, \cref{prop:vertex-recognition} uses the formula $P_{\alpha,m}=P_{\alpha,m'}N_{\alpha,m}$
whenever $m\mid m'$ to show that the closed conjugator sets for consecutive $\ell_n$ are nested, and
by compactness there is one $g\in\wh\Pi_2$ conjugating every level simultaneously. Intersecting over
$n$, we obtain
\[
\theta(\cl H_1)=g\cl H_2g^{-1}.
\]

The vertex space of the cut tree is $\wh\Pi_i/\cl H_i$, so this equality induces a continuous
equivariant bijection on vertices. Since two distinct vertices are adjacent when their stabilizers
have infinite intersection, by \cref{lem:full-acylindricity} this condition is preserved by
$\theta$. An adjacent pair determines a unique edge, whose stabilizer is the intersection of the two
vertex stabilizers and hence a conjugate of $\cl A_i$, so \cref{thm:tree-correspondence} extends the
vertex map uniquely to an equivariant isomorphism of the profinite trees.

Every subgroup used here is obtained from the marked procyclic subgroup by taking power subgroups,
closed normal closures, quotient vertex groups, and intersections. So a marked profinite isomorphism
sends the completed cut surface vertex group to the corresponding conjugacy class and induces the
tree isomorphism. We now check that the result does not depend on the conjugator. For either cut
tree, let $v$ be the vertex with stabilizer $\cl H$. Then
\[
N_{\wh\Pi}(\cl H)=\cl H.
\]
Indeed, if $n$ normalizes $\cl H$, then $v$ and $nv$ have the same stabilizer $\cl H$. If they were
distinct, \cref{lem:full-acylindricity} would make them adjacent, and their stabilizer intersection
would be a procyclic edge stabilizer. This is impossible because $H$ is a nonabelian free group
embedded in $\cl H$. So $nv=v$, and $n\in\cl H$. The reverse inclusion is trivial.

If $g'$ also satisfies $\theta(\cl H_1)=g'\cl H_2(g')^{-1}$, then
$g^{-1}g'\in N_{\wh\Pi_2}(\cl H_2)=\cl H_2$. Hence $g'\cl H_2=g\cl H_2$, and the vertex map
$q\cl H_1\mapsto\theta(q)g\cl H_2$ is independent of the choice. The same normalizer calculation
shows that a vertex is determined by its stabilizer. Equivariance and bijectivity of $F_V$ give
\[
\Stab_{\wh\Pi_2}(F_V(v))
=
\theta\bigl(\Stab_{\wh\Pi_1}(v)\bigr),
\]
so this equality determines $F_V(v)$ for every vertex $v$. Because each edge is determined by its
two endpoints, the vertex map also determines the edge map. So there is a unique
$\theta$-equivariant tree isomorphism, with endpoint labels treated according to
\cref{thm:tree-correspondence}.

Let $F_\theta$ be this tree isomorphism. The identity tree map is equivariant for the identity group
isomorphism, so uniqueness forces $F_{\id}=\id$. Now let
\[
\theta:\wh\Pi_1\xrightarrow{\cong}\wh\Pi_2,
\qquad
\eta:\wh\Pi_2\xrightarrow{\cong}\wh\Pi_3
\]
be composable marked isomorphisms. For every $g\in\wh\Pi_1$ and every vertex or edge $x$ we have
\[
(F_\eta\circ F_\theta)(gx)
=
(\eta\circ\theta)(g)\,
(F_\eta\circ F_\theta)(x).
\]
Therefore $F_\eta\circ F_\theta$ is an $(\eta\circ\theta)$-equivariant tree isomorphism. By
uniqueness again, $F_{\eta\circ\theta}=F_\eta\circ F_\theta$. Since an endpoint occurrence in these
trees is a pair $(e,v)$ with $v$ an endpoint of $e$, its image is
\[
(e,v)\longmapsto(F_E(e),F_V(v)).
\]
The identity maps on edges and vertices therefore give the identity on endpoint occurrences.

With endpoint positions labelled $0$ and $1$, by \cref{thm:tree-correspondence} each tree map either
preserves both labels or interchanges them throughout the tree. A composite interchanges the labels
exactly when one of its two factors interchanges them and the other preserves them.

Return now to the original isomorphism $\theta$ and its normalization
$\theta'=\Inn(x^{-1})\circ\theta$ used above. If $\theta'(\cl H_1)=g\cl H_2g^{-1}$, then
\[
\begin{aligned}
\theta(\cl H_1)
&=x\theta'(\cl H_1)x^{-1}\\
&=(xg)\cl H_2(xg)^{-1}.
\end{aligned}
\]
Thus the conjugator for the original isomorphism is $xg$. Now let $F'$ be the tree isomorphism
constructed for $\theta'$, and define $F(y)=xF'(y)$ for every vertex or edge $y$. For
$h\in\wh\Pi_1$,
\[
\begin{aligned}
F(hy)
&=xF'(hy)\\
&=x\theta'(h)F'(y)\\
&=\theta(h)xF'(y)
=\theta(h)F(y).
\end{aligned}
\]
Therefore $F$ is $\theta$-equivariant. Its vertex map is given by
\[
F_V(q\cl H_1)
=x\theta'(q)g\cl H_2
=\theta(q)xg\cl H_2.
\]
Left multiplication by $x$ commutes with the endpoint maps, so $F$ has the same compatibility with
endpoint occurrences as $F'$. The uniqueness argument applies to $F$, so the identity and
composition conclusions also hold for the original marked isomorphisms.
\end{proof}


\section{Profinite cubical residual complex}\label{sec:core}
\subsection*{Filling core}
Let $(\alpha,\beta)$ be a nonseparating filling pair on $S$, and set $\Pi=\pi_1(S)$. We write
$T_\alpha,T_\beta$ for the original cut trees. In $\widetilde S\cong\HH^2$, the lifts of $\alpha$
and $\beta$ form locally finite families, and the geodesics in each family are pairwise disjoint.

\begin{proposition}\label{prop:original-core}
The geometric dual square complex $C=C(\alpha,\beta)$ is connected, locally finite, and regular, and
$\Pi$ acts freely, without inversions, and cocompactly. There is a $\Pi$-equivariant cellular
embedding
\[
j:C\hookrightarrow T_\alpha\times T_\beta
\]
whose squares are the products of crossing tree edges.
\end{proposition}
\begin{proof}
\smallskip\noindent\emph{Step 1.}
A vertex of $C(\alpha,\beta)$ is a complementary region $R$ of the full wall arrangement in
$\widetilde S\cong\HH^2$. The region $R$ lies in one component $U_\alpha(R)$ of the complement of
the $\alpha$-walls and one component $U_\beta(R)$ of the complement of the $\beta$-walls. These
components are the vertices of $T_\alpha$ and $T_\beta$, so set $j(R)=(U_\alpha(R),U_\beta(R))$. A
dual edge changes one coordinate, and a dual square at a crossing changes both. This defines a
cellular $\Pi$-equivariant map $j:C\to T_\alpha\times T_\beta$.

\smallskip\noindent\emph{Step 2.}
Each component of the complement of one wall family is convex, so the intersection of a chosen
$\alpha$-component and a chosen $\beta$-component is connected. Two complementary regions with the
same pair of coordinates must therefore coincide. Since an edge dual to a segment of an
$\alpha$-wall $L$ has coordinates consisting of the tree edge represented by $L$ and a component
$U_\beta$ of the complement of the $\beta$-walls, the segment is exactly $L\cap U_\beta$. Convexity
makes this intersection connected, so the coordinates determine at most one such edge. The same
argument applies with the two wall families interchanged. A product $e_\alpha\times e_\beta$
determines one lifted geodesic from each family; these two geodesics meet in at most one point, so
the product determines at most one square, and it does so exactly when they cross. Hence each edge
and square maps homeomorphically onto its product cell. Together with injectivity on vertices, this
proves cellular injectivity. Since each closed edge and each closed square maps homeomorphically
onto its product cell, its characteristic map is injective on the boundary, hence $C$ is regular.
For surfaces, this determines Guirardel's core \cite{Guirardel2005}.

\smallskip\noindent\emph{Step 3.}
Any two complementary regions can be joined by a compact path in $\HH^2$ transverse to the walls and
avoiding their crossing points. Local finiteness implies that the path crosses only finitely many
wall segments. The dual edges of these segments form an edge path in $C$, so $C$ is connected.
Cutting the compact surface along the finite filling graph $\alpha\cup\beta$ produces finitely many
compact polygonal disks, each with finitely many sides. The map of each disk lifts to $\HH^2$ and
its interior maps onto a lifted complementary region, so the closure of that region is contained in
the compact image of the disk and has finitely many sides. So every lifted complementary region is
bounded and finitely sided, and every wall segment, lying in the boundary of an adjacent such
region, is also bounded. The quotient of $C$ is the finite square cellulation dual to the filling
graph. Finite sidedness implies local finiteness, and the finite quotient makes the action
cocompact. If a deck transformation stabilizes a vertex, edge, or square, it preserves respectively
a bounded complementary polygon, a bounded wall segment, or a crossing point in $\HH^2$. Such an
isometry is elliptic. But the surface group is torsion-free and acts freely on $\HH^2$, so the
stabilizer is trivial. Thus the action is free and without inversions.

Finally, give every tree edge length one. The coordinate maps of $j$ are continuous on every closed
cell of $C$, so $j$ is continuous for both the product metric topology and the product of the CW
topologies. Fix $x\in C$. Around each coordinate of $j(x)$ choose a small open interval contained in
its open edge, or, if the coordinate is a vertex, the open star truncated at distance $1/2$. Their
product $U$ is open for both topologies. Every product cell meeting $U$ contains $j(x)$ in its
closure. Since $j$ is injective and maps each closed cell onto the corresponding closed product
cell, we have
\[
j^{-1}(U)\subseteq K_x:=\bigcup_{\sigma:\,x\in\cl\sigma}\cl\sigma.
\]
Local finiteness makes $K_x$ a finite union of closed cells, hence compact, and the continuous
injection $j|_{K_x}$ is a homeomorphism onto its image. On $U\cap j(C)$, the inverse of $j$ is the
restriction of this continuous inverse, so $j^{-1}$ is continuous at $j(x)$. As $x$ was arbitrary,
$j$ is a topological embedding.
\end{proof}

We fix lifts $\widetilde\alpha,\widetilde\beta\subset\widetilde S$ stabilized by
$A=\langle a\rangle$ and $B=\langle b\rangle$, and define
\[
\Omega=\{g\in\Pi:\widetilde\alpha\cap g\widetilde\beta\neq\varnothing\}.
\]
Let $r=i(\alpha,\beta)$ be their geometric intersection number and choose representatives
$r_1,\ldots,r_r\in\Pi$ of the corresponding double cosets. Then
\begin{equation}\label{eq:crossing-double-cosets}
\Omega=\bigsqcup_{j=1}^{r}Ar_jB.
\end{equation}
Indeed, $A\backslash\Omega/B$ corresponds to the intersection points of the two closed curves. By
\cref{thm:surface-double-cosets}(i), each double coset is separable, and in fact
\[
\cl{ArB}=\cl A\,r\,\cl B
\]
for every $r\in\Pi$, since the map $\cl A\times\cl B\to\wh\Pi$, $(a,b)\mapsto arb$, has compact
image in which the image $ArB$ of the dense subset $A\times B$ is dense. The completed double cosets
$\cl A r_j\cl B$ are pairwise disjoint. To see this, suppose
$\cl A r_j\cl B\cap\cl A r_k\cl B\neq\varnothing$, and choose $\wh a_1,\wh a_2\in\cl A$ and
$\wh b_1,\wh b_2\in\cl B$ with $\wh a_1r_j\wh b_1=\wh a_2r_k\wh b_2$. Then
\[
r_k\in\cl A r_j\cl B=\cl{Ar_jB}.
\]
Because the element $r_k$ belongs to the discrete subgroup $\Pi$, and the double coset $Ar_jB$ is
closed in the profinite topology of $\Pi$ by \cref{thm:surface-double-cosets}(i), $r_k\in Ar_jB$, so
$Ar_kB=Ar_jB$. Thus distinct double cosets give disjoint sets.

For profinite edges $e=x\cl A$ of $\wh T_\alpha$ and $f=y\cl B$ of $\wh T_\beta$, the crossing
relation of \cref{def:core} is
\begin{equation}\label{eq:profinite-crossing}
e\cross f\quad\Longleftrightarrow\quad x^{-1}y\in\cl\Omega.
\end{equation}

\begin{terminology}[\cite{WiltonZalesskii2017}]
Let $c:S^1\to S$ be a closed curve and $p:S'\to S$ a covering map. An elevation of $c$ to $S'$ is
the map $Z\to S'$ induced by the second projection, where $Z$ is a connected component of
\[
S^1\times_S S'=\{(z,s')\in S^1\times S':c(z)=p(s')\}.
\]
The first projection $Z\to S^1$ is a connected covering; compare
\cite[Definition~2.3]{WiltonZalesskii2017}.
\end{terminology}

\begin{lemma}\label{lem:finite-cover-crossing}
Let $e=x\cl A$ and $f=y\cl B$. Then $e\cross f$ if and only if, for every open normal subgroup
$U\normalo\wh\Pi$, the elevations of $\alpha$ and $\beta$ determined by $e$ and $f$ in the regular
cover corresponding to $U\cap\Pi$ intersect.
\end{lemma}
\begin{proof}
Let $q_U:\wh\Pi\to\wh\Pi/U$ be the quotient map, set $K=U\cap\Pi$, and choose discrete
representatives $x_U,y_U\in\Pi$ of the images of $x,y$ in $\wh\Pi/U$. The corresponding elevations
intersect when there exist $u,v\in K$ such that
$ux_U\widetilde\alpha\cap vy_U\widetilde\beta\neq\varnothing$; applying $(ux_U)^{-1}$, this is
equivalent to $x_U^{-1}u^{-1}vy_U\in\Omega$, and hence implies $q_U(x^{-1}y)\in q_U(\Omega)$.
Conversely, if this condition in the finite quotient holds, choose $\omega\in\Omega$ with the same
image as $x_U^{-1}y_U$. Then $x_U\omega y_U^{-1}\in K$, which yields intersecting representatives.
Since
\[
\cl\Omega=\bigcap_{U\normalo\wh\Pi}q_U^{-1}(q_U(\Omega)),
\]
the condition for all $U$ is equivalent to \eqref{eq:profinite-crossing}.
\end{proof}

\subsection*{The cell-wise completion and its embedding}
Note that each cell space, as well as $I_{10}(C)$ and $I_{21}(C)$, is a finite disjoint union of
free $\Pi$-orbits. For $0\leq d\leq2$, let $r_d$ be the number of cell orbits in dimension $d$, and
choose a representative $\sigma_{d,j}$ of each orbit, $1\leq j\leq r_d$. Then
\[
C^d=\bigsqcup_{j=1}^{r_d}\Pi\sigma_{d,j},
\qquad
\wh C^d\cong\bigsqcup_{j=1}^{r_d}\wh\Pi\sigma_{d,j},
\]
and the same holds for $I_{10}(C)$ and $I_{21}(C)$; the four structure maps extend uniquely and
compatibly over the common cofinal inverse system.

\begin{lemma}\label{lem:completed-orbit}
Let $\sigma$ be a cell of $C$, and let $P,Q\leq\Pi$ be the stabilizers of the two coordinate cells
of $j(\sigma)$ in $T_\alpha,T_\beta$. The diagonal coordinate map
\[
\iota_\sigma:\Pi\longrightarrow\Pi/P\times\Pi/Q,
\qquad
\gamma\longmapsto(\gamma P,\gamma Q),
\]
extends to a topological embedding
\[
\wh\iota_\sigma:\wh\Pi\hookrightarrow
\wh\Pi/\cl P\times\wh\Pi/\cl Q.
\]
Distinct original diagonal cell orbits in the product have disjoint closures.
\end{lemma}
\begin{proof}
Under the coset identifications, the two coordinate cells of $j(\sigma)$ correspond to the identity
cosets $P$ and $Q$. By \cref{prop:original-core}, the map $j$ is $\Pi$-equivariant and injective on
cells. Consequently, for every $\gamma\in\Pi$,
\[
\begin{aligned}
\gamma\sigma=\sigma
&\quad\Longleftrightarrow\quad
\gamma j(\sigma)=j(\sigma)\\
&\quad\Longleftrightarrow\quad
(\gamma P,\gamma Q)=(P,Q)\\
&\quad\Longleftrightarrow\quad
\gamma\in P\cap Q.
\end{aligned}
\]
Hence $\Stab_\Pi(\sigma)=P\cap Q$. Since the action on the cells of $C$ is free by
\cref{prop:original-core},
\[
P\cap Q=\Stab_\Pi(\sigma)=\{1\}.
\]
In addition, the subgroups $P,Q$ are finitely generated, so by \cref{thm:surface-double-cosets}(ii)
we have
\[
\cl P\cap\cl Q=\cl{P\cap Q}=1.
\]
It follows that the completed diagonal map is injective, and a continuous injection from the compact
space $\wh\Pi$ to a Hausdorff space is a topological embedding.

A general pair has the form $(sP,rQ)$, and acting diagonally by $s^{-1}$ we get
\[
s^{-1}\cdot(sP,rQ)=(P,s^{-1}rQ).
\]
Hence every diagonal orbit contains a pair with first coordinate $P$, so we may use representatives
$(P,rQ)$ with $r\in\Pi$, and the orbit of $(P,rQ)$ depends only on the double coset $PrQ$. Suppose
the closures of the orbits represented by $r,s\in\Pi$ meet. Then for some $\gamma,\delta\in\wh\Pi$
we have
\[
(\gamma\cl P,\gamma r\cl Q)=(\delta\cl P,\delta s\cl Q).
\]
Put $p=\delta^{-1}\gamma$. Then the first coordinate shows $p\in\cl P$, and the second shows
$s^{-1}pr\in\cl Q$, so $s\in\cl P\,r\,\cl Q$. Since the map $\cl P\times\cl Q\to\wh\Pi$,
$(p,q)\mapsto prq$, has compact image and the image of the dense subset $P\times Q$ is $PrQ$, we
have $\cl P\,r\,\cl Q=\cl{PrQ}$. Because $s$ is discrete and $PrQ$ is separable, $s\in PrQ$, so the
original orbits were equal.
\end{proof}

\begin{proposition}\label{prop:completion-embedding}
The $\Pi$-equivariant cellular embedding $j:C(\alpha,\beta)\hookrightarrow T_\alpha\times T_\beta$
of the dual square complex extends to an embedding of its profinite completion. The extended maps on
vertices, edges, and squares are continuous and commute with the structure maps. Its square space is
the set of products $e\times f$ with $e\cross f$. Every profinite edge and every profinite vertex of
$\wh C$ is a face of a profinite square.
\end{proposition}
\begin{proof}
\smallskip\noindent\emph{Step 1.}
For each original cell orbit, \cref{lem:completed-orbit} extends its coordinate map to a topological
embedding of a copy of $\wh\Pi$, and separates the closures of distinct cell orbits. We now treat
the occurrence spaces. Let $V_\xi,E_\xi$ denote the completed vertex and edge spaces of $T_\xi$ and
set $J_\xi=E_\xi\times\{0,1\}$, for $\xi\in\{\alpha,\beta\}$, as in \cref{def:core}. The coordinate
map sends endpoint occurrences into
\[
(J_\alpha\times V_\beta)\sqcup(V_\alpha\times J_\beta)
\]
and side occurrences into
\[
(J_\alpha\times E_\beta)\sqcup(E_\alpha\times J_\beta).
\]
Fix one component and one endpoint label in $\{0,1\}$. A labelled endpoint occurrence is a pair
$(e,r)\in J_\xi=E_\xi\times\{0,1\}$. Its underlying edge is $p_\xi(e,r)=e$, and it is attached to
the vertex
\[
q_\xi(e,r)=\partial_{\xi,r}(e).
\]
The label $r=0$ refers to the initial endpoint and $r=1$ to the terminal endpoint. The action of
$\wh\Pi$ on the cut tree preserves the endpoint labels, so $\gamma\cdot(e,r)=(\gamma e,r)$. Hence
\[
\gamma\cdot(e,r)=(e,r)
\quad\Longleftrightarrow\quad
\gamma e=e,
\]
and hence $\Stab(e,r)=\Stab(e)$. We now describe the occurrence maps using the product coordinates
of $j(C)$. All original coordinate pairs considered below represent cells of $C$ under the embedding
$j$.

First consider an edge whose image has coordinates $(e,v)$, where $e$ is an edge of $T_\alpha$ and
$v$ is a vertex of $T_\beta$. Its endpoint occurrence with label $r\in\{0,1\}$ has coordinates
$((e,r),v)$. The two structure maps are
\[
\begin{aligned}
p_{10}((e,r),v)&=(e,v),\\
q_{10}((e,r),v)&=(\partial_{\alpha,r}(e),v).
\end{aligned}
\]
For an edge with coordinates $(v,f)$, where $v$ is a vertex of $T_\alpha$ and $f$ is an edge of
$T_\beta$, we have
\[
\begin{aligned}
p_{10}(v,(f,r))&=(v,f),\\
q_{10}(v,(f,r))&=(v,\partial_{\beta,r}(f)).
\end{aligned}
\]

Now consider a square with coordinates $(e,f)$, representing the product of an edge $e$ of
$T_\alpha$ and an edge $f$ of $T_\beta$. The occurrence $((e,r),f)$ specifies the side obtained by
fixing the $T_\alpha$-coordinate at $\partial_{\alpha,r}(e)$. Therefore
\[
\begin{aligned}
p_{21}((e,r),f)&=(e,f),\\
q_{21}((e,r),f)&=(\partial_{\alpha,r}(e),f).
\end{aligned}
\]
Similarly, the occurrence $(e,(f,r))$ specifies the side obtained by fixing the $T_\beta$-coordinate
and in this case we have
\[
\begin{aligned}
p_{21}(e,(f,r))&=(e,f),\\
q_{21}(e,(f,r))&=(e,\partial_{\beta,r}(f)).
\end{aligned}
\]
The two choices of coordinate and the two choices of $r$ account for all four side positions. Now
fix one of the two subsets in the disjoint union defining the occurrence space, and fix a label
$r\in\{0,1\}$. With these choices, forgetting the label is a bijection between the occurrences
coming from $C$ and the images of its edges or squares. For example, $((e,r),v)\longmapsto(e,v)$ has
inverse $(e,v)\mapsto((e,r),v)$, and $((e,r),f)\longmapsto(e,f)$ has inverse
$(e,f)\mapsto((e,r),f)$. The correspondences with the label in the second coordinate have the same
property.

Since the action preserves the labels, $\gamma\cdot(e,r)=(\gamma e,r)$, all these bijections are
$\Pi$-equivariant, and they identify each occurrence orbit with the orbit of the edge or square to
which it belongs.

Since $\Stab_\Pi(e,r)=\Stab_\Pi(e)$, the two coordinate stabilizers of an endpoint occurrence
$((e,r),v)$ are $\Stab_\Pi(e)$ and $\Stab_\Pi(v)$, and those of a side occurrence $((e,r),f)$ are
$\Stab_\Pi(e)$ and $\Stab_\Pi(f)$. Hence they are the coordinate stabilizers $P,Q$ of the
corresponding edge or square, and the same holds when the label is in the second coordinate.

By \cref{lem:completed-orbit}, the coordinate map on each such cell orbit extends to an embedding of
a copy of $\wh\Pi$. Inserting the fixed label is continuous and has the continuous inverse that
forgets it, because the labels carry the discrete topology. Thus the coordinate map on each
occurrence orbit also extends to an embedding of a copy of $\wh\Pi$.

Hence within the chosen subset and with the label fixed, distinct occurrence orbits correspond to
distinct cell orbits, and by \cref{lem:completed-orbit}, their closures are also disjoint. Each
profinite occurrence space is the disjoint union of the two closed and open subsets above, within each of
which the parts labelled $0$ and $1$ are again disjoint and closed and open, and an orbit contained in one of
these subsets has its closure in the same subset. So orbit closures lying in different subsets or
having different labels cannot meet, and all distinct occurrence orbit closures are disjoint.

Since each of
\[
C^0,\quad C^1,\quad C^2,\quad
I_{10}(C),\quad I_{21}(C)
\]
is a finite disjoint union of free $\Pi$-orbits, its completion is a finite disjoint union of copies
of $\wh\Pi$. On the completed orbit $\wh\Pi\sigma$ of a chosen representative $\sigma$, we define
\[
\wh j_\sigma:\wh\Pi\sigma
\longrightarrow \wh\Pi j(\sigma),
\qquad
\wh j_\sigma(\gamma\sigma)=\gamma j(\sigma).
\]
The map $\wh j_\sigma$ is continuous and $\wh\Pi$-equivariant, and
\[
\wh j_\sigma|_{\Pi\sigma}=j|_{\Pi\sigma}.
\]
For any two chosen representatives $\sigma,\tau$ and any $\gamma,\delta\in\wh\Pi$, the orbit
argument above shows that
\[
\wh j_\sigma(\gamma\sigma)
=\wh j_\tau(\delta\tau)
\quad\text{if and only if}\quad
\sigma=\tau\ \text{and}\ \gamma=\delta,
\]
so each $\wh j_\sigma$ is injective, and the images of distinct completed orbits are disjoint. Each
completed component space is a finite disjoint union of closed and open copies of $\wh\Pi$, so the continuous
maps on these copies combine to give a continuous map on the entire component space. Each completed
component space is compact and each corresponding product component space is Hausdorff, so all five
component maps are topological embeddings with closed images. Combining the maps $\wh j_\sigma$ over
the chosen orbit representatives in each component space, we obtain the five extended coordinate
maps
\[
\wh j^d:\wh C^d\longrightarrow P^d
\qquad(0\leq d\leq2),
\]
\[
\wh j_{10}:\wh I_{10}(C)\longrightarrow I_{10}(P),
\qquad
\wh j_{21}:\wh I_{21}(C)\longrightarrow I_{21}(P),
\]
each of which agrees with $\wh j_\sigma$ on the completed orbit $\wh\Pi\sigma$ in its domain. Write
\[
\wh j=
\bigl(\wh j^0,\wh j^1,\wh j^2,
\wh j_{10},\wh j_{21}\bigr).
\]
We next show compatibility with the four structure maps. To distinguish the maps of the product
complex from those of $\wh C$, we write
\[
p_{10}^{P},\quad q_{10}^{P},\quad
p_{21}^{P},\quad q_{21}^{P}
\]
for the structure maps of the product complex. We claim that
\[
\begin{aligned}
\wh j^1\circ\wh p_{10}
&=p_{10}^{P}\circ\wh j_{10},\\
\wh j^0\circ\wh q_{10}
&=q_{10}^{P}\circ\wh j_{10},\\
\wh j^2\circ\wh p_{21}
&=p_{21}^{P}\circ\wh j_{21},\\
\wh j^1\circ\wh q_{21}
&=q_{21}^{P}\circ\wh j_{21}.
\end{aligned}
\]
For the first two identities, both sides are continuous maps defined on $\wh I_{10}(C)$ which agree
on the dense subset $I_{10}(C)$, because the original embedding commutes with the endpoint
occurrence maps; since their images lie in Hausdorff spaces, they agree on all of $\wh I_{10}(C)$.
The last two identities follow in the same way from equality on the dense subset
$I_{21}(C)\subseteq\wh I_{21}(C)$. The five component maps therefore define an embedding $\wh j$ of
profinite cubical residual complexes: every component map is a topological embedding, and all four
structure maps commute with these embeddings.

Finally, the image in each component is precisely the closure of the corresponding original image:
\[
\wh j^d(\wh C^d)
=\cl{j(C^d)}^{\,P^d}
\qquad(0\leq d\leq2),
\]
\[
\wh j_{10}(\wh I_{10}(C))
=\cl{j(I_{10}(C))}^{\,I_{10}(P)},
\]
\[
\wh j_{21}(\wh I_{21}(C))
=\cl{j(I_{21}(C))}^{\,I_{21}(P)},
\]
where $P$ stands for the product residual complex of $\wh T_\alpha$ and $\wh T_\beta$, and the
superscripts indicate the spaces in which the closures are taken. Each extended image is closed and
contains the original image, hence contains its closure. Conversely, continuity and density of the
original space in its completion imply that every point of the extended image belongs to that
closure. These component-wise equalities say that $\im(\wh j)=\cl{j(C)}$.

\smallskip\noindent\emph{Step 2.}
Since the original square orbits are indexed by the finitely many double cosets in
\eqref{eq:crossing-double-cosets}, the closure of the orbit represented by $r_j\in\Pi$ (see the
proof of \cref{lem:completed-orbit}) is
\[
\cl{\Pi\cdot(A,r_jB)}
=\wh\Pi\cdot(\cl A,r_j\cl B)
\subseteq E_\alpha\times E_\beta.
\]
Indeed, the continuous orbit map
\[
\wh\Pi\longrightarrow E_\alpha\times E_\beta,
\qquad
\gamma\longmapsto
(\gamma\cl A,\gamma r_j\cl B)
\]
has compact, hence closed, image. As $\Pi$ is dense in $\wh\Pi$, its image under this map is dense
in the full image. Hence, $\cl{\Pi\cdot(A,r_jB)}=\wh\Pi\cdot(\cl A,r_j\cl B)$. Moreover, for
$x,y\in\wh\Pi$,
\[
(x\cl A,y\cl B)
\in\wh\Pi\cdot(\cl A,r_j\cl B)
\quad\Longleftrightarrow\quad
x^{-1}y\in\cl A\,r_j\,\cl B,
\]
by the coordinate comparison in the proof of \cref{lem:completed-orbit}: if
$(x\cl A,y\cl B)=(\gamma\cl A,\gamma r_j\cl B)$ then $\gamma=xa$ and $y=\gamma r_jb$ with
$a\in\cl A$, $b\in\cl B$, so $x^{-1}y=ar_jb$, and conversely $x^{-1}y=ar_jb$ gives the equality with
$\gamma=xa$.

These orbit closures are pairwise disjoint by \cref{lem:completed-orbit}, and together they form the
image of the completed square space:
\[
\wh j^2(\wh C^2)
=\bigsqcup_{j=1}^{r}
\wh\Pi\cdot(\cl A,r_j\cl B).
\]
Since there are finitely many double cosets,
$\cl\Omega=\bigcup_{j=1}^{r}\cl{Ar_jB}=\bigcup_{j=1}^{r}\cl A\,r_j\,\cl B$, and so for $e=x\cl A$
and $f=y\cl B$ we have
\[
\begin{aligned}
(e,f)\in\wh j^2(\wh C^2)
&\quad\Longleftrightarrow\quad
x^{-1}y\in
\bigcup_{j=1}^{r}\cl A\,r_j\,\cl B\\
&\quad\Longleftrightarrow\quad
x^{-1}y\in\cl\Omega\\
&\quad\Longleftrightarrow\quad
e\cross f,
\end{aligned}
\]
where the last equivalence is \eqref{eq:profinite-crossing}.
\smallskip\noindent\emph{Step 3.}
Every original edge is a side of an original square, so the map
\[
q_{21}:I_{21}(C)\longrightarrow C^1
\]
is surjective. Its continuous extension
\[
\wh q_{21}:\wh I_{21}(C)\longrightarrow \wh C^1
\]
has compact, hence closed, image containing the dense original edge set $C^1$, and is therefore
surjective. Thus every completed edge is a side of a completed square. Because every original vertex
is an endpoint of an original edge, the map $q_{10}:I_{10}(C)\longrightarrow C^0$ is surjective. Its
continuous extension $\wh q_{10}:\wh I_{10}(C)\longrightarrow \wh C^0$ is likewise surjective by the
same argument using compactness of the image and density. So every completed vertex is an endpoint
of a completed edge, and therefore every completed edge and vertex is a face of a completed square,
using only the two occurrence spaces defined in \cref{def:cubical-occurrence}.
\end{proof}

\subsection*{Homological crossing and independent unit powers}
In this section we use the Boggi--Zalesskii criterion for finite covers; the realization argument
only requires the vanishing of algebraic intersection numbers. Most of the results and their proofs
are due to Boggi and Zalesskii \cite{BoggiZalesskii2019Magnus,BoggiZalesskii2017}. We state them in
the form needed for our constructions and give the proofs.

For arbitrary nontrivial $c,d\in\Pi$, the notation $i([c],[d])=0$ means that their
free homotopy classes admit representatives with disjoint images. The two representatives are
chosen independently, including when the classes coincide. This vanishing condition is unchanged
when either element is replaced by a nonzero integral power.
\begin{theorem}\label{thm:homological-intersection}
Let $S$ be a closed orientable surface of genus at least two and $\Pi=\pi_1(S)$. For $L\normalf\Pi$,
let $S_L$ be the associated cover space and let $V_L(c)\leq H_1(S_L,\Z)$ be the submodule generated
by the homology classes of all connected elevations of the closed curve represented by
$c\in\Pi\setminus\{1\}$. For $c,d\in\Pi\setminus\{1\}$,
\[
i([c],[d])=0
\quad\Longleftrightarrow\quad
V_L(c)\perp V_L(d)
\text{ for every }L\text{ in a cofinal normal family},
\]
where $\perp$ refers to the algebraic intersection form of the covering surface.
\end{theorem}
\begin{proof}
This is essentially the result of \cite{BoggiZalesskii2017}, with the class of all finite groups and
with no punctures. Their result applies to any cofinal family of normal finite index subgroups. In
particular, the standard characteristic family is admissible. Since algebraic intersection numbers
are integers, the criterion may equally be applied after extension of scalars to $\Q_p$, which is
how it is used below.
\end{proof}

For a finite index normal subgroup $L\normalf\Pi$, we identify $\wh L$ with its open closure in
$\wh\Pi$. For a prime $p$, we write $\Z_p$ for the ring of $p$-adic integers and $\Q_p$ for its
field of fractions. We set
\[
H_1(L,\Z_p)=H_1(L,\Z)\otimes_{\Z}\Z_p,
\qquad
H_1(L,\Q_p)=H_1(L,\Z)\otimes_{\Z}\Q_p.
\]
Because $L$ is finitely generated and its abelianization is free abelian, the maximal pro-$p$
quotient of the topological abelianization of $\wh L$ is $H_1(L,\Z_p)$. We denote the resulting
continuous map by
\[
\operatorname{ab}_{L,p}:\wh L\longrightarrow H_1(L,\Z_p).
\]
For $x\in\wh\Pi$, we let $n_L(x)$ denote the order of the image of $x$ in the finite quotient
$\wh\Pi/\wh L\cong\Pi/L$. Then $x^{n_L(x)}\in\wh L$. The next definition is due to Boggi and
Zalesskii \cite{BoggiZalesskii2019Magnus,BoggiZalesskii2017}.

\begin{terminology}
We define
\begin{equation}\label{eq:elevation-subspace}
W_{L,p}(x)
=\operatorname{span}_{\Q_p}
\left\{
\operatorname{ab}_{L,p}\bigl(hx^{n_L(x)}h^{-1}\bigr):
h\in\wh\Pi
\right\}
\subseteq H_1(L,\Q_p).
\end{equation}
\end{terminology}
The following elementary properties follow from the definition of $W_{L,p}(x)$. Part~\textup{(i)}
identifies it with the elevation module of \cite{BoggiZalesskii2017} after extension of scalars. The
corresponding construction for profinite simple closed curves appears in
\cite{BoggiZalesskii2019Magnus}. We give the proof for arbitrary profinite elements.
\begin{lemma}\label{lem:elevation-subspaces}
Let $L\normalf\Pi$ and let $p$ be a prime.
\begin{enumerate}[label=\textup{(\roman*)}]
\item If $x\in\Pi\setminus\{1\}$, then
\[
W_{L,p}(x)=V_L(x)\otimes_{\Z}\Q_p.
\]
\item For every $x,z\in\wh\Pi$ and every $\lambda\in\hZ^\times$,
\[
W_{L,p}(zxz^{-1})=W_{L,p}(x),
\qquad
W_{L,p}(x^\lambda)=W_{L,p}(x).
\]
\end{enumerate}
\end{lemma}
\begin{proof}
Conjugation by an element of $\wh L$ acts trivially after applying $\operatorname{ab}_{L,p}$, so the
classes in \eqref{eq:elevation-subspace} depend only on the image of $h$ in the finite group
$\wh\Pi/\wh L$, so the span is generated by a finite set.

Assume first that $x\in\Pi$. Since a connected elevation of the oriented closed curve represented by
$x$ to the normal cover $S_L\to S$ is determined by a double coset
$Lg\langle x\rangle\in L\backslash\Pi/\langle x\rangle$, its covering degree over the original curve
is the order of the image of $x$ in $\Pi/L$, namely $n_L(x)$. If a lift of the initial point is
chosen in the coset $Lg$ then the corresponding oriented closed elevation represents
$gx^{n_L(x)}g^{-1}\in L$. Conversely, every connected elevation is obtained in this way. So the
integral homology classes of the discrete elements occurring in \eqref{eq:elevation-subspace} are
precisely the classes of all connected elevations. Their $\Z$-span is $V_L(x)$, and extension of
scalars proves \textup{(i)}.

For the first equality in \textup{(ii)}, write $y=zxz^{-1}$. Then $n_L(y)=n_L(x)$ because $L$ is
normal, and
\[
h y^{n_L(y)}h^{-1}
=(hz)x^{n_L(x)}(hz)^{-1}.
\]
As $h$ ranges over $\wh\Pi$, so does $hz$, proving conjugacy invariance.

As for unit powers, a unit of $\hZ$ acts bijectively on every finite cyclic group, so $x$ and
$x^\lambda$ have the same order in $\Pi/L$ and $n_L(x^\lambda)=n_L(x)$. Inside the procyclic closure
of $\langle x\rangle$ we therefore have
\[
(x^\lambda)^{n_L(x)}=(x^{n_L(x)})^\lambda.
\]
Finally, the continuous homomorphism $\operatorname{ab}_{L,p}$ sends the right side to
$\lambda_p\operatorname{ab}_{L,p}(x^{n_L(x)})$, where $\lambda_p\in\Z_p^\times$ is the $p$-component
of $\lambda$, and multiplication by $\lambda_p$ is an invertible scalar on $H_1(L,\Q_p)$, so the
span is unchanged. Hence \textup{(ii)} follows.
\end{proof}
The next lemma is established for automorphisms in the proof of \cite{BoggiZalesskii2017}. We give a
cohomological proof in the form needed for isomorphisms between two profinite surface groups.
\begin{lemma}\label{lem:surface-similitude}
Let $L_i$ be closed orientable surface groups and let
\[
\vartheta:\wh L_1\xrightarrow{\cong}\wh L_2.
\]
For every prime $p$, the map induced by continuous abelianization,
\[
\vartheta_*:H_1(L_1,\Z_p)\xrightarrow{\cong}H_1(L_2,\Z_p),
\]
is a symplectic similitude for the alternating algebraic intersection forms $\langle\ ,\ \rangle_i$
on $H_1(L_i,\Z_p)$. Thus there is $\varepsilon_p\in\Z_p^\times$ such that
\begin{equation}\label{eq:surface-similitude}
\langle\vartheta_*u,\vartheta_*v\rangle_2
=\varepsilon_p\langle u,v\rangle_1
\qquad(u,v\in H_1(L_1,\Z_p)).
\end{equation}
\end{lemma}
\begin{proof}
If one surface has genus zero, its group and its profinite completion are trivial. The other surface
also has genus zero, since a positive-genus surface group has a nontrivial finite abelian quotient.
Both first homology groups vanish, and we may take $\varepsilon_p=1$. We therefore assume that both
surfaces have positive genus. We fix $r\geq1$ and set $R_r=\Z/p^r\Z$. Since surface groups are good
\cite{AschenbrennerFriedlWilton2015}, restriction from continuous cochains identifies the pullback
induced by $\vartheta$ with an isomorphism of graded rings
\[
\vartheta_r^*:H^*(L_2,R_r)\xrightarrow{\cong}H^*(L_1,R_r).
\]
Let $S_i$ be a closed oriented surface with $\pi_1(S_i)=L_i$, and let
$\omega_i\in H^2(S_i,R_r)=H^2(L_i,R_r)$ be the reductions of integral orientation classes normalized
by $\langle\omega_i,[S_i]\rangle=1$, where $[S_i]\in H_2(S_i,R_r)$ is the fundamental class. Since
the groups in degree two are free of rank one over $R_r$, there is a unique $\delta_r\in R_r^\times$
such that
\[
\vartheta_r^*(\omega_2)=\delta_r\omega_1.
\]
For $a,b\in H^1(S_2,R_r)$, multiplicativity and evaluation on the fundamental classes give
\[
\bigl\langle\vartheta_r^*(a)\smile\vartheta_r^*(b),[S_1]\bigr\rangle
=\delta_r\langle a\smile b,[S_2]\rangle.
\]
The first cohomology groups have the same rank, so the surfaces have a common genus $g$. Now choose
integral symplectic bases for their first homology groups and reduce these bases modulo $p^r$. Then
in these bases the intersection matrix is
\[
J=\begin{pmatrix}0&I_g\\-I_g&0\end{pmatrix},
\qquad J^{-1}=-J.
\]
The evaluation-dual cohomology bases have cup-product matrix $J$ as well. Let $M_r$ be the matrix of
the homology map dual to $\vartheta_r^*$, with columns giving the images of the chosen basis of
$H_1(S_1,R_r)$ so the pullback matrix is $M_r^{\mathsf T}$, and the evaluated cup-product identity
reads $M_r J M_r^{\mathsf T}=\delta_r J$. Since both $M_r$ and $\delta_r$ are invertible, taking
inverses and using $J^{-1}=-J$ we get
\[
M_r^{-\mathsf T}J M_r^{-1}=\delta_r^{-1}J,
\qquad
M_r^{\mathsf T}J M_r=\delta_r J.
\]
Thus the homology map has the same multiplier $\delta_r$, also for $p=2$, since nothing here
requires division by $2$. The coefficient reductions $R_{r+1}\to R_r$ commute with $\vartheta^*$ and
preserve the chosen integral orientation classes, so $\delta_{r+1}\bmod p^r=\delta_r$. We set
$\varepsilon_p=\varprojlim_r\delta_r\in\Z_p^\times$. Passing to the inverse limit in the last matrix
identity yields \eqref{eq:surface-similitude}. Finally, naturality of the evaluation pairing
identifies the dual homology map with the map induced by $\vartheta$ on maximal pro-$p$ abelian
quotients, hence with the map $\vartheta_*$ above.
\end{proof}

\begin{proposition}\label{prop:unit-power}
Let $\Pi_i$ be closed orientable surface groups of genus at least two. Suppose
$\theta:\wh\Pi_1\to\wh\Pi_2$ is an isomorphism and
\[
\theta(c_1)=u c_2^\lambda u^{-1},
\qquad
\theta(d_1)=v d_2^\nu v^{-1},
\]
where $c_i,d_i\in\Pi_i\setminus\{1\}$, $u,v\in\wh\Pi_2$, and $\lambda,\nu\in\hZ^\times$. Then
\[
i([c_1],[d_1])=0
\quad\Longleftrightarrow\quad
i([c_2],[d_2])=0.
\]
\end{proposition}
\begin{proof}
Fix a prime $p$ and an integer $j\geq1$. Set $L_i=K_j(\Pi_i)$. By \cref{lem:char-functor},
$\theta(\wh L_1)=\wh L_2$, and by \cref{lem:surface-similitude} its restriction induces an
isomorphism
\[
\theta_{L,*}:H_1(L_1,\Q_p)\xrightarrow{\cong}H_1(L_2,\Q_p).
\]
Since $\theta$ identifies the finite quotients $\Pi_i/L_i$, matching elements have the same order
there, hence by equivariance of continuous abelianization
\[
\theta_{L,*}W_{L_1,p}(c_1)=W_{L_2,p}(c_2),
\qquad
\theta_{L,*}W_{L_1,p}(d_1)=W_{L_2,p}(d_2),
\]
the conjugators being irrelevant by conjugacy invariance and the unit powers by
\cref{lem:elevation-subspaces}(ii).

By this isomorphism and \cref{lem:elevation-subspaces}(i), we obtain, at every standard
characteristic level,
\[
V_{L_1}(c_1)\perp V_{L_1}(d_1)
\quad\Longleftrightarrow\quad
V_{L_2}(c_2)\perp V_{L_2}(d_2),
\]
where one may test the equality after tensoring with $\Q_p$ because algebraic intersection is
integer valued. Since the standard families of characteristic subgroups are cofinal, applying
\cref{thm:homological-intersection} on both sides proves the claimed equivalence.
\end{proof}

\begin{proposition}\label{prop:crossing-invariance}
For $i=1,2$, let $(\alpha_i,\beta_i)$ be a nonseparating filling pair, and let
$A_i=\langle a_i\rangle$ and $B_i=\langle b_i\rangle$ be the cyclic subgroups represented by
$\alpha_i$ and $\beta_i$, respectively. Suppose that
\[
\theta:\wh\Pi_1\xrightarrow{\cong}\wh\Pi_2
\]
is an isomorphism for which there exist $u,v\in\wh\Pi_2$ satisfying
\[
\theta(\cl A_1)=u\cl A_2u^{-1},
\qquad
\theta(\cl B_1)=v\cl B_2v^{-1},
\]
with closures taken in the respective profinite surface groups. Let
\[
F_\alpha:\wh T_{\alpha_1}\to\wh T_{\alpha_2},
\qquad
F_\beta:\wh T_{\beta_1}\to\wh T_{\beta_2}
\]
be the equivariant tree isomorphisms of \cref{thm:tree-correspondence}. Then
\[
e\cross f
\quad\Longleftrightarrow\quad
F_\alpha(e)\cross F_\beta(f)
\]
for every pair of profinite edges $e,f$.
\end{proposition}
\begin{proof}
Write $A_i=\langle a_i\rangle$ and $B_i=\langle b_i\rangle$, where $a_i$ and $b_i$ represent one
oriented traversal of $\alpha_i$ and $\beta_i$, respectively.

Fix $m\geq1$ and set $K_i=K_m(\Pi_i)$, $U_i=\wh K_i$. By \cref{lem:char-functor}, $\theta(U_1)=U_2$.
Let $q_i:\wh\Pi_i\longrightarrow\wh\Pi_i/U_i\cong\Pi_i/K_i$ be the quotient map.

We first describe the $U_i$-orbits in the edge spaces. For $D=A_i$ or $D=B_i$, consider
\[
\rho_{i,D}:\wh\Pi_i/\cl D
\longrightarrow(\Pi_i/K_i)/q_i(D),
\qquad
x\cl D\longmapsto q_i(x)q_i(D).
\]
The fibers of $\rho_{i,D}$ are the $U_i$-orbits: multiplication on the left by an element of $U_i$
does not change the image. Conversely, if $q_i(y)q_i(D)=q_i(x)q_i(D)$, there is $d\in D$ such that
$q_i(y)=q_i(xd)$. So $u=yd^{-1}x^{-1}\in U_i$ and $y\cl D=ux\cl D$. Since $\rho_{i,D}$ is continuous
and its image is a finite discrete space, every such orbit is closed and open.

Every $U_i$-orbit contains an original edge, because the map $\Pi_i\to\Pi_i/K_i$ is surjective.
Moreover, if $x,y\in\Pi_i$ then
\[
u=yd^{-1}x^{-1}\in U_i\cap\Pi_i=K_i.
\]
Consequently, the original edges in one $U_i$-orbit form exactly one $K_i$-orbit. It follows that if
$D=A_i$, then these orbits correspond to the connected elevations of $\alpha_i$ in $S_{K_i}$, and if
$D=B_i$, then they correspond to the connected elevations of $\beta_i$.

Set
\[
e_1=e,\qquad f_1=f,\qquad
e_2=F_\alpha(e),\qquad f_2=F_\beta(f),
\]
and choose $g_i,h_i\in\Pi_i$ such that
\[
e_i\in U_i\cdot(g_i\cl A_i),
\qquad
f_i\in U_i\cdot(h_i\cl B_i).
\]
Define
\[
c_i=g_i a_i^{\,n_{K_i}(a_i)}g_i^{-1},
\qquad
d_i=h_i b_i^{\,n_{K_i}(b_i)}h_i^{-1}.
\]
Then normality of $K_i$ implies that $c_i,d_i\in K_i$ and that
\[
K_i\cap g_iA_i g_i^{-1}=\langle c_i\rangle,
\qquad
K_i\cap h_iB_i h_i^{-1}=\langle d_i\rangle.
\]
The elements $c_i$ and $d_i$ represent one oriented traversal of the connected elevations determined
by the two chosen edge orbits.

For any subgroup $H\leq\Pi_i$, openness of $U_i$ implies $U_i\cap\cl H=\cl{K_i\cap H}$. Indeed, $H$
is dense in $\cl H$, and its intersection with the open subset $U_i\cap\cl H$ is dense in that
subset, whence
\[
\begin{aligned}
\Stab_{U_i}(g_i\cl A_i)
&=U_i\cap g_i\cl A_i g_i^{-1}\\
&=\cl{K_i\cap g_iA_i g_i^{-1}}
=\cl{\langle c_i\rangle},\\
\Stab_{U_i}(h_i\cl B_i)
&=U_i\cap h_i\cl B_i h_i^{-1}\\
&=\cl{K_i\cap h_iB_i h_i^{-1}}
=\cl{\langle d_i\rangle},
\end{aligned}
\]
and these cyclic closures are isomorphic to $\hZ$.

As $\theta(U_1)=U_2$, equivariance implies that $F_\alpha$ and $F_\beta$ map $U_1$-orbits onto
$U_2$-orbits. Hence there exist $u',v'\in U_2$ such that
\[
F_\alpha(g_1\cl A_1)=u' g_2\cl A_2,
\qquad
F_\beta(h_1\cl B_1)=v' h_2\cl B_2.
\]
Since equivariance and bijectivity of the tree maps give equality of the corresponding stabilizers,
we have
\[
\theta\bigl(\cl{\langle c_1\rangle}\bigr)
=u'\cl{\langle c_2\rangle}u'^{-1},
\qquad
\theta\bigl(\cl{\langle d_1\rangle}\bigr)
=v'\cl{\langle d_2\rangle}v'^{-1}.
\]
An isomorphism between copies of $\hZ$ sends a topological generator to a unit power of a
topological generator, so
\[
\theta(c_1)=u' c_2^\lambda u'^{-1},
\qquad
\theta(d_1)=v' d_2^\nu v'^{-1}
\]
for some $\lambda,\nu\in\hZ^\times$.

Since the groups $K_i$ are fundamental groups of closed orientable surfaces of genus at least two,
we may apply \cref{prop:unit-power} to
\[
\theta|_{U_1}:\wh K_1\xrightarrow{\cong}\wh K_2.
\]
Therefore
\[
i([c_1],[d_1])=0
\quad\Longleftrightarrow\quad
i([c_2],[d_2])=0,
\]
where the intersection numbers are taken in $S_{K_1}$ and $S_{K_2}$, respectively.

Use the fixed geodesic representatives of $\alpha_i,\beta_i$ and the lifted hyperbolic metric on
$S_{K_i}$. Their connected elevations are simple closed geodesics. An elevation of $\alpha_i$ cannot
coincide with an elevation of $\beta_i$, since projecting their common image would give
$\alpha_i=\beta_i$, contrary to the filling-pair condition. Thus the two elevations determined by
$e_1,f_1$ intersect if and only if the two elevations determined by $e_2,f_2$ intersect.

Since this equivalence holds for every $m$ and the subgroups $U_i=\wh{K_m(\Pi_i)}$ form a cofinal
family of open normal subgroups, testing intersection at these levels is sufficient: given any open
normal subgroup, choose a characteristic level contained in it. Intersecting elevations in the finer
cover project to intersecting elevations in the coarser cover. Conversely, intersection at every
open normal level includes intersection at every characteristic level. Now by
\cref{lem:finite-cover-crossing},
\[
e\cross f
\quad\Longleftrightarrow\quad
F_\alpha(e)\cross F_\beta(f).
\]
\end{proof}
\begin{corollary}\label{cor:completed-core}
Under the conditions of \cref{prop:crossing-invariance}, put $C_i=C(\alpha_i,\beta_i)$. The product
map $F_\alpha\times F_\beta$, together with its induced maps on endpoint and side occurrences,
restricts to a $\theta$-equivariant isomorphism of profinite cubical residual complexes
\[
F:\wh{\Icell}(C_1)\xrightarrow{\cong}\wh{\Icell}(C_2).
\]
\end{corollary}
\begin{proof}
By \cref{prop:completion-embedding}, the completed coordinate embedding identifies
$\wh{\Icell}(C_i)$ with a subdiagram of
\[
\Icell(\wh T_{\alpha_i}\times\wh T_{\beta_i})
\]
which is closed in each of the five profinite component spaces. The subdiagram consists of the
crossing squares and all their side edges, vertices, endpoint occurrences and side occurrences, with
the restricted structure maps. By \cref{prop:crossing-invariance}, $F_\alpha\times F_\beta$ maps the
completed square space of $C_1$ bijectively onto that of $C_2$.

Each completed square has four side occurrences in $\wh I_{21}(C_i)$, and by the definition of face
positions in \cref{def:cubical-occurrence}, the product tree map takes these four positions
bijectively to the side positions of the corresponding square in $\wh C_2$. It therefore induces a
bijection on the completed side occurrence spaces commuting with $p_{21}$ and $q_{21}$. By
\cref{prop:completion-embedding}, the completed map $\wh q_{21}$ is surjective. Then commutation
with $q_{21}$ yields surjectivity on the completed edge spaces, while injectivity follows from the
fact that $F_\alpha\times F_\beta$ is a homeomorphism of the products of completed cut trees. So the
induced product map is a bijection from the completed edge space of $\wh C_1$ onto that of
$\wh C_2$.

Denote by $F_0$ and $F_1$ the maps induced by $F_\alpha\times F_\beta$ on the vertex and edge spaces
of the products $\wh T_{\alpha_i}\times\wh T_{\beta_i}$, respectively. The previous paragraph shows
that $F_1$ restricts to a bijection between the completed edge spaces of $C_1$ and $C_2$. Each
completed edge has two endpoint occurrences in $\wh I_{10}(C_i)$. The product tree map induces a
bijection on endpoint occurrence spaces: the occurrence of an endpoint $v$ of an edge $e$ is mapped
to the occurrence of $F_0(v)$ as an endpoint of $F_1(e)$. This induced map commutes with $p_{10}$
and $q_{10}$. By \cref{prop:completion-embedding}, the completed map $\wh q_{10}$ is surjective.
Hence commutation with $q_{10}$ yields surjectivity on the completed vertex spaces, while
injectivity again follows from the product homeomorphism, so the induced product map is a bijection
from the completed vertex space of $\wh C_1$ onto that of $\wh C_2$. Consequently the restriction is
bijective, in the order
\[
\wh C^2\longleftarrow \wh I_{21}(C)
\longrightarrow \wh C^1\longleftarrow \wh I_{10}(C)
\longrightarrow \wh C^0,
\]
and commutes with all four structure maps. The inverse product map induces continuous inverse maps
on all five component spaces, commuting with the four structure maps. So the restriction is the
desired $\theta$-equivariant isomorphism of profinite cubical residual complexes.
\end{proof}

\subsection*{Realization of the original G-action}
\begin{theorem}\label{thm:discrete-star}
For $i=1,2$, let $\Gamma_i$ be a finitely generated residually finite group acting freely,
cellularly, and cocompactly on a connected locally finite regular square complex $X_i$. For each
$i$, use one cofinal family of finite index normal subgroups of $\Gamma_i$ for all five component
spaces in \cref{def:cubical-occurrence,def:cubical-action-completion}. Write $\wh X_i$ for the
resulting profinite cubical residual complex. Suppose that
\[
\eta:\wh\Gamma_1\xrightarrow{\cong}\wh\Gamma_2
\]
is a continuous group isomorphism and that $F:\wh X_1\to\wh X_2$ is an $\eta$-equivariant
isomorphism of residual complexes, so $F$ is a tuple of five homeomorphisms commuting with all
four structure maps. Then there are $u\in\wh\Gamma_2$ and an isomorphism $h:\Gamma_1\to\Gamma_2$
such that
\[
\eta=\Inn(u)\circ\wh h,
\qquad
F(X_1)=uX_2.
\]
Equivalently, the normalized pair
\[
(\Inn(u^{-1})\circ\eta,\,u^{-1}F)
\]
maps $X_1$ onto $X_2$, where $(u^{-1}F)(x)=u^{-1}F(x)$. In particular, $\eta$ is discretely induced
up to profinite inner automorphism.
\end{theorem}
\begin{proof}
Write
\[
\eta_0:\wh\Gamma_1\xrightarrow{\cong}\wh\Gamma_2,
\qquad
F_0:\wh X_1\xrightarrow{\cong}\wh X_2
\]
for the originally given maps $\eta,F$.

Fix $i\in\{1,2\}$. We first show that every cell stabilizer in $\Gamma_i$ is trivial. Suppose that
$\gamma\in\Gamma_i$ stabilizes an open edge $e$ setwise. It then preserves $\cl e$. By regularity,
the characteristic map $\chi_e$ is a homeomorphism onto $\cl e$, so
$\chi_e^{-1}\circ\gamma\circ\chi_e$ is a continuous self-map of a closed interval; it has a fixed
point, whose image under $\chi_e$ is fixed by $\gamma$. Freeness of the action implies $\gamma=1$.
Similarly, if $\gamma$ stabilizes an open square $q$ setwise, the characteristic map $\chi_q$
identifies its restriction with the continuous self-map $\chi_q^{-1}\circ\gamma\circ\chi_q$ of the
closed square, and Brouwer's fixed point theorem produces a fixed point. Again, freeness implies
$\gamma=1$, and $\Stab_{\Gamma_i}(\sigma)=\{1\}$ for every cell $\sigma$.

By equivariance of the structure maps,
\[
\Stab_{\Gamma_i}(\iota)
\leq\Stab_{\Gamma_i}(p_{10}(\iota))=\{1\}
\qquad(\iota\in I_{10}(X_i)),
\]
and
\[
\Stab_{\Gamma_i}(\kappa)
\leq\Stab_{\Gamma_i}(p_{21}(\kappa))=\{1\}
\qquad(\kappa\in I_{21}(X_i)),
\]
hence all occurrence stabilizers are trivial as well.

Cocompactness and local finiteness give finitely many cell orbits: a compact subset whose translates
cover $X_i$ meets only finitely many cells, and every cell orbit meets that subset. There are also
finitely many occurrence orbits, because each edge has two endpoint occurrences and each square has
four side occurrences.

Since for every original cell or occurrence $\sigma$, the orbit map
\[
\Gamma_i\longrightarrow\Gamma_i\sigma,
\qquad
\gamma\longmapsto\gamma\sigma,
\]
is an equivariant bijection, for each finite index normal subgroup $N\normalf\Gamma_i$, it induces a
bijection
\[
\Gamma_i/N\xrightarrow{\cong}
N\backslash(\Gamma_i\sigma),
\qquad
\gamma N\longmapsto N(\gamma\sigma).
\]
These bijections commute with the quotient maps, so their inverse limit identifies the completed
orbit with a copy of $\wh\Gamma_i$, denoted $\wh\Gamma_i\sigma$. In particular, if $Z_i$ is one of
the five original component spaces and $z_1,\ldots,z_r$ are its orbit representatives, then
\[
Z_i=\bigsqcup_{j=1}^{r}\Gamma_i z_j,
\qquad
\wh Z_i=\bigsqcup_{j=1}^{r}\wh\Gamma_i z_j.
\]
Distinct original orbits remain distinct in every quotient by a subgroup of $\Gamma_i$. By residual
finiteness, each original orbit is identified with a subset of its completion, and
$(\wh\Gamma_i z_j)\cap Z_i=\Gamma_i z_j$. Thus the action on each completed orbit is free.

We next determine when a completed occurrence is original. Let $s$ be one of the four structure maps
\[
p_{10},\quad q_{10},\quad p_{21},\quad q_{21},
\]
and choose an original occurrence $\xi$ in its domain. On its completed orbit, equivariance says
that
\[
\wh s(g\xi)=g\,s(\xi)
\qquad(g\in\wh\Gamma_i).
\]
If this image is original, the orbit decomposition above shows $g\,s(\xi)=\gamma\,s(\xi)$ for some
$\gamma\in\Gamma_i$. Freeness of the completed cell orbit implies $g=\gamma$, and $g\xi$ is
original. Conversely, the image of an original occurrence under any structure map is original.

Applying this argument to every occurrence orbit proves
\[
\begin{aligned}
\wh p_{10}^{-1}(X_i^1)
&=I_{10}(X_i),&
\wh q_{10}^{-1}(X_i^0)
&=I_{10}(X_i),\\
\wh p_{21}^{-1}(X_i^2)
&=I_{21}(X_i),&
\wh q_{21}^{-1}(X_i^1)
&=I_{21}(X_i),
\end{aligned}
\]
where each preimage is taken in the corresponding completed occurrence space.

Now choose an original vertex $v_1\in X_1^0$. The point $F_0(v_1)$ belongs to $\wh\Gamma_2v_2$ for
some original vertex representative $v_2$. Let $t\in\wh\Gamma_2$ so that $F_0(v_1)=tv_2$, and set
$w=t^{-1}$. Define
\[
F'(y)=wF_0(y),
\qquad
\eta'=\Inn(w)\circ\eta_0.
\]
Since the group action commutes with the structure maps, $F'$ is again an isomorphism of profinite
residual complexes, $F'(v_1)=v_2$, and for every $\gamma\in\wh\Gamma_1$
\[
\begin{aligned}
F'(\gamma y)
&=w\eta_0(\gamma)F_0(y)\\
&=\bigl(w\eta_0(\gamma)w^{-1}\bigr)F'(y)\\
&=\eta'(\gamma)F'(y).
\end{aligned}
\]
Thus $F'$ is $\eta'$-equivariant.

Let $v$ be an original vertex whose image under $F'$ is original, and let $e$ be an original edge
with endpoint $v$. Choose the endpoint occurrence $\iota$ of $e$ at $v$. Since $F'$ commutes with
the structure maps we have
\[
\wh q_{10}(F'(\iota))=F'(v)\in X_2^0.
\]
Therefore
\[
F'(e)=\wh p_{10}(F'(\iota))\in X_2^1,
\]
and if $\iota'$ is the other endpoint occurrence of $e$, then
\[
\wh p_{10}(F'(\iota'))=F'(e)\in X_2^1.
\]
The preimage for $\wh p_{10}$ shows that $F'(\iota')$ is original, and then so is the image of the
other endpoint.

Since the $1$-skeleton of $X_1$ is connected, starting at $v_1$ and applying this argument along
finite edge paths shows that every original vertex has original image. Hence by applying it to every
original edge we get
\[
F'(X_1^0)\subseteq X_2^0,
\qquad
F'(X_1^1)\subseteq X_2^1.
\]
Now for every original endpoint occurrence $\iota$ we have
\[
\wh p_{10}(F'(\iota))
=F'(p_{10}(\iota))\in X_2^1,
\]
so the same preimage argument yields
\[
F'(I_{10}(X_1))\subseteq I_{10}(X_2).
\]

Let $q$ be an original square and choose one of its side occurrences $\kappa$. Since its attached
edge $q_{21}(\kappa)$ has original image,
\[
\wh q_{21}(F'(\kappa))
=F'(q_{21}(\kappa))\in X_2^1.
\]
The preimage for $\wh q_{21}$ then shows that $F'(\kappa)$ is original, and therefore
\[
F'(q)=\wh p_{21}(F'(\kappa))\in X_2^2.
\]
The argument applies to every original side occurrence. Consequently,
\[
F'(X_1^2)\subseteq X_2^2,
\qquad
F'(I_{21}(X_1))\subseteq I_{21}(X_2).
\]

Now the inverse $F'^{-1}$ is $\eta'^{-1}$-equivariant and satisfies $F'^{-1}(v_2)=v_1$, so repeating
the argument for this inverse we get the reverse inclusions on all five original component spaces.
So $F'(X_1)=X_2$ as original residual complexes.

Finally, for $\gamma\in\Gamma_1$, by equivariance
\[
\eta'(\gamma)v_2
=F'(\gamma v_1)\in X_2^0,
\]
and since
\[
(\wh\Gamma_2v_2)\cap X_2^0=\Gamma_2v_2,
\]
there is $\delta\in\Gamma_2$ such that $\eta'(\gamma)v_2=\delta v_2$, and by freeness of the
completed vertex orbit we have $\eta'(\gamma)=\delta\in\Gamma_2$. Therefore
$\eta'(\Gamma_1)\subseteq\Gamma_2$. The same argument for $F'^{-1}$ shows
$\eta'^{-1}(\Gamma_2)\subseteq\Gamma_1$, so $\eta'(\Gamma_1)=\Gamma_2$. Hence the restriction
$h=\eta'|_{\Gamma_1}$ is an isomorphism. Therefore, the continuous maps $\eta'$ and $\wh h$ agree on
the dense subgroup $\Gamma_1$, and hence $\eta'=\wh h$. Setting $u=w^{-1}=t$ we get
\[
\eta_0=\Inn(u)\circ\wh h,
\qquad
F_0(X_1)=uX_2,
\]
which proves the theorem.
\end{proof}
\begin{remark}\label{rem:translation-necessary}
The example after \cref{def:core} shows that the normalization in \cref{thm:discrete-star} cannot be
omitted: left multiplication by an element of $\wh\Gamma_2\setminus\Gamma_2$ moves every original
square out of the original subset.
\end{remark}

\begin{proof}[Proof of \Cref{thm:core-main}]
\smallskip\noindent\emph{Step 1.}
By \cref{prop:original-core}, the original dual complex $C_i$ is a connected locally finite regular
square complex on which $\Pi_i$ acts freely and cocompactly, and it embeds in
$T_{\alpha_i}\times T_{\beta_i}$. Taking the component-wise inverse limit over the standard
characteristic family $(K_m(\Pi_i))_{m\geq1}$ produces the profinite cubical residual complex
$\wh{\Icell}(C_i)$. On each orbit, \cref{lem:completed-orbit} proves that the diagonal coordinate
map extends to an embedding in the product of the completed coordinate orbits. Separability keeps
distinct original orbits from acquiring common limit points. Because there are finitely many cell
orbits in each dimension, the orbit-wise maps together give the unique continuous cellular embedding
of \cref{prop:completion-embedding}.

\smallskip\noindent\emph{Step 2.}
The square orbits of $C_i$ are indexed by the finitely many double cosets corresponding to
intersection points of $\alpha_i$ and $\beta_i$. For a representative $r$ of one of these double
cosets, the corresponding completed square orbit is $\wh\Pi_i\cdot(\cl A_i,r\cl B_i)$. Given
$x,y\in\wh\Pi_i$, we have $(x\cl A_i,y\cl B_i)\in\wh\Pi_i\cdot(\cl A_i,r\cl B_i)$ if and only if
$x^{-1}y\in\cl A_i\,r\,\cl B_i$. By \cref{prop:completion-embedding}, the finite union of these
completed orbits is the image of $\wh C_i^2$, namely
\[
\wh j_i(\wh C_i^2)
=\{(e,f)\in E_{\alpha_i}\times E_{\beta_i}:
e\cross f\}.
\]

Fix a characteristic level
\[
K_i=K_m(\Pi_i),
\qquad U_i=\wh K_i.
\]
By \cref{lem:finite-cover-crossing}, the condition $e\cross f$ is equivalent to intersection of the
corresponding connected elevations at every such level. An elevation of $\alpha_i$ to $S_{K_i}$
corresponds to a $U_i$-orbit of edges. Choose an original representative $g\cl A_i$ of this orbit,
with $g\in\Pi_i$; the stabilizer of this edge in $U_i$ is
\[
\begin{aligned}
\Stab_{U_i}(g\cl A_i)
&=U_i\cap g\cl A_i g^{-1}\\
&=\cl{K_i\cap gA_i g^{-1}}\\
&=\cl{
\langle g a_i^{n_{K_i}(a_i)}g^{-1}\rangle},
\end{aligned}
\]
where $n_{K_i}(a_i)$ is the least positive integer with $a_i^{n_{K_i}(a_i)}\in K_i$; the element
$g a_i^{n_{K_i}(a_i)}g^{-1}$ represents one traversal of the connected elevation. The corresponding
statement for $\beta_i$ is obtained by replacing $A_i,a_i$ with $B_i,b_i$.

By \cref{lem:char-functor}, $\theta(U_1)=U_2$. Choose original edge representatives in the
corresponding $U_i$-orbits, and denote the resulting elevation generators by $c_i,d_i\in K_i$.
Equivariance and bijectivity of the tree maps identify their completed edge stabilizers.
Consequently,
\[
\theta(c_1)=u c_2^\lambda u^{-1},
\qquad
\theta(d_1)=v d_2^\nu v^{-1}
\]
for some $u,v\in U_2$ and $\lambda,\nu\in\hZ^\times$. The exponents are units because each cyclic
closure is isomorphic to $\hZ$, and an isomorphism sends a topological generator to a topological
generator. The two conjugators, and the two unit exponents, may well be different.
\Cref{prop:unit-power}, which rests on the homological criterion of
\cref{thm:homological-intersection}, shows that these replacements preserve disjointness. Hence the
criterion, applied over a cofinal family of finite covers, preserves the closed crossing relation,
and \cref{prop:crossing-invariance,cor:completed-core} give the required equivariant isomorphism
$\wh C_1\cong\wh C_2$.
\end{proof}

\begin{proof}[Proof of \Cref{cor:two-curve}]
Set $C_i=C(\alpha_i,\beta_i)$. Applying \cref{thm:finite-cone-main} to the subgroups $A_1,A_2$ and
then to $B_1,B_2$, we obtain isomorphisms
\[
F_\alpha:\wh T_{\alpha_1}
\xrightarrow{\cong}\wh T_{\alpha_2},
\qquad
F_\beta:\wh T_{\beta_1}
\xrightarrow{\cong}\wh T_{\beta_2}.
\]
Both maps are equivariant for the same isomorphism $\theta$, so their product
$F_\alpha\times F_\beta$ is $\theta$-equivariant for the diagonal actions. By
\cref{prop:crossing-invariance}, $e\cross f$ if and only if $F_\alpha(e)\cross F_\beta(f)$. By
\cref{prop:completion-embedding}, the product map therefore maps the completed square space of $C_1$
onto that of $C_2$. By \cref{cor:completed-core}, it induces a $\theta$-equivariant isomorphism
$F:\wh C_1\xrightarrow{\cong}\wh C_2$ on all five component spaces. By \cref{prop:original-core},
each $C_i$ is a connected, locally finite, regular square complex on which the finitely generated
residually finite group $\Pi_i$ acts freely and cocompactly, and the completions are taken over the
cofinal characteristic families, so \cref{thm:discrete-star} applies to $\theta$ and $F$ and gives
$u\in\wh\Pi_2$ and $h:\Pi_1\xrightarrow{\cong}\Pi_2$ with $\theta=\Inn(u)\circ\wh h$, as claimed.
\end{proof}
\begin{proof}[Proof of \Cref{thm:no-exotic}]
Let $[\theta]\in\Out(\wh\Pi)$ have a representative preserving the two marked conjugacy classes. We
choose conjugators $x,y\in\wh\Pi$ so that the conditions of \cref{cor:two-curve} are written
explicitly as $\theta(\cl A)=x\cl A x^{-1}$ and $\theta(\cl B)=y\cl B y^{-1}$. By the corollary,
$\theta=\Inn(u)\circ\wh h$ for a discrete automorphism $h\in\Aut(\Pi)$ and $u\in\wh\Pi$.

Passing to outer automorphism groups removes $\Inn(u)$, so $[\theta]$ is the image of
$[h]\in\Out(\Pi)$. Only the profinite outer class and the two conjugacy classes are given, so the
automorphism $h$ may move the chosen representatives $A$ and $B$.
\end{proof}

\subsection*{Realization in every dimension from the completed full cellular residual complex}
We now extend the realization argument to complexes of arbitrary dimension.

\begin{proof}[Proof of \Cref{thm:all-dimensional-occurrence}]
The proof of \cref{thm:discrete-star} only uses the occurrence spaces of codimension one faces and
their two structure maps, and it goes through in every dimension; we indicate the changes. Since
$\Gamma_i$ acts freely and $X_i$ is regular, every cell stabilizer is trivial, by Brouwer's fixed
point theorem applied to the characteristic map of the closed cell as in the proof of
\cref{thm:discrete-star}, and hence every occurrence stabilizer is trivial as well. As in that
proof, there are finitely many cell and occurrence orbits, and each completed orbit is a free
$\wh\Gamma_i$-orbit meeting the original space in the original orbit. Consequently, for every $k$,
an occurrence in $\wh I_{k,k-1}(X_i)$ is original as soon as its image under $\wh p_{k,k-1}$ or
under $\wh q_{k,k-1}$ is original.

After a left translation we may assume that $F$ sends some original vertex $v_1$ of $X_1$ to an
original vertex of $X_2$. Since the $1$-skeleton is connected, ordinariness propagates along edge
paths through the endpoint occurrences, exactly as in the square case, so all original vertices,
edges and endpoint occurrences have original images. Suppose that all original cells of dimension
less than $k$ have original images, and let $\sigma$ be an original $k$-cell. Its boundary sphere
contains a $(k-1)$-cell $\tau$, giving an occurrence $(\sigma,\tau,\iota)$ whose image under
$\wh q_{k,k-1}$ is the image of $\tau$, which is original. Therefore the occurrence has original
image, and so does $\sigma=p_{k,k-1}(\sigma,\tau,\iota)$. By induction on $k$, all original cells
and occurrences of $X_1$ have original images, and the same argument applied to $F^{-1}$ yields the
reverse inclusions. The identification of the dense subgroups through the action on $v_1$ is exactly
as before.
\end{proof}


\section{Graph residual complexes}
\label{sec:graph-occurrence}
In this section we construct graph residual complexes and prove their realization theorem, which
is used in \cref{sec:nonlattice}.
\subsection*{Characteristic quotients and directed edge residual complexes}

We keep the notation $K_m(G)=\bigcap_{[G:L]\leq m}L$ from \cref{def:characteristic-cover}, together
with $\wh G$, $\cl H$, $\Inn(u)$, and $\Stab_G(x)$. We write $\mathfrak S_\Lambda$ for the group of
all permutations of a finite set $\Lambda$. For an action on a graph or its completion, $L_u$
denotes left multiplication by $u$ on every component space. The notation $C^{(k)}$ denotes the
$k$-skeleton of a complex, whereas $C^k$ denotes its set of open $k$-cells. Directed multigraphs,
graph actions, and the reversal of an oriented edge use the definitions of \cite{SerreTrees2003}.
Profinite graphs and their quotient maps use \cite{Ribes2017}.

For the following definition, we use the standard directed graph structure, with edge reversal as in
\cite{SerreTrees2003} when present. Our profinite version uses separate profinite vertex and edge
spaces, corresponding to the closed and open edge case of \cite{Ribes2017}.
\begin{definition}
\label{def:graph-diagram}
A directed multigraph $X$ consists of a vertex set $V(X)$, a directed edge set $A(X)$, and initial
and terminal maps
\[
s_X,t_X:A(X)\longrightarrow V(X).
\]
Loops and parallel directed edges are allowed. Connectedness means connectedness of the underlying
undirected multigraph. Local finiteness means that each vertex is incident to finitely many directed
edges.

When $X$ comes from an undirected graph, each geometric edge has two distinct orientations. We keep
the reversal involution $a\mapsto\bar a$, satisfying
\[
\bar{\bar a}=a,\qquad \bar a\neq a,
\qquad
s_X(\bar a)=t_X(a),\qquad
t_X(\bar a)=s_X(a).
\]
The directed edge residual complex is
\[
\Dgraph(X)=\bigl(V(X),A(X),s_X,t_X\bigr),
\]
together with reversal when present.

An isomorphism $F:\Dgraph(X)\to\Dgraph(Y)$ consists of bijections
\[
F_V:V(X)\to V(Y),\qquad
F_A:A(X)\to A(Y)
\]
such that
\[
F_V\circ s_X=s_Y\circ F_A,
\qquad
F_V\circ t_X=t_Y\circ F_A,
\]
and, when reversal is present,
\[
F_A(\bar a)=\cl{F_A(a)}.
\]
A profinite directed edge residual complex has profinite vertex and directed edge spaces and
continuous structure maps, including reversal when present. Its isomorphisms satisfy the same
identities, with $F_V$ and $F_A$ homeomorphisms.

If a group $N$ acts on $X$ by automorphisms preserving these maps, its quotient residual complex
has component sets
\[
N\backslash V(X),\qquad N\backslash A(X),
\]
with
\[
s(Na)=N s_X(a),\qquad
t(Na)=N t_X(a),
\qquad
\cl{Na}=N\bar a
\]
when reversal is present. The induced reversal may have fixed points. If $Na=N\bar a$, then
\[
s(Na)=s(N\bar a)=t(Na),
\]
so this directed edge is a loop in the quotient. Different directed edge orbits stay different, even
when their initial and terminal vertex orbits agree.
\end{definition}
\begin{definition}\label{def:graph-distinguished}
Let $G\acts X$ by directed graph automorphisms. A finite distinguished family is a finite label set
$\Lambda$ and a family $\mathcal S=(Z_\lambda)_{\lambda\in\Lambda}$ of $G$-invariant subgraphs. Thus
$V(Z_\lambda)\subseteq V(X)$ and $A(Z_\lambda)\subseteq A(X)$ are invariant, the endpoint maps send
$A(Z_\lambda)$ into $V(Z_\lambda)$, and reversal preserves $A(Z_\lambda)$ when present. An
isomorphism maps a family $(Z_{1,\lambda})$ to a family $(Z_{2,\lambda})$ according to
$\sigma\in\mathfrak S_\Lambda$ if it sends $Z_{1,\lambda}$ onto $Z_{2,\sigma(\lambda)}$ for every
$\lambda$. The distinguished family may be empty.
\end{definition}

\begin{definition}
\label{def:graph-completion}
Let $G$ be finitely generated and residually finite, acting cocompactly on a locally finite directed
multigraph $X$ with finite vertex stabilizers. We define
\[
\wh{\Dgraph}_G(X)=\varprojlim_{m\geq1}\Dgraph\bigl(K_m(G)\backslash X\bigr).
\]
The inverse limit is taken on vertices, directed edges, both endpoint maps, and reversal when
present. For a distinguished family $\mathcal S=(Z_\lambda)_{\lambda\in\Lambda}$, we also define
\[
\wh{\Dgraph}_G(Z_\lambda)
=\varprojlim_{m\geq1}\Dgraph\bigl(K_m(G)\backslash Z_\lambda\bigr)
\subseteq\wh{\Dgraph}_G(X).
\]
We abbreviate these distinguished subcomplexes by $\wh Z_\lambda$ and set
$\wh{\mathcal S}=(\wh Z_\lambda)_{\lambda\in\Lambda}$. An original vertex or directed edge means a
point in the canonical image of $V(X)$ or $A(X)$, respectively. The completed orbits with finite
stabilizers and their original subsets are described in \cref{lem:graph-finite-orbit}.
\end{definition}

\begin{lemma}\label{lem:graph-finite-orbit}
Let $G$ be finitely generated and residually finite and let $H\leq G$ be finite. The canonical map
\begin{equation}\label{eq:graph-coset-completion}
\wh G/H\longrightarrow\varprojlim_m K_m(G)\backslash G/H
\end{equation}
is a homeomorphism of profinite $\wh G$-sets. Under this identification, $G/H$ embeds as the subset
of cosets represented by elements of $G$.
\end{lemma}
\begin{proof}
As $K_m(G)$ is normal, $K_m(G)\backslash G/H\cong G/(K_m(G)H)$ as coset sets. The finite subgroup
$H$ is closed in $\wh G$, so the quotient maps define \eqref{eq:graph-coset-completion}. For any
closed subgroup $B$ of a profinite group and cofinal family $(N_m)$ of open normal subgroups,
\[
B=\bigcap_m BN_m.
\]
If $z\notin B$, choose an open normal subgroup $N$ such that $zN\cap B=\varnothing$, and then choose
$m$ with $N_m\leq N$. We have $z\notin BN_m$. Applying the identity to $B=H$ proves injectivity of
\eqref{eq:graph-coset-completion}.

For surjectivity, a compatible family of quotient cosets defines nested nonempty compact sets
\[
g_m\cl{K_m(G)}H\subseteq\wh G.
\]
Any point in their intersection maps to the given family. The map is thus a continuous bijection
from a compact space to a Hausdorff space. The claim about the original subset follows from the
construction and the injection $G\hookrightarrow\wh G$.
\end{proof}

For $a\in A(X)$ one has $\Stab_G(a)\leq\Stab_G(s_X(a))$, so directed edge stabilizers are finite. By
cocompactness and local finiteness there are finitely many vertex and directed edge orbits, so each
completed orbit is a coset space $\wh G/H$ with $H$ finite, and every invariant subgraph is a union
of some of these finitely many orbits on vertices and edges. Hence its completed subcomplex meets
the original complex in that invariant subgraph.

\begin{definition}\label{def:graph-cofinal-matches}
For $i=1,2$, let $G_i\acts X_i$ satisfy \cref{def:graph-completion}, let
$\mathcal S_i=(Z_{i,\lambda})_{\lambda\in\Lambda}$ be distinguished families, and let
$\Phi:\wh G_1\xrightarrow{\cong}\wh G_2$. We fix a finite subgroup
$\mathfrak P\leq\mathfrak S_\Lambda$ of permitted label permutations. A cofinal family of
equivariant isomorphisms of finite quotients consists of an increasing unbounded sequence
$m_1<m_2<\cdots$ and, for every $j$, a $\Phi_{m_j}$-equivariant isomorphism
\begin{equation}\label{eq:graph-finite-matches}
F_j:\Dgraph\bigl(K_{m_j}(G_1)\backslash X_1\bigr)
\xrightarrow{\cong}
\Dgraph\bigl(K_{m_j}(G_2)\backslash X_2\bigr)
\end{equation}
which maps the quotient images of $Z_{1,\lambda}$ onto those of $Z_{2,\sigma_j(\lambda)}$ for some
$\sigma_j\in\mathfrak P$, for different $j$ unrelated to one another. If there are no distinguished
subgraphs, the permutation coordinate is ignored.
\end{definition}

\subsection*{Compatible selection at characteristic finite quotients}
\begin{lemma}\label{lem:graph-descent}
Let $k\geq j$. Every isomorphism \eqref{eq:graph-finite-matches} at level $m_k$ descends uniquely to
a $\Phi_{m_j}$-equivariant isomorphism at level $m_j$. The descended map has the same permutation of
the distinguished labels.
\end{lemma}
\begin{proof}
We set $N_{i,k}=K_{m_k}(G_i)$ and $N_{i,j}=K_{m_j}(G_i)$. Then $N_{i,k}\leq N_{i,j}$, and the
complex at level $m_j$ is the quotient of the complex at level $m_k$ by $N_{i,j}/N_{i,k}$. Two
vertices or directed edges at level $m_k$ with the same image at level $m_j$ differ by an element of
$N_{1,j}/N_{1,k}$. Equivariance and \cref{lem:char-functor} imply that their images differ by an
element of $N_{2,j}/N_{2,k}$. The map therefore descends and commutes with endpoints and reversal.
The inverse descends in the same way, so the descended map is bijective. The quotient images of the
distinguished subgraphs are permuted by the same permutation. Uniqueness follows from surjectivity
of the quotient projections.
\end{proof}

\begin{theorem}\label{thm:graph-compatible-selection}
Under \cref{def:graph-cofinal-matches}, permitted isomorphisms of finite quotients $E_j$ can be
chosen at all selected levels so that
\[
r_{k,j}(E_k)=E_j\qquad(k\geq j),
\]
where $r_{k,j}$ is the descent map of \cref{lem:graph-descent}. Their inverse limit is a
$\Phi$-equivariant isomorphism
\begin{equation}\label{eq:graph-compatible-limit}
F:\wh{\Dgraph}_{G_1}(X_1)\xrightarrow{\cong}\wh{\Dgraph}_{G_2}(X_2).
\end{equation}
If distinguished subgraphs are present, the same permutation $\sigma\in\mathfrak P$ occurs at every
selected level and in the inverse limit.
\end{theorem}
\begin{proof}
Let $\mathscr E_j$ be the set of all permitted $\Phi_{m_j}$-equivariant isomorphisms, including the
permutation coordinate. Both quotient complexes and the permitted permutation group are finite, so
$\mathscr E_j$ is finite. It is nonempty by assumption. Descent satisfies
\[
r_{\ell,j}=r_{k,j}\circ r_{\ell,k}\qquad(\ell\geq k\geq j).
\]
We give each $\mathscr E_j$ the discrete topology. The product
$\mathscr E=\prod_{j\geq1}\mathscr E_j$ is compact. For $k\geq j$, the condition $r_{k,j}(E_k)=E_j$
defines a closed subset of $\mathscr E$. Any finite collection of these conditions involves finitely
many indices. We choose $N$ at least as large as all of them, choose any $E_N\in\mathscr E_N$, and
set every required lower coordinate equal to its descent from $E_N$. Transitivity satisfies all the
conditions in the finite collection. The closed conditions therefore have the finite intersection
property, and by compactness there is a compatible family. See also \cite{RibesZalesskii2010}.

The selected subgroups are cofinal because $(m_j)$ is unbounded. Taking inverse limits yields
\eqref{eq:graph-compatible-limit}, and the compatible inverse maps give its inverse. The permutation
coordinate is unchanged by descent, so the selected family involves a single permutation.
\end{proof}

\begin{remark}\label{rem:graph-selection-surjectivity}
The descent maps $r_{k,j}$ are not necessarily surjective, and the maps selected by
\cref{thm:graph-compatible-selection} may differ from the initial choices.
\end{remark}

\subsection*{Realization of original graph actions and their outer automorphisms}

\begin{proof}[Proof of \Cref{thm:graph-realization}]
Under condition \textup{(a)}, \cref{thm:graph-compatible-selection} provides compatible maps at the
chosen levels, a common permutation $\sigma\in\mathfrak P$, and a $\Phi$-equivariant inverse limit
isomorphism $F$ satisfying condition \textup{(b)}. Under condition \textup{(b)}, we use the given
map $F$. It is therefore enough to prove the realization conclusion from condition \textup{(b)}.

We choose $v_1\in V(X_1)$. The point $F(v_1)$ belongs to a completed orbit $\wh G_2/V_2$, where
$V_2$ is the finite stabilizer of some $v_2\in V(X_2)$. We choose $w\in\wh G_2$ with $wF(v_1)=v_2$,
and set
\[
F'=wF,\qquad \Phi'=\Inn(w)\circ\Phi.
\]
Then $F'$ is $\Phi'$-equivariant and $F'(v_1)=v_2$.

Fix a directed edge orbit of $X_2$ and choose an original representative $a_0$. Choose original
representatives $x_0,y_0$ of its initial and terminal vertex orbits, and set
\[
L=\Stab_{G_2}(a_0),\qquad V=\Stab_{G_2}(x_0),\qquad W=\Stab_{G_2}(y_0).
\]
All three stabilizers are finite. Since $x_0$ represents the $G_2$-orbit containing $s_{X_2}(a_0)$,
and $y_0$ represents the orbit containing $t_{X_2}(a_0)$, there exist $c,d\in G_2$ such that
\[
s_{X_2}(a_0)=cx_0,
\qquad
t_{X_2}(a_0)=dy_0.
\]
By equivariance of the completed endpoint maps,
\[
s(ga_0)=gcx_0,\qquad
t(ga_0)=gdy_0,
\]
so in these coset coordinates
\[
s(gL)=gcV,\qquad
t(gL)=gdW
\qquad(g\in\wh G_2).
\]

These formulas are well defined, since an element of $L$ fixes $a_0$ and hence both of its
endpoints, so that
\[
c^{-1}Lc\leq V,\qquad d^{-1}Ld\leq W.
\]
Hence, replacing $g$ by $g\ell$ with $\ell\in L$ changes neither endpoint coset:
\[g\ell cV=gc(c^{-1}\ell c)V=gcV,\qquad g\ell dW=gd(d^{-1}\ell d)W=gdW.\]
If $gcV=\gamma V$ for some $\gamma\in G_2$, then
\[
g\in\gamma Vc^{-1}\subseteq G_2.
\]
Hence the directed edge and its other endpoint are original. The argument for the terminal map is
identical. This is the local occurrence argument of \cref{thm:discrete-star}, adapted to finite
stabilizers.

Starting at $v_1$, apply the argument at the initial or terminal vertex to every incident directed
edge. Its image and its other endpoint are original. Induction on distance in the connected
underlying graph yields $F'(V(X_1))\subseteq V(X_2)$ and $F'(A(X_1))\subseteq A(X_2)$, and the local
argument applied to $(F')^{-1}$ turns these inclusions into equalities. Because distinguished
subgraphs are invariant unions of orbits, their completions meet the original complex in their
original vertices and directed edges. The distinguished family claim follows.

For $\gamma\in G_1$, by equivariance $F'(\gamma v_1)=\Phi'(\gamma)v_2$. This vertex is original, so
$\Phi'(\gamma)V_2=\delta V_2$ for some $\delta\in G_2$, and
$\Phi'(\gamma)\in\delta V_2\subseteq G_2$. The inverse argument shows $\Phi'(G_1)=G_2$. We set
$h=\Phi'|_{G_1}$. Continuity and density imply $\Phi'=\wh h$, and
\eqref{eq:graph-discrete-realization} follows with $u=w^{-1}$.

Conversely, a map $F$ as in condition \textup{(b)} descends to every characteristic quotient:
quotienting each completed orbit by $\cl{K_m(G_i)}$ determines the orbit at level $m$, while $\Phi$
identifies the relevant open normal subgroups by \cref{lem:char-functor}, and the descended maps
commute with the endpoint maps, reversal and the distinguished subgraphs. So condition \textup{(b)}
implies condition \textup{(a)} as well.
\end{proof}

\begin{corollary}\label{cor:graph-free}
Under the remaining conditions of \cref{thm:graph-realization}, suppose the actions are free on
vertices and directed edges. Then its realization conclusion holds, and every completed vertex and
directed edge orbit is a free profinite orbit.
\end{corollary}
\begin{proof}
The vertex and directed edge stabilizers are trivial. We apply
\cref{lem:graph-finite-orbit,thm:graph-realization}.
\end{proof}

\begin{definition}
\label{def:graph-action-automorphisms}
Let $G\acts X$ satisfy the action condition of \cref{thm:graph-realization} and fix a finite
labelled family $\mathcal S=(Z_\lambda)_{\lambda\in\Lambda}$ of invariant subgraphs. We define
$\Aut(G\acts X,\mathcal S)$ to consist of pairs $(\alpha,A)$ with $\alpha\in\Aut(G)$ and a complex
automorphism $A:\Dgraph(X)\to\Dgraph(X)$ such that
\[
A(gz)=\alpha(g)A(z),\qquad A(Z_\lambda)=Z_\lambda\quad(\lambda\in\Lambda).
\]
Each label is fixed in this definition. We set
\[
\Daction(G\acts X)=\{(\Inn(g),L_g):g\in G\},\qquad
\Out(G\acts X,\mathcal S)=\Aut(G\acts X,\mathcal S)/\Daction(G\acts X).
\]
The diagonal subgroup is normal, since conjugation by $(\alpha,A)$ sends $(\Inn(g),L_g)$ to
$(\Inn(\alpha(g)),L_{\alpha(g)})$. We define the completed action groups by the same equations, with
continuous maps and with diagonal translations by $u\in\wh G$. Since the subgroups $K_m(G)$ are
characteristic, the completion of a discrete pair is well defined, and we obtain
\begin{equation}\label{eq:graph-outer-completion}
\Out(G\acts X,\mathcal S)\longrightarrow
\Out\bigl(\wh G\acts\wh{\Dgraph}_G(X),\wh{\mathcal S}\bigr).
\end{equation}
\end{definition}

\begin{theorem}\label{thm:graph-action-outer}
Under the action conditions of \cref{thm:graph-realization}, the homomorphism
\eqref{eq:graph-outer-completion} is an isomorphism.
\end{theorem}
\begin{proof}
A continuous pair $(\Phi,F)$ on the completed action descends to every characteristic quotient by
\cref{lem:char-functor}. By \cref{thm:graph-realization}, a single diagonal left translation makes
this pair restrict to a discrete pair $(h,F'|_X)$ preserving every distinguished label. Its
completion is the translated profinite pair because the maps agree on the dense original vertex and
directed edge sets. So the map is surjective.

For injectivity, suppose a discrete pair $(\alpha,A)$ becomes diagonal: $\wh\alpha=\Inn(u)$ and
$\wh A=L_u$ for some $u\in\wh G$. We choose $v\in V(X)$ and set $V=\Stab_G(v)$, a finite subgroup.
As $A(v)=uv$ is original, $uV=\gamma V$ for some $\gamma\in G$, and $u\in\gamma V\subseteq G$.
Restricting both equalities to the original subsets, we get $\alpha=\Inn(u)$ and $A=L_u$. The
discrete pair was already diagonal.
\end{proof}

\subsection*{Realization from a graph}

\begin{definition}\label{def:graph-pi1-surjective}
Let $C$ be a connected finite simplicial complex. A finite connected subgraph $Y\subseteq C^{(1)}$
is \emph{$\pi_1$-surjective} if its inclusion induces a surjection
$\pi_1(Y)\twoheadrightarrow\pi_1(C)$. We give $Y$ the directed edge definition with both
orientations of every geometric edge.
\end{definition}

\begin{lemma}\label{lem:graph-full-lift}
Every connected finite simplicial complex contains a $\pi_1$-surjective finite connected subgraph.
If $p:\widetilde C\to C$ is the universal cover, $G=\pi_1(C)$ is its deck group, and $Y$ is such a
subgraph, then $X=p^{-1}(Y)$ is connected and locally finite. The action $G\acts X$ is free and
cocompact and $G\backslash X=Y$.
\end{lemma}
\begin{proof}
The inclusion of the $1$-skeleton $C^{(1)}$ induces a surjection on fundamental groups, since every
loop in $C$ is homotopic to an edge path \cite{Hatcher2002}, so $C^{(1)}$ itself is a
$\pi_1$-surjective finite connected subgraph. Now let $Y$ be any such subgraph. The components of
$p^{-1}(Y)$ correspond to the cosets of the image of $\pi_1(Y)$ in $\pi_1(C)$, so $X=p^{-1}(Y)$ is
connected precisely because the inclusion of $Y$ is $\pi_1$-surjective. It is locally finite because
$C$ is finite, and the deck group acts freely on $\widetilde C$, hence on $X$, with quotient
$p(X)=Y$; in particular the action is cocompact.
\end{proof}

\begin{theorem}\label{thm:graph-core-realization}
For $i=1,2$, let $C_i$ be connected finite simplicial complexes, let $G_i=\pi_1(C_i)$ be residually
finite, and identify $G_i$ with the deck group of $p_i:\widetilde C_i\to C_i$. We choose
$\pi_1$-surjective graphs $Y_i\subseteq C_i^{(1)}$ and set $X_i=p_i^{-1}(Y_i)$. Suppose
$\Phi:\wh G_1\xrightarrow{\cong}\wh G_2$ and the directed edge complexes admit a cofinal family as
in \cref{def:graph-cofinal-matches}. Then the finite maps can be selected compatibly, and there are
$u\in\wh G_2$ and $h:G_1\xrightarrow{\cong}G_2$ with $\Phi=\Inn(u)\circ\wh h$. The normalized
inverse limit map restricts to an $h$-equivariant graph isomorphism $F':X_1\to X_2$ and descends to
a graph isomorphism $\cl F:Y_1\to Y_2$.

For compatible base vertices, let $\rho_i:\pi_1(Y_i)\twoheadrightarrow G_i$ be induced by inclusion.
Then
\begin{equation}\label{eq:graph-monodromy-compatibility}
\rho_2\circ\cl F_*=h\circ\rho_1.
\end{equation}
Changing the base vertices or connecting paths conjugates this identity in $G_2$.
\end{theorem}
\begin{proof}
By \cref{lem:graph-full-lift}, the actions are free and cocompact. We apply
\cref{thm:graph-compatible-selection,cor:graph-free} and take the quotient of the normalized
original map. We choose $x_1\in X_1$ and $x_2=F'(x_1)$. Lift a based loop in $Y_1$ from $x_1$ and
apply $F'$. Equivariance maps its terminal deck transformation by $h$ to that of the image loop from
$x_2$. So \eqref{eq:graph-monodromy-compatibility} holds, and the formula for a change of basepoint
takes care of the remaining statement.
\end{proof}

\subsection*{Full simplex face residual complexes and PL realization}

\begin{definition}\label{def:graph-simplex-occurrence}
For a finite dimensional simplicial complex $X$, we use the full cellular residual complex
$\Icell_\bullet(X)$ of \cref{def:full-cellular-occurrence}. A simplex and one of its codimension one
faces determine a unique face position, so the occurrence spaces identify with
\[
I_{k,k-1}(X)=\{(\sigma,\tau)\in X^k\times X^{k-1}:\tau\text{ is a face of }\sigma\},
\]
with $p_{k,k-1}(\sigma,\tau)=\sigma$ and $q_{k,k-1}(\sigma,\tau)=\tau$. This pair notation
identifies the occurrence triples of \cref{def:full-cellular-occurrence}. For any subgroup $N$
acting cellularly on $X$, we keep
\[
\Icell_\bullet(N\backslash X)=N\backslash\Icell_\bullet(X)
\]
on every cell and occurrence space. In particular, the occurrence space is
$N\backslash I_{k,k-1}(X)$. Distinct occurrence orbits remain distinct even when their incident
quotient cells coincide. The same rule for orbits of face occurrences applies to regular CW
complexes.

A distinguished invariant subcomplex $Z_\lambda$ specifies the subsets $Z_\lambda^k\subseteq X^k$
and $I_{k,k-1}(Z_\lambda)\subseteq I_{k,k-1}(X)$ in every dimension. A map preserves a finite family
of such subcomplexes if it preserves all these subsets according to the same label permutation.
Completions are the component-wise characteristic inverse limits of
\cref{def:full-cellular-occurrence}.
\end{definition}

\begin{proof}[Proof of \Cref{thm:graph-pl-realization}]
The descent argument of \cref{lem:graph-descent} applies simultaneously to every cell space and
every codimension one occurrence space. At each selected level, the set of permitted isomorphisms is
finite and nonempty. The compactness argument of \cref{thm:graph-compatible-selection} therefore
produces compatible isomorphisms, with one permutation of the distinguished labels when present.
Their inverse limit is a $\Phi$-equivariant isomorphism
\[
F:\wh{\Icell}_\bullet(X_1)
\xrightarrow{\cong}
\wh{\Icell}_\bullet(X_2).
\]

Each $X_i$ is a connected, locally finite, finite dimensional simplicial complex. Its deck group
$G_i$ acts freely and cocompactly, and $G_i$ is finitely generated because $C_i$ is finite. It
follows that \cref{thm:all-dimensional-occurrence} applies and provides $u\in\wh G_2$ and an
isomorphism $h:G_1\to G_2$ such that
\[
\Phi=\Inn(u)\circ\wh h,
\qquad
F'=L_{u^{-1}}\circ F
\]
restricts to an $h$-equivariant isomorphism between the full original residual complexes.

By iterating the codimension one face relations, $F'$ preserves every face relation. In particular,
it sends the vertex set of each simplex onto the vertex set of its image. A simplex is determined by
its vertices, so these maps define a simplicial isomorphism $X_1\to X_2$. Its affine extensions
agree on common faces. By equivariance we therefore have a simplicial, hence PL, homeomorphism
\[
C_1=G_1\backslash X_1
\xrightarrow{\cong}
G_2\backslash X_2=C_2.
\]
Lifting based loops shows that its induced homomorphism on fundamental groups is $h$, up to the
choice of basepoint paths.

Finally, each distinguished invariant subcomplex is a union of cell and occurrence orbits. Its
completion intersects the original residual complex in precisely those orbits. Consequently, $F'$
and the descended homeomorphism preserve the distinguished subcomplexes according to the selected
permutation.
\end{proof}


\section{Massey triple products}\label{sec:closed-sign}
Let $M$ be a compact oriented $3$-manifold. We use the Thurston seminorm on $H^1(M,\R)$ and, when
present, its fibered cones \cite{Thurston1986,AschenbrennerFriedlWilton2015}. For a closed
hyperbolic $M$, or a compact core of a complete finite volume hyperbolic $3$-manifold, the Thurston
seminorm is a norm.
\subsection*{Cohomological background}

\begin{terminology}
Let $R$ be a commutative ring. A differential graded $R$-algebra $A^*$ is a graded unital
associative $R$-algebra with an $R$-linear differential $d:A^k\to A^{k+1}$ satisfying $d^2=0$ and
$d(xy)=d(x)y+(-1)^kxd(y)$ for $x\in A^k$. Its cohomology is
$H^k(A)=\ker(d:A^k\to A^{k+1})/\im(d:A^{k-1}\to A^k)$. A morphism of differential graded
$R$-algebras is a cochain map preserving the unit and multiplication. Such a morphism is a
\emph{quasi-isomorphism} if it induces an isomorphism in every cohomological degree. For
$a,b,c\in H^1(A)$ with $ab=bc=0$, we define
\[
\langle a,b,c\rangle_A
=\{[AV+UC]:[A]=a,\ [B]=b,\ [C]=c,\ dU=AB,\ dV=BC\},
\]
where $A,B,C$ are cocycles of degree one and $U,V\in A^1$. Juxtaposition denotes the algebra
product, written $\smile$ for cup products. This is the triple Massey product with the above sign
definition. Its indeterminacy is $aH^1(A)+H^1(A)c$, and the word \emph{complete} refers to the whole
set of classes in this definition. See \cite[Definition~1.2, Theorem~1.5 and
Proposition~2.3]{May1969}.
\end{terminology}

Let $\Phi:\wh{\pi_1M_1}\to\wh{\pi_1M_2}$ be an isomorphism of the completions of orientable
connected finite volume hyperbolic $3$-manifold groups. We begin with Liu's regularity theorem
\cite[Theorems~1.2--1.3 and Corollary~6.2]{Liu2023}, which gives
\begin{equation}
\Phi^*=\mu(F\otimes\hZ),\qquad
F:H^1(M_2,\Z)\xrightarrow{\cong}H^1(M_1,\Z),\quad
\mu\in\hZ^\times.
\label{eq:liu}
\end{equation}
The integral map $F$ and its inverse preserve the Thurston norm and fibered classes, and
corresponding fiber kernels have corresponding closures. From here on in this section all manifolds
are connected and closed.

\begin{lemma}\label{lem:massey-dga}
Let $R$ be a commutative ring and let
\[
F:A^*\longrightarrow B^*
\]
be a unital morphism of differential graded $R$-algebras that induces an isomorphism on cohomology.
Let $a,b,c\in H^1(A)$ satisfy $a\smile b=b\smile c=0$. Then
\begin{equation}\label{eq:massey-dga-naturality}
F_*\bigl(\langle a,b,c\rangle_A\bigr)
=
\langle F_*a,F_*b,F_*c\rangle_B.
\end{equation}
Moreover, $F_*$ maps the full indeterminacy subgroup
\[
a\smile H^1(A)+H^1(A)\smile c
\]
isomorphically onto
\[
F_*a\smile H^1(B)+H^1(B)\smile F_*c.
\]
The same conclusion holds for the cohomology isomorphism obtained from a finite sequence of unital
multiplicative quasi-isomorphisms pointing in either direction.
\end{lemma}
\begin{proof}
Choose cocycles $A,B,C$ representing $a,b,c$ and cochains $U,V$ with
\[
dU=A\smile B,
\qquad
dV=B\smile C.
\]
With the definition above in degree one, the corresponding Massey cocycle is $A\smile V+U\smile C$.
Since $F$ commutes with the differential and product, the five cochains $F(A),F(B),F(C),F(U),F(V)$
form a defining system in $B^*$, and
\[
F(A\smile V+U\smile C)
=F(A)\smile F(V)+F(U)\smile F(C).
\]
Thus the left side of \eqref{eq:massey-dga-naturality} is contained in the right side.

We check directly that the full indeterminacy is the subgroup in the statement. Replacing $U$ by
$U+Z$ for a cocycle of degree one $Z$ changes the Massey class by $[Z]\smile c$, and replacing $V$
by $V+Z'$ changes it by $a\smile[Z']$. Conversely, every element of $a\smile H^1(A)+H^1(A)\smile c$
is obtained in this way. We also check changes of representatives. For $s\in A^0$, replacing $A$ by
$A+ds$ and $U$ by $U+sB$ changes the Massey cocycle by $d(sV)$. Replacing $B$ by $B+ds$, $U$ by
$U-As$, and $V$ by $V+sC$ leaves it unchanged. Replacing $C$ by $C+ds$ and $V$ by $V-Bs$ changes it
by $-d(Us)$. Each replacement preserves the defining equations, so changing representatives does not
change the coset, and a nonempty complete triple Massey product is exactly one coset of its full
indeterminacy subgroup, as in \cite[Proposition~2.3]{May1969}.

The ring isomorphism $F_*$ maps the first indeterminacy subgroup isomorphically onto the second. Let
$m_A$ be the class of the chosen defining system. The image $F_*(\langle a,b,c\rangle_A)$ is
therefore the coset
\[
F_*(m_A)+F_*a\smile H^1(B)+H^1(B)\smile F_*c.
\]
The complete product in $B$ contains $F_*(m_A)$ and is a coset of the same subgroup. The two cosets
are equal, which proves \eqref{eq:massey-dga-naturality}. For a leftward quasi-isomorphism, apply
the result to its induced cohomology isomorphism and then invert that isomorphism. Composition
proves the statement for a finite sequence of these maps.
\end{proof}
Our next lemma compares finite-coefficient cohomology and complete triple Massey products through
the goodness of $3$-manifold groups.
\begin{lemma}\label{lem:comparison}
Let $P,Q$ be connected closed orientable aspherical $3$-manifolds, let
$\Psi:\wh{\pi_1P}\xrightarrow{\cong}\wh{\pi_1Q}$ be an isomorphism, and set $R_n=\Z/n\Z$ for
$n\geq1$. Then there are isomorphisms
\[
\Psi_n^*:H^*(Q,R_n)\xrightarrow{\cong}H^*(P,R_n),
\]
natural under change of coefficients, which preserve cup products and complete triple Massey
products of classes of degree one, including their full indeterminacy. In addition, if
$\omega_P\in H^3(P,\Z)$ and $\omega_Q\in H^3(Q,\Z)$ are orientation classes evaluating to one, then
\[
\Psi_n^*(\omega_Q\bmod n)
=\delta_n(\omega_P\bmod n),
\qquad
(\delta_n)_n=\delta\in\hZ^\times.
\]
\end{lemma}
\begin{proof}
Choose a basepoint in $P$ and set $G=\pi_1(P)$. Both $G$ and $\wh G$ act trivially on the
coefficient module $R_n$. Let $C^*_b(G,R_n)$ denote the normalized bar complex for group cohomology
\cite{Brown1982}. The continuous complex $C^*_{\mathrm{cts}}(\wh G,R_n)$ is defined by the same
formulas, with continuous functions on $(\wh G)^k$ and the discrete topology on $R_n$
\cite{RibesZalesskii2010}. The usual normalization maps and homotopies are finite sums of composites
of face and degeneracy operators in each degree. They also preserve continuity, so the normalized
continuous complex computes continuous cohomology. In both complexes the product is
\[
(f\smile g)(g_1,\ldots,g_{p+q})
=f(g_1,\ldots,g_p)g(g_{p+1},\ldots,g_{p+q})
\qquad (\deg f=p,\ \deg g=q).
\]
Restriction along $G\to\wh G$ therefore preserves the unit, differential, and product. As compact
$3$-manifold groups are good \cite{AschenbrennerFriedlWilton2015}, this restriction induces an
isomorphism on cohomology with coefficients in $R_n$. It is consequently a unital multiplicative
quasi-isomorphism.

Let $BG$ be the simplicial bar construction with $k$-simplices $G^k$, and let $|BG|$ be its
geometric realization \cite{Brown1982}. For a simplicial set $K$, write $N^*(K,R_n)$ for the
normalized simplicial cochain complex, whose degree-$q$ cochains are functions $K_q\to R_n$
vanishing on degenerate simplices. We equip this complex with the usual simplicial coboundary and
the Alexander--Whitney cup product, making it a differential graded $R_n$-algebra. Write
$\operatorname{Sing}X$ for the singular simplicial set of a space $X$. Because $P$ is aspherical and
has the homotopy type of a CW complex, there is a based classifying homotopy equivalence
$c_P:P\to |BG|$ inducing the identity of $G$ under the canonical identification $\pi_1(|BG|)=G$.
Recall that for normalized cochains, the Alexander--Whitney product is used \cite{Brown1982}. For a
simplex $\sigma$ of dimension $p+q$, this product evaluates $f$ on the face with vertices
$0,\ldots,p$ and $g$ on the face with vertices $p,\ldots,p+q$, and it is natural under simplicial
maps and agrees with the product under $C^*_{b}(G,R_n)\cong N^*(BG,R_n)$. Altogether we obtain the
following sequence of unital multiplicative quasi-isomorphisms:
\begin{equation}\label{eq:absolute-dga-zigzag}
C^*_{\mathrm{cts}}(\wh G,R_n)
\longrightarrow
C^*_{b}(G,R_n)
\cong N^*(BG,R_n)
\longleftarrow N^*(\operatorname{Sing}|BG|,R_n)
\xrightarrow{c_P^*}N^*(\operatorname{Sing}P,R_n).
\end{equation}
The middle arrow is induced by the canonical simplicial map $BG\to\operatorname{Sing}|BG|$. This map
is a weak equivalence and induces a homology isomorphism on normalized integral chains. Its mapping
cone is an acyclic complex of free abelian groups. Its cycles are free, so the short exact sequences
in that cone split and the cone is contractible. Applying $\Hom(-,R_n)$ therefore shows that the
middle arrow is a quasi-isomorphism for every $n$. The last arrow is a quasi-isomorphism because
$c_P$ is a homotopy equivalence. All arrows obviously preserve the Alexander--Whitney product.

Now we do exactly the same for $Q$ and then pull back continuous cochains by using $\Psi$, thereby
obtaining $\Psi_n^*$. As above, simplicial maps and classifying maps are chosen independently of
$n$, and restriction of group cochains also commutes with coefficient reduction. Hence their induced
cohomology isomorphisms and inverses commute with $R_m\to R_n$ whenever $n\mid m$. In degree one a
class corresponds under these identifications to its homomorphism to $R_n$, so $\Psi_n^*$ is
precomposition with $\Psi$. They also commute with restriction to finite index subgroups: for a
connected finite cover $p:P'\to P$ and the corresponding inclusion $\iota:G'\hookrightarrow G$, the
two classifying maps
\[
c_P\circ p,
\qquad
|B\iota|\circ c_{P'}:P'\longrightarrow |BG|
\]
induce the same homomorphism on fundamental groups, after compatible choices of basepoints. As
$|BG|$ is a $K(G,1)$, these maps are homotopic and induce the same pullback on cohomology.
Restriction of bar and continuous cochains commutes with the subgroup inclusions. The same argument
applies to $Q$ and its covers.

Naturality of cup products follows from multiplicativity. Naturality of triple Massey products as
sets, including the full indeterminacy, follows from \cref{lem:massey-dga} applied to the two
sequences of identification maps.

Since $P$ and $Q$ are connected closed orientable $3$-manifolds,
\[
H^3(P,R_n)\cong R_n,
\qquad
H^3(Q,R_n)\cong R_n.
\]
Thus $\Psi_n^*$ multiplies the chosen orientation generator by a unique unit
$\delta_n\in R_n^\times$. If $n\mid m$, naturality under $R_m\to R_n$ shows
$\delta_m\bmod n=\delta_n$. It follows that $(\delta_n)_n$ defines an element
$\delta\in\varprojlim_n R_n=\hZ$. Every component is a unit, so $\delta\in\hZ^\times$.
\end{proof}

Finally, the definition of triple Massey products gives
\[
\langle\lambda a,\lambda b,\lambda c\rangle
=\lambda^3\langle a,b,c\rangle\qquad(\lambda\in R_n^\times).
\]
Indeed, multiplying the three cocycle representatives by $\lambda$ and the two cochains solving the
defining equations by $\lambda^2$ proves one inclusion, and the same construction with
$\lambda^{-1}$ proves the reverse inclusion.

\begin{lemma}
\label{lem:integral-test}
There exists a connected closed oriented $3$-manifold $Y$ with integral classes $x_1,x_2,x_3,u,v$ of
degree one and $m\in\langle u,u,v\rangle$ such that
\begin{equation}
\langle x_1x_2x_3,[Y]\rangle=1,
\qquad \langle mv,[Y]\rangle=1,
\qquad u^2=uv=v^2=0.
\label{eq:test}
\end{equation}
\end{lemma}
\begin{proof}
Let $H=\Z^3$ with multiplication
\[
(x,y,z)(x',y',z')=(x+x',y+y',z+z'+xy'),
\]
and let $N=H\backslash H(\R)$ be the compact Heisenberg nilmanifold, where $H(\R)$ is $\R^3$ with
the same multiplication. Its universal cover $H(\R)$ is contractible, so $N$ is a $K(H,1)$. We
identify integral group cohomology with $H^*(N,\Z)$ using the group, simplicial, and singular
identifications in \eqref{eq:absolute-dga-zigzag}, with $G=H$, $P=N$, and coefficient ring $\Z$,
ignoring the continuous cochain complex. The same contracting homotopy on the integral chain mapping
cone proves that these identifications are quasi-isomorphisms over $\Z$. Their multiplicativity and
\cref{lem:massey-dga} identify cup products and complete triple Massey products. On normalized
integral group cochains, we define
\[
A(x,y,z)=x,\quad B(x,y,z)=y,\quad
C(x,y,z)=-z,\quad T(x,y,z)=-\binom{x}{2}.
\]
For $df(g,h)=f(h)-f(gh)+f(g)$, direct substitution then shows
\[
dA=dB=0,\qquad dC=A\smile B,\qquad dT=A\smile A.
\]
Hence we have
\begin{equation}
M=A\smile C+T\smile B,
\qquad M(g,h)=-x_gz_h-\binom{x_g}{2}y_h
\label{eq:explicit-M}
\end{equation}
where $g=(x_g,y_g,z_g)$ and $h=(x_h,y_h,z_h)$. So $M$ is an integral cocycle representing an element
of $\langle a,a,b\rangle$, where $a=[A]$ and $b=[B]$. Also $b^2=0$, by the same calculation with the
binomial cochain in $y$.

The map $N\to\R/\Z$ induced by $y$ is a torus bundle whose fiber group is $K=\{(x,0,z)\}\cong\Z^2$.
Let $i:F=K\backslash\R^2\hookrightarrow N$ be the fiber inclusion. Since these identifications are
compatible with the maps induced by $K\hookrightarrow H$ and by $i$, the class $i^*[M]$ is
represented by the restriction of $M$ to $K$, and on $K$ the cocycle \eqref{eq:explicit-M} is
$-A|_K\smile Z$, where $Z(x,0,z)=z$. The torus fiber $F$ has fundamental group $K$ and contractible
universal cover, so $H^*_b(K,\Z)\cong H^*(F,\Z)$, and under this identification restriction of group
cochains along $K\hookrightarrow H$ induces the same map on cohomology as pullback along
$i:F\hookrightarrow N$. The ordered generators $e_x=(1,0,0)$ and $e_z=(0,0,1)$ give an integral bar
cycle $[e_x\mid e_z]-[e_z\mid e_x]$ representing a fundamental class of $F$, and
\[
M(e_x,e_z)-M(e_z,e_x)=-1.
\]
Thus $i^*[M]$ is a primitive generator of $H^2(F,\Z)$. Given an orientation of $N$, orient $F$ so
that $i_*[F]\in H_2(N,\Z)$ is the Poincar\'e dual of $b$. Then
$\langle i^*\eta,[F]\rangle=\langle \eta\smile b,[N]\rangle$ for every $\eta\in H^2(N,\Z)$, and so
\[
\langle[M]\smile b,[N]\rangle
=\langle i^*[M],[F]\rangle\in\{1,-1\}.
\]
Finally, choose the orientation of $N$ to make this value $1$.

Now set $Y=T^3\#N$, where $T^3=(\R/\Z)^3$. Let $p_T:Y\to T^3$ and $p_N:Y\to N$ be obtained by
collapsing the separating sphere of the connected sum and then collapsing the other summand to a
point. With the orientations defining the oriented connected sum,
\[
(p_T)_*[Y]=[T^3],
\qquad
(p_N)_*[Y]=[N],
\]
so both maps have degree one. Pulling back the three standard classes of $T^3$ and the classes
$a,b,[M]$ along these two pinch maps yields \eqref{eq:test}, since products of classes of positive
degree coming from different summands vanish.
\end{proof}

\begin{theorem}
\label{thm:closed-integral}
Let $M_1,M_2$ be connected closed orientable hyperbolic $3$-manifolds. For every isomorphism of
their profinite completions, the unit $\mu$ in \eqref{eq:liu} can be chosen in $\{1,-1\}$.
Equivalently, $\Phi^*$ on $H^1(-,\hZ)$ is the completion of an integral isomorphism.
\end{theorem}
\begin{proof}
We first consider connected finite covers $P\to M_1$ and $Q\to M_2$ whose fundamental groups have
isomorphic profinite completions. Suppose that there is a continuous map
\[
q:Q\longrightarrow Y,\qquad d=\deg(q)\neq0,
\]
where $Y$ is the manifold constructed in \cref{lem:integral-test}. We will obtain such a pair of
covers later in the proof. Let $\Psi:\wh{\pi_1P}\xrightarrow{\cong}\wh{\pi_1Q}$. Liu's regularity
theorem \eqref{eq:liu}, applied to the closed hyperbolic manifolds $P,Q$ and the isomorphism $\Psi$,
gives an integral isomorphism
\[
F_0:H^1(Q,\Z)\xrightarrow{\cong}H^1(P,\Z)
\]
and a unit $\lambda\in\hZ^\times$ such that $\Psi^*=\lambda(F_0\otimes\hZ)$.

Pull back the five classes of degree one and the Massey value of \cref{lem:integral-test} along $q$,
and suppress $q^*$ from the notation; then
\[
A_Q=\langle x_1x_2x_3,[Q]\rangle=d,
\qquad
B_Q=\langle mv,[Q]\rangle=d.
\]
Let $x_i',u',v'$ denote the images under $F_0$. For every $n\geq1$, naturality of $\Psi_n^*$ and
invertibility of the reduction $\lambda_n\in(\Z/n)^\times$ imply that the reductions of
$u'^2,u'v',v'^2\in H^2(P,\Z)$ modulo $n$ vanish. The exact sequence
$0\to\Z\xrightarrow{\,n\,}\Z\to\Z/n\to0$ gives
\[
\ker\bigl(H^2(P,\Z)\to H^2(P,\Z/n)\bigr)=nH^2(P,\Z).
\]
As $H^2(P,\Z)$ is finitely generated, $\bigcap_{n\geq1}nH^2(P,\Z)=0$, and in $H^2(P,\Z)$ we have
\[
u'^2=u'v'=v'^2=0.
\]
Choose integral cocycles $U',V'\in C^1(P,\Z)$ representing $u',v'$. Since $u'^2=u'v'=0$, we can
choose integral cochains $T',C'\in C^1(P,\Z)$ satisfying
\[
dT'=U'\smile U',
\qquad
dC'=U'\smile V'.
\]
They determine an integral value $m_P\in\langle u',u',v'\rangle$.

Let $R$ be either $\Z$ or $R_n=\Z/n\Z$, and write $u'_R,v'_R$ for the corresponding classes in
$H^1(P,R)$. For $R=\Z$ these are $u',v'$, and for $R=R_n$ they are their coefficient reductions.
Different choices of cochains change the Massey value by exactly the elements of the subgroup
\[
I_R=u'_R\smile H^1(P,R)
+H^1(P,R)\smile v'_R,
\]
the indeterminacy of the Massey product, so for any value $m_0\in\langle u'_R,u'_R,v'_R\rangle$ we
have $\langle u'_R,u'_R,v'_R\rangle=m_0+I_R$. For arbitrary $h,h'\in H^1(P,R)$, graded commutativity
in cohomology gives
\[
(u'_Rh+h'v'_R)v'_R=-h(u'_Rv'_R)+h'(v'_R)^2=0.
\]
Since coefficient reduction commutes with coboundary and cup product, the reduction of $m_P$ modulo
$n$ lies in $\langle u'_{R_n}, u'_{R_n}, v'_{R_n}\rangle$. Every other value with finite
coefficients differs from it by an element of $I_R$, including values whose defining equations do
not lift over $\Z$. Therefore all values have the same evaluation after cupping with $v'_R$. Set
\[
A_P=\langle x_1'x_2'x_3',[P]\rangle,
\qquad
B_P=\langle m_Pv',[P]\rangle.
\]

By \cref{lem:comparison}, the orientation multiplier in degree three is the same for the cubic and
the quartic evaluation. The three factors of degree one in the first evaluation contribute
$\lambda_n^3$; in the second, the Massey value contributes $\lambda_n^3$ and the additional factor
$v$ one more $\lambda_n$. Therefore, for every $n$,
\[
A_Q\delta_n=\lambda_n^3A_P,
\qquad
B_Q\delta_n=\lambda_n^4B_P
\qquad\text{in }\Z/n.
\]
Compatibility in $n$ implies
\begin{equation}\label{eq:two-weights}
A_Q\delta=\lambda^3A_P,
\qquad
B_Q\delta=\lambda^4B_P
\qquad\text{in }\hZ.
\end{equation}
The left sides are nonzero because $A_Q=B_Q=d\neq0$ and $\delta$ is a unit, so $A_P,B_P\neq0$. For a
prime $p$, we write $\Z_p=\varprojlim_r\Z/p^r\Z$ and let $\Q_p$ be its field of fractions. We embed
$\hZ$ in $\prod_p\Q_p$. Then from \eqref{eq:two-weights} we get
\[
\delta\lambda^{-3}=A_P/A_Q,
\qquad
\delta\lambda^{-4}=B_P/B_Q.
\]
Each rational number on the right is a unit in every $\Z_p$, and therefore belongs to
$\Q^\times\cap\hZ^\times=\{1,-1\}$. Dividing the two equalities gives
\[
\lambda=\frac{A_PB_Q}{A_QB_P}\in\{1,-1\}.
\]

We now return to the original isomorphism $\Phi:\wh{\pi_1M_1}\to\wh{\pi_1M_2}$ and its scalar $\mu$.
By the virtual $1$-domination theorem \cite[Theorem~1.1]{LiuSun2018} there is a finite cover
$p_2:Q\to M_2$ and a map of degree one $Q\to Y$. Choose basepoints and identify $\pi_1Q$ with the
finite index subgroup $(p_2)_*\pi_1Q\leq\pi_1M_2$. Then its closure $\wh{\pi_1Q}$ in $\wh{\pi_1M_2}$
is open and is canonically its profinite completion. Let $\iota_1:\pi_1M_1\to\wh{\pi_1M_1}$ be the
canonical map and set
\[
L=\iota_1^{-1}\bigl(\Phi^{-1}(\wh{\pi_1Q})\bigr).
\]
Then the correspondence between finite index subgroups and open subgroups of the profinite
completion gives $[\pi_1M_1:L]<\infty$ and $\wh L=\Phi^{-1}(\wh{\pi_1Q})$. Let $p_1:P\to M_1$ be the
connected finite cover with $p_{1*}\pi_1 P=L$. As $\Phi(\wh{\pi_1P})=\wh{\pi_1Q}$, restricting
$\Phi$ to $\wh{\pi_1P}$ gives an isomorphism $\Psi:\wh{\pi_1P}\to\wh{\pi_1Q}$. Both $P$ and $Q$ are
closed orientable hyperbolic $3$-manifolds, and $Q$ admits a degree one map to $Y$, so the argument
above shows that the scalar $\lambda$ in $\Psi^*=\lambda(F_0\otimes\hZ)$ belongs to $\{1,-1\}$.

We treat the trivial and nontrivial cases separately. If $H^1(M_2,\Z)=0$, the claimed integral
description is vacuous and the scalar can be taken to be $1$. Otherwise take
$0\neq z\in H^1(M_2,\Z)$. For any finite covering $p:E\to M$ of degree $e$ we have
\[
\operatorname{tr}_p\circ p^*=e\,\id
\qquad\text{on }H^1(M,\Q),
\]
as in \cite{Hatcher2002}. Since $e$ is invertible in $\Q$, the map $e^{-1}\operatorname{tr}_p$ is a
left inverse of $p^*:H^1(M,\Q)\longrightarrow H^1(E,\Q)$. Thus pullback on rational cohomology is
injective. Moreover, $H^1(M_2,\Z)$ is torsion-free, so the nonzero class $z$ remains nonzero in
$H^1(M_2,\Q)$, hence $p_2^*z\neq0$. Now let $j_1=(p_1)_*$ and $j_2=(p_2)_*$ be the inclusions of the
covering groups. Since $\Psi$ is the restriction of $\Phi$, we have
$\Phi\circ\wh{j_1}=\wh{j_2}\circ\Psi$, and since in degree one pullback is precomposition with the
corresponding homomorphism,
\[
\Psi^*\circ p_2^*=p_1^*\circ\Phi^*
:H^1(M_2,\hZ)\longrightarrow H^1(P,\hZ).
\]
Substituting the scalar decompositions of $\Phi^*$ and $\Psi^*$ obtained above gives
\[
\lambda F_0(p_2^*z)=\mu p_1^*F(z)
\qquad\text{in }H^1(P,\hZ).
\]
Since the integral vector $F_0(p_2^*z)$ is nonzero, we can choose an integral basis of $H^1(P,\Z)$
and a coordinate $a\neq0$ of this vector, and then let $b$ be the same coordinate of $p_1^*F(z)$.
Then $\lambda a=\mu b$ in $\hZ$, so $b\neq0$ and
\[
\lambda\mu^{-1}=b/a\in\Q^\times\cap\hZ^\times=\{1,-1\},
\]
where division is taken in $\prod_p\Q_p$. Because $\lambda\in\{1,-1\}$, we obtain $\mu\in\{1,-1\}$,
as required.
\end{proof}


\section{Periodic orbits}
Recall that a homeomorphism $f$ of a closed surface is called \emph{pseudo-Anosov} if there are two
transverse measured singular foliations $(\mathcal F_1,\mu_1)$ and $(\mathcal F_2,\mu_2)$ and a
number $\lambda>1$ such that $f$ preserves each foliation and $f^*\mu_1=\lambda\mu_1$,
$f^*\mu_2=\lambda^{-1}\mu_2$. The foliations have a common finite singular set, with at least three
prongs at each singular point. The measures are transverse measures invariant under homotopies along
leaves, see \cite{FathiLaudenbachPoenaru2012}.

Let $f:S\to S$ be an orientation preserving pseudo-Anosov homeomorphism of a closed orientable
surface of genus at least two, let $M_f$ be its \emph{mapping torus} and $\varphi^s$ its
\emph{suspension flow}. The first return map to $S\times\{0\}$ is $f$. Denote by $G_f=\pi_1(M_f)$
and let $\chi:G_f\to\Z$ be induced by the projection $M_f\to\R/\Z$. We choose $t$ with suspension
coordinate $\chi(t)=1$. If $\Pi=\ker\chi$, then $a=\conj{t}|_\Pi$ gives
$G_f=\Pi\rtimes_a\langle t\rangle$, and $[a]=[f_*^{-1}]$. The fixed point indices and return maps
below concern the forward map $f$, whereas the calculations with the stable letter in
\cref{sec:realization} use $a$. Both mapping classes are pseudo-Anosov. Thurston's hyperbolization
theorem makes $M_f$ closed hyperbolic \cite{Otal1996}, so \cref{thm:normalizer-comparisons} applies.
Given an orbit $O$ of least return period $d\mid m$, we let $a_O$ be the element represented by one
positive traversal, up to conjugacy, and set
\[
\begin{aligned}
b_{O,m}&=a_O^{m/d},\\
w_m(O)&=\sum_{x\in O}\ind(f^m,x)
=d\,\ind(f^m,x_0)
\qquad(x_0\in O).
\end{aligned}
\]
Let $\mathcal O_m(f)$ be the finite collection of orbits of period dividing $m$. At an isolated
fixed point $x$, the local index $\ind(f^m,x)$ is the degree of $z\mapsto(z-f^m(z))/\|z-f^m(z)\|$ on
the boundary of a sufficiently small oriented coordinate disk about $x$, and the coefficient
$w_m(O)$ is the sum of the local indices over the whole orbit.

Two points of $\Fix(f^m)$ are in the same \emph{fixed point class} if there is a path between them
that is homotopic, relative to its endpoints, to its image under $f^m$. The map $f$ permutes these
classes. An orbit of this permutation is a \emph{periodic Nielsen class} at return time $m$. We use
the fixed point and periodic orbit definitions of \cite{Jiang1996}.

\begin{lemma}\label{lem:orbit-roots}
Given the suspension flow of $f:S\to S$ on its mapping torus $M_f$, we have $w_m(O)\neq0$ for every
$O\in\mathcal O_m(f)$. Distinct orbits give distinct conjugacy classes $[b_{O,m}]$ in $G_f$. Every
prime positive (forward traversal) suspension orbit determines, by one positive traversal, an
element that is not a proper power in $G_f$ and has no proper integer root in $\wh G_f$.
\end{lemma}
\begin{proof}
For a pseudo-Anosov iterate, every nonempty fixed point class is a singleton
\cite[Lemma~3.1]{Jiang1996}. Also at a fixed point with $k$ stable prongs, the index is $1-k$ when
the prongs are fixed and is $1$ otherwise. At a regular point we have $k=2$, and at a singular point
$k\geq3$. So every local index is nonzero. Periodic Nielsen classes are the unions of these
singleton point classes under $f$, so they are the individual geometric periodic orbits. The
suspension correspondence between periodic Nielsen classes and conjugacy classes \cite{Jiang1996}
then shows that the classes $[b_{O,m}]$ are distinct.

We next prove the absence of proper roots in the group. Identify the universal cover of the mapping
torus with $\widetilde S\times\R$ using the lifted suspension. Collapsing each lifted flow line
gives the orbit plane $\widetilde S\cong\R^2$. Let $\widetilde\theta:\widetilde M_f\to\R$ be the
lift of the fibration $M_f\to\R/\Z$; every deck transformation $g\in G_f$ satisfies
$\widetilde\theta(gz)=\widetilde\theta(z)+\chi(g)$. Since deck transformations commute with the
lifted suspension flow by uniqueness of lifting, in the coordinates
$\widetilde M_f\cong\widetilde S\times\R$ we have $g(x,s)=\bigl(\bar g(x),s+\chi(g)\bigr)$, where
$\bar g$ is the induced homeomorphism of the orbit plane. Because the return map preserves the
orientation of $S$, each element of $G_f$ induces an orientation preserving homeomorphism of the
orbit plane. Let $b$ be represented by one positive traversal of a prime trajectory, and suppose
that $b=g^r$ for some $r\geq 2$. The element $b$ fixes the point of the orbit plane represented by a
lift of its closed trajectory. Thus $g^r$ has a fixed point. By the Brouwer plane translation
theorem, an orientation preserving fixed point free homeomorphism of $\R^2$ has no periodic point
\cite{Franks1992}. Therefore $g$ has a fixed point in the orbit plane. It preserves a lifted flow
line and translates it in the positive direction, because $r\chi(g)=\chi(b)>0$, so $g$ represents a
positive closed trajectory. Its $r$-fold return element is $g^r=b$, so this trajectory and the one
represented by $b$ determine the same periodic Nielsen class. Each such class consists of one
geometric periodic orbit, so these trajectories are the same orbit. But $\chi(g)=\chi(b)/r>0$ would
give that orbit a shorter positive return time. This contradicts primeness, and $b$ is not a proper
power in $G_f$.

Now for an element $b$ that is not a proper power, the discrete centralizer is $\langle b\rangle$.
The profinite centralizer comparison in \cref{thm:normalizer-comparisons} implies
\[
C_{\wh G_f}(b)=\cl{\langle b\rangle}\cong\hZ.
\]
On the other hand, if $z^r=b$ in $\wh G_f$ with $r\geq2$, then $z$ centralizes $b$, so $z=b^\xi$ for
some $\xi\in\hZ$ and $r\xi=1$ in $\hZ$, which is impossible modulo any prime divisor of $r$.
\end{proof}

\begin{lemma}
\label{lem:unit-separation}
Let $G$ be conjugacy separable and let $\chi:G\to\Z$ be a homomorphism. A finite family of pairwise
nonconjugate elements $c_1,\ldots,c_s$ such that $\chi(c_i)=m\neq0$ can be separated in a finite
quotient into distinct classes under conjugacy and profinite unit powers. The quotient may refine
any given finite family of finite quotients.
\end{lemma}
\begin{proof}
If a pair could not be separated, then for every open normal $U\normalo\wh G$ the closed subset
\[
E_U=\{(q,\nu)\in\wh G\times\hZ^\times:
c_j\equiv q c_i^\nu q^{-1}\pmod U\}
\]
would be nonempty. If the image of $c_i$ in the finite quotient has order $N$ then every exponent
$a$ with $\gcd(a,N)=1$ lifts to an element $\nu\in\hZ^\times$. Indeed, under
$\hZ^\times=\prod_p\Z_p^\times$, choose $\nu_p=a$ for $p\mid N$ and $\nu_p=1$ otherwise. The Chinese
remainder theorem gives $\nu\equiv a\pmod N$, so the two exponents give the same power in that
quotient. For any finite collection $U_1,\ldots,U_r$ of open normal subgroups of $\wh G$, the
subgroup $V=U_1\cap\cdots\cap U_r$ is open and normal, so $E_V$ is nonempty by assumption; and since
a congruence modulo $V$ is equivalent to the corresponding congruences modulo all the $U_k$, we get
$E_V=\bigcap_{k=1}^r E_{U_k}\neq\varnothing$. Thus the family $(E_U)_U$ has the finite intersection
property. Since the sets $E_U$ are closed in the compact space $\wh G\times\hZ^\times$, there exists
$(q,\nu)\in\bigcap_{U\normalo\wh G}E_U$. For this pair, we have
\[
c_j^{-1}qc_i^\nu q^{-1}\in U
\qquad\text{for every }U\normalo\wh G.
\]
The intersection of all open normal subgroups of $\wh G$ is $\{1\}$, so $c_j=qc_i^\nu q^{-1}$ in
$\wh G$. Now applying $\wh\chi$ gives
\[
m=\wh\chi(c_j)
=\wh\chi(q)+\nu\wh\chi(c_i)-\wh\chi(q)
=\nu m.
\]
So $m(\nu-1)=0$ in $\hZ$, and for every prime $p$ we have $m(\nu_p-1)=0$ in $\Z_p$. As $m\neq0$ and
$\Z_p$ is an integral domain, $\nu_p=1$ for every $p$, that is, $\nu=1$. This implies the elements
$c_i,c_j$ are conjugate in $\wh G$, and hence their images are conjugate in every finite quotient of
$G$. This contradicts conjugacy separability.

Finally, for each pair $i<j$, choose a finite quotient $\theta_{ij}:G\twoheadrightarrow F_{ij}$ in
which the images of $c_i,c_j$ are not related by conjugacy and profinite unit powers. Let
$\rho_\ell:G\twoheadrightarrow Q_\ell$, $1\leq\ell\leq r$, be the given finite quotients. Set
\[
\Theta:G\longrightarrow
\prod_{i<j}F_{ij}\times\prod_{\ell=1}^r Q_\ell,
\qquad
\Theta(g)=
\bigl((\theta_{ij}(g))_{i<j},(\rho_\ell(g))_\ell\bigr),
\]
and set $F=\im\Theta$. Then $G\twoheadrightarrow F$ is a finite quotient through which every given
$\rho_\ell$ factors. Now suppose that, for some $i<j$ we have $\Theta(c_j)=h\Theta(c_i)^\nu h^{-1}$
for $h\in F,\ \nu\in\hZ^\times$. Then projection to $F_{ij}$ gives the same relation between
$\theta_{ij}(c_j)$ and $\theta_{ij}(c_i)$ which contradicts the choice of $\theta_{ij}$. Thus this
one finite quotient separates every pair as required.
\end{proof}

We return to $f$. Let $O\in\mathcal O_m(f)$ have least period $d\mid m$, let $x\in O$, let
$\gamma:G_f\twoheadrightarrow Q$ be a finite quotient, and let $\rho:Q\to\GL(V)$ be a rational
representation on a finite dimensional $\Q$-vector space $V$. The orbit $O$ contains $d$ distinct
points fixed by $f^m$. By \cref{lem:orbit-roots}, these points lie in $d$ singleton fixed point
classes of $f^m$. Their local indices are equal because $f^j$ conjugates the germ of $f^m$ at $x$ to
the germ of $f^m$ at $f^j(x)$. Their corresponding $m$-return elements are conjugate to $b_{O,m}$.
Since the Lefschetz fixed point sum with local coefficients is initially summed over the individual
points of $\Fix(f^m)$, the $d$ points belonging to $O$ contribute
\[
\begin{aligned}
&\sum_{x\in O}\ind(f^m,x)\,
\tr\rho\bigl(\gamma(a_{x,m})\bigr)\\
&\qquad =
d\,\ind(f^m,x_0)\,
\tr\rho\bigl(\gamma(b_{O,m})\bigr)\\
&\qquad =
w_m(O)\tr\rho\bigl(\gamma(b_{O,m})\bigr),
\qquad x_0\in O,
\end{aligned}
\]
where $a_{x,m}$ represents the positive $m$-return loop at $x$.

\begin{proposition}\label{prop:orbit-sums}
Suppose $\Phi:\wh G_{f_1}\to\wh G_{f_2}$ is an isomorphism with $\wh\chi_2\Phi=\wh\chi_1$. Let
$\gamma_i:G_{f_i}\twoheadrightarrow Q$ be epimorphisms with $\wh\gamma_1=\wh\gamma_2\Phi$. For every
rational conjugacy class $\omega$ in $Q$, equivalently every class of elements that generate
conjugate cyclic subgroups, and every $m\geq1$ we have
\begin{equation}\label{eq:orbit-sum}
\sum_{\gamma_1(b_{O,m})\in\omega} w_m(O)
=
\sum_{\gamma_2(b_{O,m})\in\omega} w_m(O).
\end{equation}
\end{proposition}
\begin{proof}
By \cref{thm:closed-integral}, the cohomology correspondence may be normalized so that its scalar is
one and it sends $\chi_2$ to $\chi_1$. Let $\rho:Q\to\GL(V)$ be a finite dimensional rational
representation, and let $\mathcal V_i$ be the local coefficient system of $\Q$-vector spaces on
$M_{f_i}$ associated to $\rho\circ\gamma_i:\pi_1(M_{f_i})\to\GL(V)$. After identifying the fiber at
a chosen basepoint with $V$, the monodromy along a loop representing $g\in G_{f_i}$ is
$\rho(\gamma_i(g))$. For $x\in S_i$, let
\[
\eta_{i,x}:[0,1]\longrightarrow M_{f_i},
\qquad \eta_{i,x}(s)=[x,s],
\]
be the positive suspension segment. Its endpoints are $x$ and $f_i(x)$ in the surface. The local
coefficient system assigns to this path an isomorphism
$\theta_{i,x}:(\mathcal V_i)_x\longrightarrow(\mathcal V_i)_{f_i(x)}$. The map $f_i$, together with
these coefficient isomorphisms, induces
$T_{i,j}:H_j(S_i,\mathcal V_i|_{S_i}) \longrightarrow H_j(S_i,\mathcal V_i|_{S_i})$. The $T_{i,j}$
are the homology maps induced by the positive first return with local coefficients. Define the
normalized torsion function \cite{Liu2023} by
\[
Z_i^\rho(t)
=\prod_{j=0}^{2}
\det(I-tT_{i,j})^{(-1)^{j+1}}.
\]
The mapping torus determinant formula \cite[Lemma~7.7]{Liu2023}, together with the alternating
product formula for twisted Reidemeister torsion \cite[equation~(2.2)]{Liu2023}, identifies
$Z_i^\rho(t)$ with the twisted Reidemeister torsion for $(\rho\circ\gamma_i,\chi_i)$ up to a factor
$ct^k$, where $c\in\Q^\times$ and $k\in\Z$. The determinant formula uses the forward return maps
with the coefficient isomorphisms defined above. Liu's torsion comparison
\cite[Theorem~7.2]{Liu2023} therefore identifies $Z_1^\rho$ and $Z_2^\rho$ up to such a factor. Both
functions have order zero and value one at $t=0$, so $k=0$, $c=1$ and $Z_1^\rho(t)=Z_2^\rho(t)$.
Interpret $\log Z_i^\rho(t)$ as the formal logarithm in $\Q[[t]]$. The determinant identity and the
twisted Lefschetz trace formula \cite[Section~1.6 and Theorem~1.2]{Jiang1996} give
\begin{equation}\label{eq:log-torsion-fixed-point}
\log Z_i^\rho(t)
=\sum_{m\geq1}\frac{t^m}{m}
\sum_{x\in\Fix(f_i^m)}
\ind(f_i^m,x)\,
\tr\rho\bigl(\gamma_i(a_{x,m})\bigr),
\end{equation}
where $a_{x,m}$ is the positive $m$-return element of the singleton fixed point class of $x$. By the
computation preceding the proposition, an orbit $O\in\mathcal O_m(f_i)$ contributes
$w_m(O)\tr\rho\bigl(\gamma_i(b_{O,m})\bigr)$ to the inner sum in \eqref{eq:log-torsion-fixed-point}.
Hence by equality of the logarithmic coefficients we get
\begin{equation}\label{eq:character-orbit-trace}
\sum_{O\in\mathcal O_m(f_1)}w_m(O)
\tr\rho\bigl(\gamma_1(b_{O,m})\bigr)
=
\sum_{O\in\mathcal O_m(f_2)}w_m(O)
\tr\rho\bigl(\gamma_2(b_{O,m})\bigr).
\end{equation}
By \cite[Lemma~8.4]{Liu2023}, every function $Q\to\Q$ that is constant on rational conjugacy classes
is a rational linear combination of characters of finite dimensional rational representations of
$Q$; in particular this holds for the indicator function of $\omega$. Taking that linear combination
in \eqref{eq:character-orbit-trace} proves \eqref{eq:orbit-sum}.
\end{proof}
For $i=1,2$, let $O_i$ be a periodic orbit of $f_i$ of least period $m$. An element $b_i\in G_{f_i}$
representing one positive traversal of its closed suspension trajectory is defined after choosing a
path from the basepoint to that trajectory. Changing this path conjugates $b_i$ in $G_{f_i}$, and in
every case
\[
\chi_i(b_i)=m.
\]

\begin{theorem}
\label{thm:orbit-correspondence}
Suppose
\[
\Phi:\wh G_{f_1}\xrightarrow{\cong}\wh G_{f_2}
\]
is a continuous isomorphism satisfying
\[
\wh\chi_2\circ\Phi=\epsilon\wh\chi_1,
\qquad \epsilon\in\{1,-1\}.
\]
For every prime periodic orbit $O_1$ of $f_1$ of least period $m$, there is a unique prime periodic
orbit $O_2$ of $f_2$ of least period $m$ with the following property: for any choices of elements
$b_i\in G_{f_i}$ representing one positive traversal of the corresponding suspension trajectories,
there exists $q\in\wh G_{f_2}$ such that
\begin{equation}\label{eq:orbit-exact}
\Phi(b_1)=q b_2^\epsilon q^{-1}.
\end{equation}
The resulting correspondences on prime periodic orbits associated to $\Phi$ and $\Phi^{-1}$ are
mutually inverse and preserve least periods.
\end{theorem}
\begin{proof}
First suppose $\epsilon=1$. Fix $m$ and write $\mathscr B_i(m)=\{b_{O,m}:O\in\mathcal O_m(f_i)\}$.
By \cref{lem:orbit-roots}, the elements in each finite family are pairwise nonconjugate, have
suspension coordinate $m$, and have nonzero weights $w_m(O)=d\,\ind(f_i^m,x)$, $x\in O$.

Let $U\normalo\wh G_{f_2}$ be arbitrary. Apply \cref{lem:unit-separation} to obtain finite quotients
$\theta_i:G_{f_i}\twoheadrightarrow F_i$ that separate the elements of $\mathscr B_i(m)$ under
conjugacy and profinite unit powers. Set $K_i=\ker\wh\theta_i$, open normal subgroups, and choose
the second quotient so that $K_2\subseteq U$; the reason is that the common quotient constructed
below then maps onto $\wh G_{f_2}/U$, so any orbit matching obtained in it also holds modulo $U$,
which is what the compactness argument at the end needs for every $U$. Define
\[
N_2=K_2\cap\Phi(K_1),
\qquad
N_1=\Phi^{-1}(N_2).
\]
Then $N_i\subseteq K_i$, $N_2\subseteq U$, and $\Phi(N_1)=N_2$, so the isomorphism $\Phi$ induces
\[
\cl\Phi:
\wh G_{f_1}/N_1
\xrightarrow{\cong}
\wh G_{f_2}/N_2,
\qquad
gN_1\longmapsto\Phi(g)N_2.
\]
We identify these finite groups using $\cl\Phi$ and write $Q=\wh G_{f_2}/N_2$. Each $G_{f_i}$ is
dense in its completion and therefore surjects onto the corresponding finite quotient, so $Q$ is a
common finite quotient $\gamma_i:G_{f_i}\twoheadrightarrow Q$ of both discrete groups, and the
continuous extensions satisfy $\wh\gamma_1=\wh\gamma_2\Phi$.

The quotient $\gamma_i$ refines $\theta_i$. So if two distinct elements of $\mathscr B_i(m)$ had
images in the same rational conjugacy class of $Q$, then their images in $F_i$ would also be related
by conjugacy and a profinite unit power, which contradicts the choice of $\theta_i$. It follows that
each rational conjugacy class of $Q$ contains the image of at most one element of $\mathscr B_i(m)$
for each $i$.

Let $O_1$ be the orbit represented by $b_1$. In the rational conjugacy class containing the image of
$b_1$, the left side of \eqref{eq:orbit-sum} is the nonzero integer $w_m(O_1)$ alone. Hence the
right side contains exactly one orbit $O_U\in\mathcal O_m(f_2)$. Since a generator of the same
cyclic subgroup differs by an exponent prime to its finite order, and such an exponent lifts to a
unit of $\hZ$, there are $\nu_U\in\hZ^\times$ and $q_U\in\wh G_{f_2}$ such that
\[
\Phi(b_1)\equiv q_U b_{O_U,m}^{\nu_U}q_U^{-1}\pmod U.
\]

Now for every open normal $U$, denote by $E_U$ the collection of all triples
$(O,\nu,q)\in \mathcal O_m(f_2)\times\hZ^\times\times\wh G_{f_2}$ satisfying the congruence above.
It is a nonempty closed subset of a compact space. Because for finitely many open normal subgroups,
the construction at their intersection produces a triple belonging to all the sets $E_U$ in
question, the family $(E_U)$ has the finite intersection property. By compactness there is a triple
$(O,\nu,q)$ satisfying
\begin{equation}\label{eq:orbit-unit-equality}
\Phi(b_1)=q b_{O,m}^{\nu}q^{-1}
\end{equation}
in the completion.

Applying the suspension coordinate to \eqref{eq:orbit-unit-equality}, we get $m=\nu m$ in $\hZ$, and
since multiplication by the nonzero integer $m$ is injective on $\hZ$, $\nu=1$. Let $d\mid m$ be the
least period of $O$, so that $b_{O,m}=a_O^{m/d}$ and $\Phi(b_1)=q a_O^{m/d}q^{-1}$. Setting
$z=\Phi^{-1}(q a_Oq^{-1})$ we get $z^{m/d}=\Phi^{-1}(q a_O^{m/d}q^{-1})=b_1$, so if $d<m$ this is a
proper integer root of $b_1$ in $\wh G_{f_1}$, contradicting \cref{lem:orbit-roots}. Hence $d=m$,
and $b_{O,m}=a_O$ is a prime orbit element of least period $m$.

To prove uniqueness, suppose that two orbits $O,O'$ of least period $m$ satisfy the required
conjugacy relation for the same $b_1$. Using the representatives in $\mathscr B_2(m)$, we then have
\[
\Phi(b_1)
=q b_{O,m}q^{-1}
=q' b_{O',m}(q')^{-1}
\]
for some $q,q'\in\wh G_{f_2}$. Consequently,
\[
b_{O',m}
=\bigl((q')^{-1}q\bigr)b_{O,m}
\bigl((q')^{-1}q\bigr)^{-1}.
\]
By \cref{lem:unit-separation}, there is a finite quotient in which the elements of $\mathscr B_2(m)$
associated to distinct orbits have nonconjugate images. Applying this quotient to this equality, we
get $O'=O$. Thus $O$ is unique.

If $\epsilon=-1$, reverse the second suspension flow and replace $\chi_2$ by $-\chi_2$. The return
map is then $f_2^{-1}$, whose periodic orbits are the same sets as those of $f_2$, with the same
least periods. A positive traversal for the reversed flow represents $b_2^{-1}$, where $b_2$
represents a positive traversal for the original flow. The case already proved therefore yields
\[
\Phi(b_1)=q b_2^{-1}q^{-1},
\]
with the same existence and uniqueness conclusion for the orbit $O_2$.

We verify that the conclusion is independent of the chosen based representatives. Suppose
\[
b_i'=g_i b_i g_i^{-1},
\qquad g_i\in G_{f_i},
\]
and that \eqref{eq:orbit-exact} holds. Set
\[
q'=\Phi(g_1)qg_2^{-1}.
\]
Then
\[
\begin{aligned}
\Phi(b_1')
&=\Phi(g_1)\Phi(b_1)\Phi(g_1)^{-1}\\
&=q'(b_2')^\epsilon(q')^{-1}.
\end{aligned}
\]
Therefore the defining condition depends only on the two periodic orbits.

Finally, \eqref{eq:orbit-exact} implies
\[
\Phi^{-1}(b_2)
=\Phi^{-1}(q)^{-1}
b_1^\epsilon
\Phi^{-1}(q).
\]
The inverse isomorphism satisfies
\[
\wh\chi_1\circ\Phi^{-1}
=\epsilon\wh\chi_2.
\]
Applying the uniqueness conclusion to $\Phi^{-1}$ therefore sends $O_2$ back to $O_1$. Interchanging
the two sides proves the other composition is also the identity, so the two orbit correspondences
are mutually inverse and preserve least periods.
\end{proof}


\section{Fibered classes}
Let $\mathcal G$ be a finite directed graph with directed edge set $A(\mathcal G)$ and endpoint maps
$s,t$. We define its two sided edge shift by
\[
\Sigma=\{(e_j)_{j\in\Z}\in A(\mathcal G)^\Z:t(e_j)=s(e_{j+1})\text{ for all }j\},
\qquad \sigma((e_j)_j)=(e_{j+1})_j,
\]
with the subspace topology from the product of finite discrete sets. The graph is \emph{strongly
connected} if every ordered pair of vertices is joined by a directed path. An \emph{admissible word}
is a finite nonempty directed edge path. It is closed if its last terminal vertex equals its first
initial vertex. Repeating a closed word $W$ of length $k$ defines a point $\omega_W\in\Sigma$ with
$\sigma^k\omega_W=\omega_W$. For a continuous map $\pi:\Sigma\to S$ with $\pi\sigma=f\pi$, a
\emph{periodic realization of $W$} means the orbit of $x_W=\pi(\omega_W)$, traversed for $k$ return
times. We write $c_W^{(k)}$ for this oriented closed suspension curve and
$[c_W^{(k)}]\in H_1(M_f,\Z)$ for its homology class. See \cite{ParryPollicott1990}.

\begin{lemma}\label{lem:markov-coding}
Let $S$ be a connected closed orientable surface of genus at least two, and let $f:S\to S$ be an
orientation preserving pseudo-Anosov homeomorphism. Then there are a finite strongly connected
directed graph $\mathcal G$, its two sided edge shift $(\Sigma,\sigma)$, and a continuous surjective
map with finite fibers
\[
\pi:\Sigma\longrightarrow S
\]
such that $\pi\sigma=f\pi$. In addition, every closed admissible word has a periodic realization.
Moreover, if $f^m(x)=x$, then there are $\omega\in\pi^{-1}(x)$ and $r\geq1$ with
$\sigma^{mr}(\omega)=\omega$.
\end{lemma}
\begin{proof}
Choose a finite Markov partition $R_1,\ldots,R_q$ by good rectangles, that is, closed rectangles
which are embedded and homeomorphic to closed disks, with sides in stable and unstable leaves, and
which satisfy the one-crossing condition of \cite{FathiShub1979}. For each pair $(i,j)$, put
$U_{ij}=\Int R_i \cap f^{-1}(\Int R_j)$, the set of points in $\Int R_i$ whose image lies in
$\Int R_j$. If $U_{ij}$ is nonempty, the one-crossing condition makes it connected. The Markov
rectangle structure then shows that $Q_e=\cl{U_{ij}}$ is a closed topological rectangle, hence a
closed disk. We introduce one directed edge $e:i\to j$ for this transition and associate the disk
$Q_e$ to it. If $U_{ij}$ is empty, we introduce no such edge, even if $R_i\cap f^{-1}(R_j)$ is
nonempty. Boundary-only intersections give no edge, and the rectangle transition matrix is binary.
Its positive power in \cite[Lemma~1]{FathiShub1979} makes $\mathcal G$ strongly connected. For an
admissible edge sequence $\omega=(e_j)_{j\in\Z}$, set $i_j=s(e_j)$. The nonempty cylinder rectangles
\[
C_N(\omega)=\cl{\bigcap_{j=-N}^{N}
f^{-j}(\Int R_{i_j})}
\]
are nested and shrink to a point. Their intersection defines a continuous surjection $\pi$ with
$\pi\sigma=f\pi$ by \cite{FathiShub1979}.

For the finiteness of the fibers $\pi^{-1}(x)$ we follow \cite[Lemmas~9--10 and
Corollary~1]{CruzDiaz2025}. At a point $x$, the local stable and unstable prongs bound finitely many
open sector germs, four at a regular point and $2k$ at a $k$-prong singularity. Their images under
every iterate are again sector germs. Each such image also lies in the interior of a unique
partition rectangle, since its boundary consists of foliation arcs. So a sector determines at most
one complete rectangle itinerary and so by the one-crossing condition, at most one edge sequence.
Representatives of the germ are taken small enough to avoid the finitely many partition boundary
arcs that do not pass through the point. Conversely, suppose $\pi(\omega)=x$. Then each
$C_N(\omega)$ contains $x$ and has nonempty interior bounded by stable and unstable arcs. So its
interior contains an open sector germ at $x$. Note that the sets of such germs are finite, nonempty,
and nested as $N$ increases. Hence a germ in their intersection has itinerary $(i_j)_{j\in\Z}$, so
it determines $\omega$. Therefore every name of $x$ comes from a sector, and $\pi^{-1}(x)$ is
finite. This also proves finiteness at boundaries and singularities.

Repeating a closed admissible word produces a periodic sequence and hence a periodic realization. If
$f^m(x)=x$, then both $\sigma^m$ and $\sigma^{-m}$ preserve the finite nonempty set $\pi^{-1}(x)$.
Thus $\sigma^m$ permutes that set. Some positive power of the permutation fixes a name, which
implies the final claim.
\end{proof}

\begin{lemma}\label{lem:markov-homology}
Let $M_f$ be the suspension of $f$. Fix the Markov rectangles $R_1,\ldots,R_q$, the closed
rectangles $Q_e$, the directed graph $\mathcal G$, and the mapping $\pi$ constructed in the proof of
\cref{lem:markov-coding}. There are classes $h(e)\in H_1(M_f,\Z)$, one for each directed edge $e$ of
$\mathcal G$, such that, for every closed admissible word $W=e_1\cdots e_k$ and every periodic
realization traversed for exactly $k$ return times, we have
\begin{equation}\label{eq:integral-symbols}
[c_W^{(k)}]=\sum_{j=1}^{k}h(e_j)
\quad\text{in }H_1(M_f,\Z).
\end{equation}
The formula remains valid when the realization meets a boundary of a Markov rectangle or a
singularity of the invariant foliations.
\end{lemma}
\begin{proof}
We use the rectangles and transition disks of \cref{lem:markov-coding}. Fix a basepoint $o\in S$, a
point $r_i\in R_i$ for each $i$, and a path $\gamma_i$ from $o$ to $r_i$. For $e:i\to j$, choose
$z_e\in Q_e$ and paths $\alpha_e$ from $r_i$ to $z_e$ in $R_i$ and $\beta_e$ from $f(z_e)$ to $r_j$
in $R_j$. Write $\tau_z$ for the positively oriented unit suspension segment from $z$ to $f(z)$. We
define
\[
h(e)=[\gamma_i\alpha_e\tau_{z_e}\beta_e\gamma_j^{-1}]
\in H_1(M_f,\Z).
\]
Let $W=e_1\cdots e_k$ be closed and let $x_1,\ldots,x_k$ be its successive realization points, with
$x_{k+1}=x_1$ and $f(x_\ell)=x_{\ell+1}$. Since $x_\ell\in Q_{e_\ell}$ and $Q_{e_\ell}$ is a disk,
we can choose a path $a_\ell$ in $Q_{e_\ell}$ from $z_{e_\ell}$ to $x_\ell$. The map
\[
H_\ell(u,t)=\varphi^t(a_\ell(u))
\qquad((u,t)\in[0,1]^2)
\]
deforms the chosen suspension segment into the realized one. Its endpoint paths remain in
$R_{s(e_\ell)}$ and $R_{t(e_\ell)}$. Next choose paths $\delta_\ell$ from $r_{s(e_\ell)}$ to
$x_\ell$ in $R_{s(e_\ell)}$, with $\delta_{k+1}=\delta_1$. Contractibility of both endpoint
rectangles shows that the loop defining $h(e_\ell)$ is based homotopic to
$\gamma_{s(e_\ell)}\delta_\ell\tau_{x_\ell} \delta_{\ell+1}^{-1}\gamma_{t(e_\ell)}^{-1}$. Let
$L_\ell$ denote the based loop defining $h(e_\ell)$, and set
\[
A_\ell=\gamma_{s(e_\ell)}\delta_\ell,
\qquad A_{k+1}=A_1.
\]
Then the based homotopies above give
\[
L_1\cdots L_k
\simeq
A_1(\tau_{x_1}\cdots\tau_{x_k})A_1^{-1},
\]
because each intervening path $A_{\ell+1}^{-1}A_{\ell+1}$ contracts relative to its endpoints. The
middle concatenation is the realized closed suspension curve $c_W^{(k)}$. Passing to the
abelianization $\pi_1(M_f,o)^{\mathrm{ab}}\cong H_1(M_f,\Z)$ proves \eqref{eq:integral-symbols} in
the full integral homology group, including torsion.

The paths $a_\ell$ lie in the closed transition disks, and the comparisons of endpoint paths take
place in the closed contractible rectangles. Since the suspension homotopies are continuous at every
point, the argument also applies to realization points on partition boundaries or at singularities
of the invariant foliations.
\end{proof}

\begin{lemma}\label{lem:zero-weight-word}
Let $\mathcal G$ be a finite strongly connected directed graph with an integer weight $w(e)$ on each
directed edge. For $W=e_1\cdots e_k$, we set $w(W)=\sum_{j=1}^k w(e_j)$. If $\mathcal G$ has closed
admissible words of positive and negative total weight, then it has a closed admissible word of
total weight zero.
\end{lemma}
\begin{proof}
Choose a vertex $v$. Let $W_+$ and $W_-$ be closed words of positive and negative weight, based at
vertices $v_+$ and $v_-$. By strong connectivity there are directed paths from $v$ to $v_+$ and
back, and from $v$ to $v_-$ and back. The connecting paths have fixed total weights. Since
$w(W_+)>0$ and $w(W_-)<0$, repeating $W_+$ sufficiently many times makes the resulting closed word
based at $v$ have positive weight, and repeating $W_-$ sufficiently many times makes the resulting
closed word based at $v$ have negative weight. We obtain closed words $P_+$ and $P_-$ based at $v$
with weights
\[
w(P_+)=p>0,
\qquad
w(P_-)=-q<0.
\]
Then the closed word $P_+^qP_-^p$ has weight $qp-pq=0$.
\end{proof}

\begin{lemma}\label{lem:zero-word}
Let $S$ be a connected closed orientable surface of genus at least two, and let $f:S\to S$ be an
orientation preserving pseudo-Anosov homeomorphism. Equip its mapping torus $M=M_f$ with the
suspension flow, with positive direction given by increasing time. Suppose that $\eta\in H^1(M,\Z)$
has opposite signs on two positive periodic orbits $c_+,c_-$. Then the flow has a positive prime
orbit with representative $b\in\pi_1(M)$ such that $\eta(b)=0$, and $b$ is not a proper power in
$\pi_1(M)$.
\end{lemma}
\begin{proof}
Apply \cref{lem:markov-coding} to periodic points on $c_+$ and $c_-$. After passing to positive
multiples of their return periods, we obtain periodic admissible words whose realizations traverse
$c_+$ and $c_-$, respectively, $r_+$ and $r_-$ times, where $r_+,r_-\geq1$. Their homology classes
are therefore $r_+[c_+]$ and $r_-[c_-]$. By additivity of $\eta$,
\[
\eta(r_+[c_+])=r_+\eta([c_+])>0,
\qquad
\eta(r_-[c_-])=r_-\eta([c_-])<0.
\]
We regard $\eta\in H^1(M,\Z)$ as the homomorphism $H_1(M,\Z)\longrightarrow\Z$,
$a\longmapsto\langle\eta,a\rangle$. For each directed edge $e$, we define its integer weight by
$w(e)=\langle\eta,h(e)\rangle$ where $h(e)$ is from \cref{lem:markov-homology}. By
\cref{lem:zero-weight-word}, there is a closed admissible word of total weight zero, and its
periodic realization has zero evaluation under $\eta$ by \eqref{eq:integral-symbols}.

Let $b$ be the element represented by one positive traversal of the resulting prime trajectory. The
original realization traverses this prime trajectory a positive integer number of times. As $\Z$ has
no torsion, the zero evaluation under $\eta$ forces $\eta(b)=0$. Finally by \cref{lem:orbit-roots},
this element $b$ is not a proper power in the suspension group.
\end{proof}

\begin{theorem}\label{thm:two-fibers}
Let $M$ be a closed orientable hyperbolic $3$-manifold and set $G=\pi_1(M)$. There is a standard
characteristic subgroup $K=K_m(G)$ such that the corresponding cover $p_K:M_K\to M$ admits linearly
independent primitive fibered classes
\[
\chi_{0,K},\chi_{1,K}\in H^1(M_K,\Z)
\]
together with a positive prime $\chi_{0,K}$-orbit element $b_K\in K$ such that $\chi_{1,K}(b_K)=0$,
and $b_K$ is not a proper power in $K$.
\end{theorem}
\begin{proof}
We first construct a fibered compact special cover with first Betti number at least two. By virtual
fibering, we can choose a finite cover $p_f:M_f\to M$ with group $G_f=\pi_1(M_f)\leq G$ and a
primitive fibered class $\chi_f\in H^1(M_f,\Z)$ \cite{Agol2008,Agol2013,Wise2021}. By virtual
specialness, we can choose a finite index compact special subgroup $G_v\leq G$
\cite{Agol2013,Wise2021}. Replace both by their intersection $G_s=G_f\cap G_v$. Let $p_s:M_s\to M$
be the corresponding cover. Since $G_s$ has finite index in $G_f$, the image $\chi_f(G_s)$ is
$d_s\Z$ for some $d_s>0$, and
\[
\chi_s=d_s^{-1}\chi_f|_{G_s}:G_s\twoheadrightarrow\Z
\]
is primitive and fibered. Thus $M_s$ is again a surface bundle whose return map is a lift of a
positive power of the original pseudo-Anosov return map. Having finite index in the compact special
group $G_v$, the group $G_s$ is compact special.

Now choose two noncommensurable loxodromic elements $u,v\in G_s$. Their cyclic subgroups are
quasiconvex in the word-hyperbolic group $G_s$ and have trivial intersection. The subgroups
generated by high powers are also quasiconvex, since they have finite index in these cyclic groups.
In the word metric, all nonidentity elements of $\langle u^N\rangle$ and $\langle v^N\rangle$ have
arbitrarily large length as $N$ increases. We may therefore apply the short-element condition in
\cite[Theorem~1]{Gitik1999} to these two cyclic subgroups. For sufficiently large $N$ it yields
\[
F=\langle u^N,v^N\rangle
=\langle u^N\rangle*\langle v^N\rangle\cong F_2,
\]
and also $F$ is quasiconvex. Then the quasiconvex virtual retraction theorem provides a finite index
subgroup $G_0\leq G_s$ containing $F$ and a retraction $r:G_0\twoheadrightarrow F$
\cite{HaglundWise2008,Wise2021}. Let $p_0:M_0\to M_s$ be the corresponding cover. The retraction
induces a surjection $H_1(G_0,\Q)\to H_1(F,\Q)\cong\Q^2$, so $b_1(M_0)\geq2$. If $d_0>0$ generates
$\chi_s(G_0)$, then set $\chi_0=d_0^{-1}\chi_s|_{G_0}$. This is a primitive fibered class on $M_0$.

The compact special group $G_0$ embeds in a right-angled Artin group $A$ \cite{HaglundWise2008}. By
\cite[Corollary~2.3]{Agol2008}, this $A$ has a finite index RFRS subgroup $A_R$. Set
$G_R=G_0\cap A_R$. Then $G_R$ has finite index in $G_0$ and is RFRS because the RFRS condition is
inherited by subgroups \cite{Agol2008}. Let $p_R:M_R\longrightarrow M_0$ be the corresponding cover.
If $d_R>0$ generates $\chi_0(G_R)$, define
\[
\chi_R=d_R^{-1}\chi_0|_{G_R}:G_R\twoheadrightarrow\Z.
\]
This is the primitive class of the lift of the first fibration. Pullback in rational cohomology is
injective for a finite cover, so
\[
b_1(M_R)\geq b_1(M_0)\geq2.
\]

Next we find an integral class with positive and negative orbit evaluations. Choose a connected
fiber $S_R$ for the fibration represented by $\chi_R$, and write $f:S_R\to S_R$ for its
pseudo-Anosov return map. We use the mapping torus description
\[
M_R=(S_R\times\R)/
\bigl((x,s+1)\sim(f(x),s)\bigr),
\]
in which the projection $\theta_R([x,s])=s\pmod{\Z}$ represents $\chi_R$. Define the suspension flow
by
\[
\varphi_R^t([x,s])=[x,s+t].
\]
The first return map on $S_R\times\{0\}$ is then $f$, and the first return time is one. We call this
the $\chi_R$-suspension flow.

If $f^m(x)=x$, then following this flow from $[x,0]$ for time $m$ traces a closed trajectory $c_x$.
We orient it by increasing flow time. Its projection to the base circle has degree $m$, so
$\chi_R([c_x])=m>0$. Let $D\subset H_1(M_R,\R)$ be the closed cone generated by the homology classes
of positive closed trajectories. Let $C$ be the closure of the Thurston fibered cone containing
$\chi_R$. Fried's theorem \cite{Fried1982,FathiLaudenbachPoenaru2012} says that
\[
D=C^*
=\{v\in H_1(M_R,\R):\langle\xi,v\rangle\geq0
\text{ for every }\xi\in C\}.
\]
In fact, the Thurston norm $\|\cdot\|_T$ on $H^1(M_R,\R)$ is a norm \cite{Thurston1986}. Since a
supporting linear functional $\ell$ for the chosen fibered face satisfies $\ell(\xi)=\|\xi\|_T$ for
every $\xi\in C$, $\ell$ is strictly positive on $C\setminus\{0\}$, and $C\cap(-C)=\{0\}$. Choose
any Euclidean norm on $H^1(M_R,\R)$. The intersection of $C$ with its unit sphere is compact, so
$\ell$ has a positive minimum there, and every sufficiently small perturbation of $\ell$ is still
nonnegative on $C$. Identifying linear functionals on $H^1(M_R,\R)$ with $H_1(M_R,\R)$, this shows
that $\Int D$ is nonempty. Choose a nonzero rational point $x\in\Int D$. Because $b_1(M_R)\geq2$,
the rational annihilator of $x$ contains a nonzero integral class, and dividing this class by the
greatest common divisor of its coordinates we get a primitive class $\eta\in H^1(M_R,\Z)$ with
$\eta(x)=0$. Choose $y\in H_1(M_R,\R)$ with $\eta(y)>0$; then for all sufficiently small $t>0$, both
$x+ty$ and $x-ty$ lie in $D$, and their evaluations under $\eta$ have opposite signs. If every
positive closed trajectory had nonnegative evaluation, then every finite nonnegative sum of their
homology classes would have nonnegative evaluation, as would every point of its closure $D$. This
contradicts $\eta(x-ty)<0$. The same argument with the signs reversed rules out nonpositive
evaluation on every positive closed trajectory. There are therefore positive closed trajectories of
both strict signs. Replacing each by its underlying prime trajectory divides the evaluation by a
positive integer and preserves its sign, so we obtain positive prime trajectories $c_+,c_-$ with
\[
\eta(c_+)>0,
\qquad
\eta(c_-)<0.
\]

We now pass to an RFRS cover and find a second fibered class. There are two cases. If $\eta$ is
fibered, then set $p_1=\id_{M_R}$ and $M_1=M_R$. Otherwise, Agol's RFRS theorem provides a finite
cover
\[
p_1:M_1\longrightarrow M_R
\]
such that $p_1^*\eta$ lies in the closure of a fibered cone \cite[Theorem~5.1]{Agol2008}. In both
cases we set
\[
G_1=\pi_1(M_1),\qquad \eta_1=p_1^*\eta.
\]
In the first case the associated fibered cone is the one containing $\eta_1$. In either case, we
choose closed elevations $c'_+$ and $c'_-$ of positive powers of $c_+$ and $c_-$. Then
\[
\eta_1(c'_+)>0,
\qquad
\eta_1(c'_-)<0.
\]
Now choose a rational class in the interior of the indicated fibered cone sufficiently close to
$\eta_1$ that these two strict inequalities remain true. Take
\[
\chi_{1,1}\in H^1(M_1,\Z)
\]
to be the primitive integral class on its positive ray; it is fibered. If $d_{0,1}>0$ generates the
image of $p_1^*\chi_R$ on $G_1$, then define
\[
\chi_{0,1}=d_{0,1}^{-1}p_1^*\chi_R.
\]
Lifting the local product trivializations of the projection $\theta_R$ through $p_1$ shows that
$\theta_R\circ p_1:M_1\to\R/\Z$ is a surface bundle, possibly with disconnected fiber. Its induced
homomorphism has image $d_{0,1}\Z$, so by the covering space lifting criterion there is a mapping
\[
\theta_1:M_1\longrightarrow\R/d_{0,1}\Z
\]
whose composition with the natural covering $\R/d_{0,1}\Z\to\R/\Z$ is $\theta_R\circ p_1$. Bundle
trivializations lift through this circle covering, so $\theta_1$ is also a surface bundle.
Identifying $\pi_1(\R/d_{0,1}\Z)$ with $\Z$, the map $(\theta_1)_*=\chi_{0,1}$ is surjective. As
$M_1$ is connected, the homotopy exact sequence of the bundle implies that its fiber is connected.
So composing $\theta_1$ with
\[
\R/d_{0,1}\Z\longrightarrow\R/\Z,
\qquad [s]\longmapsto[s/d_{0,1}],
\]
gives a surface bundle representing the primitive class $\chi_{0,1}$.

Let $\widetilde\varphi_R^t$ be the lifted flow on $M_1$. It satisfies
\[
\theta_1(\widetilde\varphi_R^t(x))
=\theta_1(x)+t\pmod{d_{0,1}\Z},
\]
and its first positive return time to a fiber of $\theta_1$ is consequently $d_{0,1}$. Hence
$\varphi_1^t=\widetilde\varphi_R^{d_{0,1}t}$ is the unit-time suspension for $\chi_{0,1}$. Since
positive time rescaling preserves the oriented trajectories, including $c'_+$ and $c'_-$, the return
map on the connected fiber is a lift of $f^{d_{0,1}}$ and is hence pseudo-Anosov.

By \cref{lem:zero-word} for $\eta=\chi_{1,1}$, we obtain a positive prime orbit element $b\in G_1$
such that $\chi_{1,1}(b)=0$. By \cref{lem:orbit-roots}, $b$ is not a proper power in the mapping
torus group $G_1$. As $\chi_{0,1}(b)>0$, the two classes $\chi_{0,1}$ and $\chi_{1,1}$ are linearly
independent.

Now we pass to a standard characteristic cover. The subgroup $G_1$ has finite index in $G$, so we
can choose $m$ such that
\[
K=K_m(G)\leq\normalcore{G}{G_1}\leq G_1.
\]
Let $p_{K,1}:M_K\to M_1$ be the cover corresponding to $K\leq G_1$, and let $p_K:M_K\to M$ be its
composition with the previously chosen finite covers. Thus $p_K$ is the standard characteristic
cover corresponding to $K=K_m(G)$. For $j=0,1$, let $d_{j,K}>0$ be the positive generator of
$\chi_{j,1}(K)$ and define
\[
\chi_{j,K}=d_{j,K}^{-1}\chi_{j,1}|_K:K\twoheadrightarrow\Z.
\]
A finite cover of a surface bundle is a surface bundle after primitive normalization, so both
classes are fibered.

Let $k>0$ be least with $b^k\in K$ and set $b_K=b^k$. Let $c$ be the positive closed trajectory
represented by $b$. With compatible choices of basepoints and base paths, its lift closes after $r$
traversals when $b^r\in K$. By the minimality of $k$, the element $b_K=b^k$ therefore represents one
traversal of a connected component of $p_{K,1}^{-1}(c)$. This component covers $c$ with degree $k$
and inherits the positive flow orientation.

For $j=0,1$, by definition of $\chi_{j,K}$,
\[
\chi_{j,K}(b_K)
=\frac{1}{d_{j,K}}\chi_{j,1}(b^k)
=\frac{k}{d_{j,K}}\chi_{j,1}(b).
\]
Since $k,d_{0,K},d_{1,K}>0$, $\chi_{1,1}(b)=0$, and $\chi_{0,1}(b)>0$, we obtain
\[
\chi_{1,K}(b_K)=0,
\qquad
\chi_{0,K}(b_K)>0.
\]
So the two normalized classes remain linearly independent.

It remains to prove that $b_K$ is not a proper power in $K$. As $b$ is not a proper power in $G_1$,
the centralizer of $b^k$ in $G_1$ is $\langle b\rangle$. Also by the minimality of $k$ we have
\[
K\cap\langle b\rangle=\langle b^k\rangle.
\]
Now if $z^r=b_K$ with $z\in K$ and $r\geq2$, then $z$ centralizes $b^k$, so $z=b^s$ for some
$s\in\Z$. Since $z\in K$, the equality $K\cap\langle b\rangle=\langle b^k\rangle$ forces $s=kq$, and
then $b^{sr}=b^k$ forces $rq=1$, a contradiction. Thus $b_K$ is not a proper power in $K$. Finally,
suppose that the closed trajectory represented by $b_K$ were an $r$-fold traversal of a shorter
closed orbit, with $r\geq2$. Taking $z\in\pi_1(M_K)=K$ representing one traversal of that shorter
orbit, with the same base path, we would have $b_K=z^r$, contradicting the fact that $b_K$ is not a
proper power in $K$. Thus $b_K$ represents a positive prime closed trajectory. This completes the
proof.
\end{proof}

Let $N$ be a connected closed $3$-manifold and let $\chi\in H^1(N,\Z)$ be a primitive fibered class.
Choose a surface bundle map $p:N\to S^1$ representing $\chi$, with connected fiber $S$ and basepoint
$x\in S$. The \emph{$\chi$-fiber group} is
\[
\Pi_\chi
:=\ker\bigl(\chi:\pi_1(N,x)\to\Z\bigr).
\]
The inclusion $S\hookrightarrow N$ induces an injective homomorphism whose image is $\Pi_\chi$, so
we identify $\Pi_\chi$ with $\pi_1(S,x)$. For a primitive fibered class $\chi_{1,i}:K_i\to\Z$, we
write
\[
\Pi_i:=\ker\chi_{1,i}
\]
for the $\chi_{1,i}$-fiber group.

\begin{corollary}
\label{cor:configuration}
Let $M_1,M_2$ be closed orientable hyperbolic $3$-manifolds, set $G_i=\pi_1(M_i)$, and let
$\Phi:\wh G_1\xrightarrow{\cong}\wh G_2$ be a continuous isomorphism. Then there is an integer $m$
such that, for $K_i=K_m(G_i)$ we have $\Phi(\wh K_1)=\wh K_2$. Set $\Phi_K:=\Phi|_{\wh K_1}$. Then
the corresponding covers admit linearly independent primitive fibered classes
$\chi_{0,i},\chi_{1,i}\in H^1(K_i,\Z)$ and positive prime $\chi_{0,i}$-orbit elements $b_i\in K_i$
of the same least return period. These elements are not proper powers in $K_i$. There are
$\epsilon\in\{1,-1\}$ and $q\in\wh K_2$ such that
\[
\begin{gathered}
\wh\chi_{j,2}\circ\Phi_K
=\epsilon\wh\chi_{j,1}\qquad(j=0,1),\\
\Phi_K(b_1)=q b_2^\epsilon q^{-1},
\qquad b_i\in\Pi_i\qquad(i=1,2).
\end{gathered}
\]
In addition, $\Phi_K$ maps $\wh{\Pi_1}$ onto $\wh{\Pi_2}$ and restricts to an isomorphism between
these completed fiber groups. The induced map
\[
\wh K_1/\wh{\Pi_1}\longrightarrow\wh K_2/\wh{\Pi_2}
\]
is multiplication by $\epsilon$ under the identifications of both quotients with $\hZ$ given by
$\wh\chi_{1,i}$.
\end{corollary}
\begin{proof}
By \cref{thm:two-fibers}, there are an integer $m$, the subgroup $K_1=K_m(G_1)$, linearly
independent primitive fibered classes $\chi_{0,1},\chi_{1,1}\in H^1(K_1,\Z)$, and a positive prime
$\chi_{0,1}$-orbit element $b_1\in K_1$ with $\chi_{1,1}(b_1)=0$ which is not a proper power in
$K_1$. By \cref{lem:char-functor} we have
\[
\Phi(\wh K_1)=\wh K_2,
\qquad K_2=K_m(G_2).
\]
Thus the restriction $\Phi_K$ is an isomorphism between the completions of the corresponding
standard characteristic subgroups.

By \cref{thm:closed-integral}, there are $\epsilon\in\{1,-1\}$ and an integral isomorphism
\[
F:H^1(K_2,\Z)\xrightarrow{\cong}H^1(K_1,\Z)
\]
such that $\Phi_K^*=\epsilon(F\otimes\hZ)$. For $j=0,1$ define
\[
\chi_{j,2}=F^{-1}(\chi_{j,1}).
\]
Then these classes are primitive and linearly independent because $F$ is an integral isomorphism and
they are fibered by \cite[Theorem~1.3]{Liu2023}. The cohomology identity then reads
\[
\wh\chi_{j,2}\circ\Phi_K
=\epsilon\wh\chi_{j,1}\qquad(j=0,1).
\]
For $j=0,1$ and $i=1,2$, let $\Pi_{j,i}=\ker(\chi_{j,i}:K_i\to\Z)$ be the surface fiber groups, so
that $\Pi_{1,i}=\Pi_i$. By \cref{lem:extensions}(i) we have $\wh{\Pi_{j,i}}=\ker\wh\chi_{j,i}$.
Hence $\Phi_K$ restricts to an isomorphism
\[
\Phi_K|_{\wh{\Pi_{j,1}}}:
\wh{\Pi_{j,1}}\xrightarrow{\cong}\wh{\Pi_{j,2}}
\qquad(j=0,1).
\]
For $j=1$, the same cohomology identity shows that the induced map on the quotients, identified with
$\hZ$ by $\wh\chi_{1,i}$, is multiplication by $\epsilon$.

Finally, by \cref{thm:orbit-correspondence} and the suspensions defined by $\chi_{0,1}$ and
$\chi_{0,2}$, we have a positive prime $\chi_{0,2}$-orbit element $b_2\in K_2$ with the same least
return period as $b_1$, and an element $q\in\wh K_2$, such that $\Phi_K(b_1)=q b_2^\epsilon q^{-1}$.
By \cref{lem:orbit-roots}, $b_2$ is not a proper power in $K_2$. Therefore,
\[
\epsilon\chi_{1,2}(b_2)
=\wh\chi_{1,2}(\Phi_K(b_1))
=\epsilon\chi_{1,1}(b_1)=0,
\]
hence $b_2\in\ker\chi_{1,2}$, completing the proof.
\end{proof}


\section{Cellular realization}
\label{sec:realization}

We use a single marked element in a surface fiber to pass from the completed fiber isomorphism to
the full mapping torus isomorphism. The first lemma determines the completion of a semidirect
product and proves the finite extension statement used at the end of the argument.

\begin{lemma}
\label{lem:extensions}
\begin{enumerate}[label=\textup{(\roman*)}]
\item Let $\Pi$ be finitely generated and residually finite, let $a\in\Aut(\Pi)$, and set
$G=\Pi\rtimes_a\langle t\rangle$. Then there is a canonical isomorphism
\begin{equation}\label{eq:normal-form}
\wh G\cong\wh\Pi\rtimes_{\wh a}\hZ.
\end{equation}
Every $q\in\wh G$ can be written uniquely in the form
\[
q=c t^\lambda,
\qquad c\in\wh\Pi,
\quad \lambda\in\hZ.
\]
The continuous $\hZ$-action on $\wh\Pi$ is the inverse limit of the cyclic actions
on the finite $a$-invariant quotients of $\Pi$. We define $t^\lambda$ by the continuous
extension of $n\mapsto t^n$ from $\Z$ to $\hZ$.

\item Let $H_i\normalf G_i$ have finite index in finitely generated residually finite
groups. Let $\Phi:\wh G_1\to\wh G_2$ be a continuous isomorphism such that
$\Phi(\wh H_1)=\wh H_2$ and
\[
\Phi|_{\wh H_1}=\Inn(c)\circ\wh h
\]
for $c\in\wh H_2$ and an isomorphism $h:H_1\to H_2$. If $Z(\wh H_i)=1$ and
$\Out(H_i)\to\Out(\wh H_i)$ is injective, then
\[
\Phi=\Inn(u)\circ\wh f
\]
for some $u\in\wh G_2$ and an isomorphism $f:G_1\to G_2$.
\end{enumerate}
\end{lemma}
\begin{proof}
We prove \textup{(i)}. Let $N\normalf\Pi$. There are only finitely many subgroups of index
$[\Pi:N]$, because $\Pi$ is finitely generated, so the orbit of $N$ under the powers of $a$ is
finite, and
\[
K=\bigcap_{j\in\Z}a^j(N)
\]
is a finite index normal $a$-invariant subgroup with $K\leq N$. We denote the automorphism induced
by $a$ on the finite group $\Pi/K$ by $\bar a$, and its finite order by $r_K$. Therefore
\[
G\longrightarrow (\Pi/K)\rtimes_{\bar a}\Z/r_K\Z
\]
is a finite quotient whose restriction to $\Pi$ has kernel $K\leq N$. We give $G$ its profinite
topology, in which the finite index normal subgroups $V\normalf G$ form a basis of identity
neighborhoods, so the induced topology on $\Pi$ has the subgroups $\Pi\cap V$ as a basis of identity
neighborhoods, and each of these is a finite index normal subgroup of $\Pi$.

Conversely, for every finite index normal subgroup $N\normalf\Pi$, the finite quotient constructed
above has a kernel $V\normalf G$ satisfying
\[
\Pi\cap V=K\leq N.
\]
Hence the topology induced on $\Pi$ by the profinite topology of $G$ equals the full profinite
topology of $\Pi$. Moreover $\Pi$ is closed in $G$, because it is the kernel of the homomorphism
$G\to\Z$ and $\Z$ is residually finite. Completion therefore yields an exact sequence
\[
1\longrightarrow\wh\Pi\longrightarrow\wh G
\xrightarrow{\wh\chi}\hZ\longrightarrow1.
\]

The section $n\mapsto t^n$ of $\chi:G\to\Z$ is continuous for the profinite topologies and extends
to a continuous section $\hZ\to\wh G$. For each finite index normal $a$-invariant subgroup
$K\normalf\Pi$, with $\bar a$ and its order $r_K$ as above, $\bar a^n$ depends only on $n\bmod r_K$,
and hence the action of $\Z$ factors as
\[
\Z\longrightarrow\Z/r_K\Z
\longrightarrow\Aut(\Pi/K),
\qquad
n\longmapsto n\bmod r_K\longmapsto\bar a^n.
\]
Composing the reduction map $\hZ\to\Z/r_K\Z$ with this finite action, we get a continuous action of
$\hZ$ on $\Pi/K$, and in the inverse limit the continuous action $\lambda\mapsto\wh a^{\lambda}$ of
$\hZ$ on $\wh\Pi$. Multiplication now defines a continuous homomorphism
\[
\wh\Pi\rtimes_{\wh a}\hZ\longrightarrow\wh G.
\]
One sees that its image is compact, hence closed, and contains the dense subgroup generated by $\Pi$
and $t$. Thus it is surjective. Denote this multiplication homomorphism by $\Theta$, and suppose
that $(c,\lambda)\in\ker\Theta$, so that $ct^\lambda=1$. Since $\wh\chi$ vanishes on $\wh\Pi$ and
$\wh\chi(t^\lambda)=\lambda$, we obtain
\[
0=\wh\chi(ct^\lambda)=\lambda.
\]
It follows that $c=1$ in $\wh G$, hence in $\wh\Pi$ by the embedding $\wh\Pi\hookrightarrow\wh G$
established above. Thus $\Theta$ is injective, and being a continuous bijection between profinite
groups it is a topological isomorphism by compactness. This proves \eqref{eq:normal-form}, including
the uniqueness of the normal form.

For \textup{(ii)}, recall that a finite index subgroup of a residually finite group inherits its
full profinite topology. Also $Z(\wh H_i)=1$ implies $Z(H_i)=1$ because $H_i$ embeds in $\wh H_i$.
So all conditions of \cref{prop:centerless-extension} hold, and that proposition yields the required
$f$ and $u$.
\end{proof}

Recall that for a covering $p:S'\to S$ and a loop $\alpha:S^1\to S$, an \emph{elevation} of $\alpha$
is a loop $\alpha':S^1\to S'$ such that $p\circ\alpha'=\alpha\circ c_k$, where $c_k:S^1\to S^1$ is
the covering of degree $k>0$ and $k$ is the least positive degree for which the chosen lift closes,
see \cite{Hatcher2002}.

\begin{lemma}
\label{lem:elevations}
Let $\Pi=\pi_1(S)$ for a closed orientable surface $S$ of genus at least two, and let
$b\in\Pi\setminus\{1\}$ be an element that is not a proper power in $\Pi$. For every sufficiently
large $m$, if $k_m>0$ is least with $b^{k_m}\in K_m(\Pi)$, then $b^{k_m}$ is represented by a
nonseparating simple closed curve in the cover corresponding to $K_m(\Pi)$, and it is not a proper
power in $K_m(\Pi)$.
\end{lemma}
\begin{proof}
By Scott's theorem on subgroups of surface groups \cite{Scott1978}, with its correction, there is a
finite cover $S'\to S$ in which the conjugacy class of $b$ has a simple closed elevation $\alpha'$.
This elevation is essential, because $b$ is nontrivial and not a proper power in the surface group.

If $\alpha'$ is separating, write
\[
S'\setminus\alpha'=S'_+\sqcup S'_-.
\]
Each side has positive genus, because an essential separating curve in a closed surface cannot bound
a disk. Choose homomorphisms $\varphi_\pm:\pi_1(S'_\pm)\to\Z/2$ that are nonzero and kill the
boundary class. The two maps agree on the amalgamated boundary subgroup and define a surjection
\[
\varphi:\pi_1(S')\to\Z/2.
\]
In the corresponding connected double cover, the preimage of each side is connected and $\alpha'$
has two lifts, so deleting either one of these lifts leaves the two lifted sides joined along the
other. Thus both lifts are nonseparating. Replacing $S'$ by this double cover if necessary, we
obtain a finite cover containing a nonseparating simple elevation $\alpha$ of $b$.

Let $J\leq\Pi$ be the subgroup defining this cover, and let $L=\bigcap_{g\in\Pi}gJg^{-1}$ be its
normal core. Choose $m_0$ with $K_{m_0}(\Pi)\leq L$. Then for every $m\geq m_0$, each connected
elevation $\widetilde\alpha$ to the cover corresponding to $K_m(\Pi)$ is simple. Denote the covering
from this surface to $S'$ by $p$. If the degree of $\widetilde\alpha$ over $\alpha$ is $r>0$, then
\[
p_*[\widetilde\alpha]=r[\alpha]\neq0
\qquad\text{in }H_1(S',\Z).
\]
Hence $[\widetilde\alpha]\neq0$. Since on an orientable closed surface a simple closed curve is
nonseparating when its homology class is nonzero, every such $\widetilde\alpha$ is nonseparating.

By construction, the cover corresponding to $K_m(\Pi)$ is normal over $S$, and its deck group is
transitive on the fiber over a basepoint of the loop representing $b$ and therefore on the connected
elevations of that loop. A lift of $\alpha$ is one such elevation, and every other elevation is a
deck translate of it, so it is also simple and nonseparating. In particular, the elevation beginning
at the chosen sheet has first-return degree $k_m$ and represents the free homotopy class of
$b^{k_m}$. This proves our first claim.

Now let $K=K_m(\Pi)$ and take $k=k_m$ to be least with $b^k\in K$. That $b^k$ is not a proper power
in $K$ follows exactly as at the end of the proof of \cref{thm:two-fibers}, using that the
centralizer of $b^k$ in the surface group is $\langle b\rangle$.
\end{proof}

\begin{theorem}
\label{thm:one-orbit}
Let $\Pi_i=\pi_1(S_i)$ for closed orientable surfaces of genus at least two, and let $a_i$ represent
orientation preserving pseudo-Anosov homeomorphisms. We set
$G_i=\Pi_i\rtimes_{a_i}\langle t_i\rangle$. Suppose $\Phi:\wh G_1\to\wh G_2$ is a continuous
isomorphism such that $\Phi(\wh\Pi_1)=\wh\Pi_2$, induces multiplication by $\epsilon\in\{\pm1\}$ on
the quotients $\hZ$, and satisfies
\[
\Phi(b_1)=q b_2 q^{-1}
\]
for some $q\in\wh G_2$ and nontrivial elements $b_i\in\Pi_i$ which are not proper powers in $\Pi_i$.
Then $\Phi=\Inn(w)\circ\wh f$ for a discrete isomorphism $f:G_1\to G_2$.
\end{theorem}
\begin{proof}
Let $\chi_i:G_i\to\Z$ be the quotient maps with $\chi_i(t_i)=1$. By \cref{lem:extensions}(i), every
element of $\wh G_2$ can be written uniquely as
\[
q=c t_2^\lambda,
\qquad c\in\wh\Pi_2,
\quad\lambda\in\hZ.
\]
We replace $\Phi$ by $\Inn(q^{-1})\Phi$. This does not change the sign on the quotient, and the
restriction $\psi=\Phi|_{\wh\Pi_1}:\wh\Pi_1\to\wh\Pi_2$ now satisfies $\psi(b_1)=b_2$. Since
$\wh\chi_2\Phi=\epsilon\,\wh\chi_1$ and $\chi_1(t_1)=1$, we have $\wh\chi_2(\Phi(t_1))=\epsilon$, so
by the normal form of \cref{lem:extensions}(i), $\Phi(t_1)=z t_2^\epsilon$ for some $z\in\wh\Pi_2$.
Hence for every $x\in\wh\Pi_1$, we obtain
\[
\begin{aligned}
\psi(\wh a_1(x))
&=\Phi(t_1xt_1^{-1})\\
&=\Phi(t_1)\psi(x)\Phi(t_1)^{-1}\\
&=z t_2^\epsilon\psi(x)t_2^{-\epsilon}z^{-1}\\
&=z\,\wh a_2^\epsilon(\psi(x))\,z^{-1}.
\end{aligned}
\]
This implies
\begin{equation}\label{eq:aligned}
\psi\wh a_1\psi^{-1}
=\Inn(z)\wh a_2^\epsilon.
\end{equation}
Now choose $m$ sufficiently large for \cref{lem:elevations} on both sides and set $K_i=K_m(\Pi_i)$.
The subgroups $K_i$ are characteristic, so they are invariant under $a_i$, and
$\psi(\wh K_1)=\wh K_2$ by \cref{lem:char-functor}. The least positive exponent $k$ with
$b_i^k\in K_i$ is the order of the image of $b_i$ in $\Pi_i/K_i$. It is the same for $i=1,2$ because
$\psi(b_1)=b_2$ and $\psi$ identifies the two finite quotients. Set $x_i=b_i^k$; then $x_i$
represents a nonseparating simple closed curve in the surface with fundamental group $K_i$.

Choose $h_0\in\Pi_2$ with $z\wh K_2=h_0\wh K_2$, and write $z=z'h_0$ with $z'\in\wh K_2$. We define
\[
f_1=a_1|_{K_1},
\qquad
g_2=(\Inn(h_0)\circ a_2^\epsilon)|_{K_2}.
\]
Both automorphisms are induced by pseudo-Anosov homeomorphisms of the corresponding finite covers:
$a_1$ and $a_2^\epsilon$ lift because the $K_i$ are characteristic, and composition with the deck
transformation represented by $h_0$ preserves the lifted stable and unstable measured foliations and
their scaling factors. We set $\psi_K=\psi|_{\wh K_1}$. Restricting \eqref{eq:aligned} we have
$\psi_K\wh f_1\psi_K^{-1}=\Inn(z')\wh g_2$, and by induction
\begin{equation}\label{eq:iterated-alignment}
\psi_K(f_1^n(x_1))=z_n g_2^n(x_2)z_n^{-1},
\end{equation}
where
\[
z_n=z'\,\wh g_2(z')\cdots \wh g_2^{n-1}(z')\in\wh K_2
\qquad(n\geq1).
\]

For a closed surface $R$ of genus at least two, the curve graph $\mathcal C(R)$
\cite{MasurMinsky1999} has isotopy classes of essential simple closed curves as vertices. Two
distinct vertices are joined by an edge when they have disjoint representatives. We write
$d_{\mathcal C}$ for its path metric with every edge of length one, and use $x_i$ also for the curve
class represented by $x_i$ in the surface with fundamental group $K_i$. For pseudo-Anosov mapping
classes, these distances have a positive linear lower bound in the number of iterates
\cite{MasurMinsky1999}. Therefore
\[
d_{\mathcal C}(x_1,f_1^n(x_1))\to\infty,
\qquad
d_{\mathcal C}(x_2,g_2^n(x_2))\to\infty.
\]
Choose $n$ so large that both distances are at least three. Then each pair fills its surface, since
a curve disjoint from both would give a path of length at most two in the curve graph, and all four
curves are nonseparating. The two chosen curves determine the closed cyclic subgroups
\[
\cl{\langle x_1\rangle},
\qquad
\cl{\langle f_1^n(x_1)\rangle}
\quad\text{in }\wh K_1,
\]
and
\[
\cl{\langle x_2\rangle},
\qquad
\cl{\langle g_2^n(x_2)\rangle}
\quad\text{in }\wh K_2.
\]
Since $\psi_K(x_1)=x_2$, by continuity
$\psi_K\bigl(\cl{\langle x_1\rangle}\bigr)=\cl{\langle x_2\rangle}$, and
\eqref{eq:iterated-alignment} yields
\[
\psi_K\bigl(\cl{\langle f_1^n(x_1)\rangle}\bigr)
=z_n\,\cl{\langle g_2^n(x_2)\rangle}\,z_n^{-1}.
\]
Hence the conditions of \cref{cor:two-curve} are satisfied for the restricted map
$\psi_K:\wh K_1\to\wh K_2$, so we have
\[
\psi_K=\Inn(v_K)\circ\wh h_K
\]
for $v_K\in\wh K_2$ and an isomorphism $h_K:K_1\to K_2$.

Applying \cref{lem:extensions}(ii) to the finite index normal subgroups $K_i\normalf\Pi_i$, using
\cref{lem:normalizer-rigidity} for the triviality of $Z(\wh K_i)$ and the injectivity of
$\Out(K_i)\to\Out(\wh K_i)$, we obtain
\[
\psi=\Inn(s)\circ\wh h,
\qquad s\in\wh\Pi_2,
\]
for an isomorphism $h:\Pi_1\to\Pi_2$. We now replace $\Phi$ by $\Inn(s^{-1})\circ\Phi$, replace
$\psi$ by $\Inn(s^{-1})\circ\psi$, and then replace $z$ by
\[
s^{-1}z\wh a_2^\epsilon(s).
\]
Indeed, conjugating \eqref{eq:aligned} on the left by $\Inn(s^{-1})$ and on the right by $\Inn(s)$
changes the conjugator in this way. With these choices we have $\psi=\wh h$, and \eqref{eq:aligned}
keeps the same form. It now says that the completions of $ha_1h^{-1}$ and $a_2^\epsilon$ differ by
an inner automorphism. Then by injectivity of $\Out(\Pi_2)\to\Out(\wh\Pi_2)$ we have $d\in\Pi_2$
such that
\[
ha_1h^{-1}=\Inn(d)\circ a_2^\epsilon.
\]
Hence
\[
f(t_1)=dt_2^\epsilon, \qquad f|_{\Pi_1}=h
\]
defines an isomorphism $f:G_1\to G_2$.

Finally, for the same choices, write $\Phi(t_1)=vt_2^\epsilon$ with $v\in\wh\Pi_2$. Comparing
conjugation by this element with the conjugation given by $f(t_1)=dt_2^\epsilon$ shows that
$d^{-1}v$ centralizes $\wh\Pi_2$. By \cref{lem:normalizer-rigidity}, $Z(\wh\Pi_2)=1$, so $v=d$. Let
$\Phi^{(0)}$ denote the original isomorphism, and let
\[
\Phi^{(1)}
=\Inn(q^{-1})\circ\Phi^{(0)},
\qquad
\Phi^{(2)}
=\Inn(s^{-1})\circ\Phi^{(1)}
\]
be its two successive conjugates. The argument above shows that $\Phi^{(2)}$ and $\wh f$ agree on
$\wh\Pi_1$ and on $t_1$. These elements topologically generate $\wh G_1$, so continuity forces
$\Phi^{(2)}=\wh f$. Since $\Phi^{(2)}=\Inn((qs)^{-1})\circ\Phi^{(0)}$, the original isomorphism has
the required form $\Phi^{(0)}=\Inn(qs)\circ\wh f$ with $w=qs\in\wh G_2$.
\end{proof}

\begin{corollary}
\label{cor:closed-rigidity}
Every isomorphism of profinite completions of closed orientable hyperbolic $3$-manifold groups, and
every such isomorphism between cocompact lattices in $\PSL_2(\C)$, is discretely induced up to
profinite inner automorphism.
\end{corollary}
\begin{proof}
For manifolds, \cref{cor:configuration} provides corresponding characteristic covers with positive
prime suspension orbit elements $b_i\in\ker\chi_{1,i}$. By \cref{lem:orbit-roots}, these elements
are not proper powers in the covering groups and hence are not proper powers in their fiber
subgroups $\ker\chi_{1,i}$. When applying \cref{thm:one-orbit} we replace $b_2$ by $b_2^\epsilon$;
inversion preserves this root property. Then the theorem realizes the isomorphism between the
completions of the covering groups. These characteristic covering groups are normal of finite index,
and the conditions of \cref{lem:extensions}(ii) hold by \cref{lem:normalizer-rigidity}, so that
lemma yields the realization on the original groups.

For cocompact lattices, choose a torsion-free finite index normal subgroup in each lattice by
Selberg's lemma. A common sufficiently deep standard characteristic level is contained in the chosen
subgroups and is matched by the given profinite isomorphism, and the corresponding quotients of
$\HH^3$ are closed manifolds. We apply the manifold case to these matched characteristic subgroups,
which are normal of finite index, and conclude with \cref{lem:extensions}(ii), whose conditions hold
by \cref{lem:normalizer-rigidity}.
\end{proof}


\section{Global lattice}
\label{sec:cusps}
This section uses a standard cusp filling argument. Recall that for a cusp torus $T$, the peripheral
subgroup is the image of $\pi_1(T)$ in the manifold group, defined using a chosen base path, and
that a slope is an unoriented isotopy class of essential simple closed curves on $T$. After choosing
an orientation, we represent a slope by an element $d\in\pi_1(T)\cong\Z^2$ belonging to an integral
basis; changing the orientation replaces $d$ by $d^{-1}$.

For the filling argument we only need each peripheral map to be a profinite unit times an integral
map: after passing to the closed normal subgroup generated by the filling relations, the unit
disappears.

\begin{lemma}
\label{lem:peripheral-scalar}
Let $\Phi:\wh\Gamma\to\wh\Delta$ be an isomorphism between completions of orientable cusped finite
volume hyperbolic $3$-manifold groups. Then $\Phi$ induces a bijection between the conjugacy classes
of cusp closures. Let $P,Q\cong\Z^2$ be matched cusp groups and let $u\in\wh\Delta$ satisfy
$\Phi(\wh P)=u\wh Q u^{-1}$. Then there exists an integral isomorphism $A:P\to Q$ and a unit
$\mu\in\hZ^\times$ such that
\begin{equation}\label{eq:peripheral-scalar}
\Inn(u^{-1})\Phi|_{\wh P}=\mu\wh A.
\end{equation}
The unit and the normalizing conjugator may depend on the cusp.
\end{lemma}
\begin{proof}
By \cite{AschenbrennerFriedlWilton2015}, a cusp group is geometrically finite and is a virtual
retract. So \cref{lem:virtual-retract-topology} identifies its closure with its full completion
$\hZ^2$. By \cite[Lemma~4.5]{WiltonZalesskii2017} peripheral closures form a malnormal family, so
each cusp closure equals its normalizer and is a maximal abelian subgroup.

The subgroup $\Phi(\wh P)$ is a closed copy of $\hZ^2$. By \cite[Theorem~9.3]{WiltonZalesskii2017},
every such subgroup is contained in a conjugate of a cusp closure: it has no nonabelian free pro-$p$
subgroup and is not projective, since its $p$-cohomological dimension is two. Maximality of the
abelian subgroup makes the containment an equality. Repeating the argument for $\Phi^{-1}$, we get a
bijection of peripheral conjugacy classes. Conjugating $\Phi$ by $u^{-1}$, we may assume for the
remainder of the proof that $\Phi(\wh P)=\wh Q$.

Now choose finite index subgroups $\Gamma^P\leq\Gamma$ and $\Delta^Q\leq\Delta$ that contain $P$ and
$Q$ and admit retractions
\[
r_P:\Gamma^P\to P,
\qquad
r_Q:\Delta^Q\to Q.
\]
The open subgroup
\[
U=\Phi(\cl{\Gamma^P})\cap\cl{\Delta^Q}\leq\wh\Delta
\]
contains $\wh Q$. Put
\[
\Delta_0=U\cap\Delta,
\qquad
\Gamma_0=\Phi^{-1}(U)\cap\Gamma.
\]
Then $\cl{\Delta_0}=U$, $\cl{\Gamma_0}=\Phi^{-1}(U)$, and $\Phi$ restricts to
$\wh{\Gamma_0}\cong\wh{\Delta_0}$. Moreover we have
\[
P\leq\Gamma_0\leq\Gamma^P,
\qquad
Q\leq\Delta_0\leq\Delta^Q.
\]
The smaller subgroups still contain the images of the retractions, so $r_P$ and $r_Q$ restrict to
retractions of $\Gamma_0$ and $\Delta_0$.

Put
\[
L_\Gamma=H_1(\Gamma_0,\Z)/\tors,
\qquad
L_\Delta=H_1(\Delta_0,\Z)/\tors,
\]
where $\tors$ denotes the torsion subgroup. The two retractions show that the homomorphisms induced
by inclusion,
\[
i_P:P\hookrightarrow L_\Gamma,
\qquad
i_Q:Q\hookrightarrow L_\Delta,
\]
are split injections. In particular, both first Betti numbers are at least two. The finite covers
corresponding to $\Gamma_0$ and $\Delta_0$ have orientable finite volume hyperbolic interiors. Hence
the homological form of Liu's regularity theorem applies to the restricted profinite isomorphism
\cite[Theorem~6.1 and Proposition~3.2(2)]{Liu2023}. It provides an integral isomorphism
$F:L_\Gamma\to L_\Delta$ and $\mu\in\hZ^\times$ such that the induced map on completed free
abelianizations is $\mu\wh F$, and then by naturality we get
\begin{equation}\label{eq:peripheral-abelianization-square}
\mu\wh F\,\wh i_P
=\wh i_Q\,\Phi|_{\wh P}.
\end{equation}
Since the right side has image $\wh i_Q(\wh Q)$, and multiplication by the unit $\mu$ is an
automorphism preserving $\wh i_Q(\wh Q)$, we get
\[
\wh F\bigl(\wh i_P(\wh P)\bigr)=\wh i_Q(\wh Q).
\]

Now for every $p\in P$, this equality of completed subgroups shows that
\[
F(i_P(p))
\in L_\Delta\cap\wh i_Q(\wh Q)
=i_Q(Q).
\]
The same argument applied to $F^{-1}$ shows
\[
F(i_P(P))\subseteq i_Q(Q),
\qquad
F^{-1}(i_Q(Q))\subseteq i_P(P),
\]
hence $F(i_P(P))=i_Q(Q)$. Therefore we obtain an integral isomorphism
\[
A=i_Q^{-1}\circ F\circ i_P:P\xrightarrow{\cong}Q,
\]
where $i_Q^{-1}$ is defined on $i_Q(Q)$. Completing the equality $Fi_P=i_QA$, substituting into
\eqref{eq:peripheral-abelianization-square}, and using the $\hZ$-linearity of $\wh i_Q$, we get
$\wh i_Q\circ(\mu\wh A)=\wh i_Q\circ\Phi|_{\wh P}$, and since $\wh i_Q$ is injective,
$\Phi|_{\wh P}=\mu\wh A$. The map used in this calculation is $\Inn(u^{-1})\circ\Phi$, so for the
original isomorphism this is \eqref{eq:peripheral-scalar}.
\end{proof}

\begin{lemma}
\label{lem:unit-fillings}
Let $P_1,\ldots,P_s$ and $Q_1,\ldots,Q_s$ represent the corresponding cusp conjugacy classes for the
groups $\Gamma$ and $\Delta$ of \cref{lem:peripheral-scalar}, and let $A_i:P_i\to Q_i$ be the
integral isomorphisms given there. We choose primitive slopes $d_i\in P_i$, set
$e_i=A_i(d_i)\in Q_i$, and choose pairwise distinct positive integers $r_i$. We set
$\ell_n=\lcm(1,\ldots,n)$ and
\begin{equation}
\begin{split}
k_i(n)   & =r_i\ell_n,                                            \\
\Gamma_n & =\Gamma/\Ncl{d_1^{k_1(n)},\ldots,d_s^{k_s(n)}}{\Gamma}, \\
\Delta_n & =\Delta/\Ncl{e_1^{k_1(n)},\ldots,e_s^{k_s(n)}}{\Delta}.
\end{split}
\label{eq:fillings}
\end{equation}
Then $\Phi$ induces $\Phi_n:\wh\Gamma_n\xrightarrow{\cong}\wh\Delta_n$, and every fixed finite
quotient of $\Gamma$ or $\Delta$ factors through the corresponding fillings for all sufficiently
large $n$.
\end{lemma}
\begin{proof}
For a topological generator $e$ of a procyclic group, a positive integer $k$, and a unit
$\mu\in\hZ^\times$ we have
\[
\cl{\langle e^{k\mu}\rangle}
=e^{k\mu\hZ}
=e^{k\hZ}
=\cl{\langle e^k\rangle}.
\]
By \cref{lem:peripheral-scalar}, $\Phi$ therefore maps the closed subgroup generated by
$d_i^{k_i(n)}$ onto a conjugate of the closed subgroup generated by $e_i^{k_i(n)}$, for each cusp.
The individual cusp conjugators have no effect once we pass to the closed normal subgroup generated
by all the relations. Hence $\Phi$ identifies the two closed normal filling subgroups. By
\cref{lem:completion-quotient}, quotienting yields
$\Phi_n:\wh\Gamma_n\xrightarrow{\cong}\wh\Delta_n$. Let $\rho:\Gamma\to F$ be a fixed finite
quotient. If $o_i$ is the order of $\rho(d_i)$, then $o_i\mid\ell_n$ for every sufficiently large
$n$. Hence $\rho(d_i^{r_i\ell_n})=1$ for every cusp, so $\rho$ kills the defining normal subgroup of
$\Gamma_n$ and factors through $\Gamma_n$. The same argument applies to $\Delta$, and in particular
to the characteristic quotients used below.
\end{proof}

\begin{proposition}
\label{prop:geometric-fillings}
Let $\cl M_\Gamma$ and $\cl M_\Delta$ be compact horospherical truncations of the cusped hyperbolic
manifolds with groups $\Gamma$ and $\Delta$. Use the primitive slopes $d_i,e_i$ and the pairwise
distinct orders $k_i(n)=r_i\ell_n$ from \eqref{eq:fillings}. For every sufficiently large $n$, the
following statements hold.
\begin{enumerate}[label=\textup{(\roman*)}]
\item Attaching an orbifold solid torus of generalized coefficient $(k_i(n),0)$ to the $i$th
cusp of $\cl M_\Gamma$ produces a compact orientable hyperbolic orbifold
$\mathcal O_{\Gamma,n}$ with
\[
\pi_1^{\mathrm{orb}}(\mathcal O_{\Gamma,n})\cong\Gamma_n.
\]
The corresponding filling of $\cl M_\Delta$ produces $\mathcal O_{\Delta,n}$ with
orbifold group $\Delta_n$.

\item The singular locus consists exactly of the filling singular closed geodesic $\kappa_{i,n}$. The local
group along $\kappa_{i,n}$ is cyclic of order $k_i(n)$, and the image of $d_i$,
respectively $e_i$, has exact order $k_i(n)$ in the orbifold group.

\item Every nontrivial finite subgroup of $\Gamma_n$, respectively $\Delta_n$, is conjugate
into the local cyclic group of some filling singular closed geodesic.

\item Every orbifold isometry
\[
F_n:\mathcal O_{\Gamma,n}\longrightarrow\mathcal O_{\Delta,n}
\]
maps the singular geodesic of order $k_i(n)$ to the singular geodesic of the same order. After deleting
invariant open tubes around all singular closed geodesic and adjoining product collars, its restriction
is a homeomorphism
\[
\cl F_n:\cl M_\Gamma\longrightarrow\cl M_\Delta
\]
that takes the slope $d_i$ to $e_i$ up to sign.

\item Let
\[
q_n:\Gamma\twoheadrightarrow\Gamma_n,
\qquad
q'_n:\Delta\twoheadrightarrow\Delta_n
\]
be the filling maps. Choose basepoints outside the singular closed geodesic tubes and compatible base
paths. If $f_n=(F_n)_*$ and $h_n=(\cl F_n)_*$ then, after replacing $f_n$ by a
discrete inner conjugate,
\begin{equation}\label{eq:based-drilling-square}
f_nq_n=q'_nh_n.
\end{equation}
\end{enumerate}
\end{proposition}
\begin{proof}
In the basis consisting of the primitive slope and a longitude on the $i$th cusp, the generalized
coefficient $(k_i(n),0)$ makes $d_i$ the orbifold meridian with cone angle $2\pi/k_i(n)$. Because
all $k_i(n)$ tend to infinity, the simultaneous orbifold Dehn filling theorem applies for all
sufficiently large $n$ \cite{BoileauPorti2001} and produces the hyperbolic orbifolds in
part~\textup{(i)}. Since the local filling model is
\[
(D^2\times S^1)/C_{k_i(n)},
\]
where the cyclic group $C_{k_i(n)}$ of order $k_i(n)$ acts by rotations on the disk and trivially on
the circle, the singular locus in the attached solid torus is exactly its singular closed geodesic and has effective
local order $k_i(n)$. The complement of the attached solid tori is the original manifold truncation
and has no singular points. This proves the statement about the singular locus in
part~\textup{(ii)}. Since developability makes each local group inject into the orbifold fundamental
group, the orbifold van Kampen theorem yields
\[
\pi_1^{\mathrm{orb}}(\mathcal O_{\Gamma,n})
=\Gamma/\Ncl{d_1^{k_1(n)},\ldots,d_s^{k_s(n)}}{\Gamma}
=\Gamma_n,
\]
and similarly for $\Delta_n$, which also shows that the slope image has exact order $k_i(n)$.

Let $H$ be a nontrivial finite subgroup of either orbifold group. By the Cartan fixed point theorem,
its action on $\HH^3$ has a common fixed point, and the point projects to the singular locus, since
regular points have trivial stabilizer. By part~\textup{(ii)}, its local stabilizer is the cyclic
group of a singular geodesic, and $H$ is conjugate into that group. This proves part~\textup{(iii)}.

Since the numbers $k_i(n)=r_i\ell_n$ are pairwise distinct, $F_n$ maps each singular geodesic to the
singular geodesic with the same label. Choose pairwise disjoint sufficiently small closed tubes
around the singular geodesics in the orbifold, and use their images under $F_n$ as the
corresponding tubes in the second orbifold. Denote the underlying solid tori by $V_i$ and $V'_i$,
respectively. Choose the tubes small enough to lie inside the corresponding attached solid tori on
both sides. The regions between their boundaries and the original filling boundaries are product
collars. Thus the restriction of $F_n$ to the complements of the tube interiors, followed by the
collar identifications, is a homeomorphism
\[
\cl F_n:\cl M_\Gamma\longrightarrow\cl M_\Delta.
\]
The restriction to the tubes yields a homeomorphism of the underlying solid tori $V_i\to V'_i$,
whose boundary restriction maps
\[
\ker\bigl(\pi_1(\partial V_i)\to\pi_1(V_i)\bigr)
\quad\text{onto}\quad
\ker\bigl(\pi_1(\partial V'_i)\to\pi_1(V'_i)\bigr).
\]
These infinite cyclic kernels are generated by the primitive disk-bounding meridians. Under the
collar identifications their generators represent $d_i$ and $e_i$, respectively, so $\cl F_n$ takes
$d_i$ to $e_i$ up to sign. This proves part~\textup{(iv)}.

We obtain the filling of $\cl M_\Gamma$ by attaching the orbifold solid tori along their boundary
tori, and we use the same construction on the $\Delta$ side. Choose the collar identifications so
that, after inclusion into the filled orbifolds, they are homotopic to the original filling
inclusions. Under these identifications, $\cl F_n$ is induced by the restriction of $F_n$ to the
tube complements. Hence the two composites from $\cl M_\Gamma$ to $\mathcal O_{\Delta,n}$ are
homotopic. Their induced homomorphisms satisfy $f_nq_n=\Inn(b_n)q'_nh_n$ for some $b_n\in\Delta_n$.
Replacing $f_n$ by $\Inn(b_n^{-1})f_n$ we obtain \eqref{eq:based-drilling-square}.
\end{proof}

\begin{corollary}
\label{cor:cusped-rigidity}
Every isomorphism $\Phi:\wh\Gamma\to\wh\Delta$ between completions of orientable cusped finite
volume hyperbolic $3$-manifold groups has the form
\begin{equation}
\Phi=\Inn(u)\circ\wh f,
\qquad u\in\wh\Delta,\quad f:\Gamma\xrightarrow{\cong}\Delta.
\label{eq:cusped-rigidity}
\end{equation}
\end{corollary}
\begin{proof}
We form the fillings \eqref{eq:fillings} using \cref{lem:peripheral-scalar,lem:unit-fillings}. By
\cref{cor:closed-rigidity}, for all sufficiently large $n$ the filled map is
\[
\Phi_n=\Inn(v_n)\wh f_n,\qquad f_n:\Gamma_n\xrightarrow{\cong}\Delta_n.
\]
Mostow--Prasad rigidity realizes $f_n$, up to a discrete inner automorphism, by an orbifold isometry
\cite{Mostow1968,Prasad1973}. We absorb that inner automorphism in $v_n$. Applying
\cref{prop:geometric-fillings} to the orbifold isometry inducing $f_n$, we get a discrete
isomorphism $h_n:\Gamma\to\Delta$. If obtaining the commuting identity $f_nq_n=q'_nh_n$ requires
replacing $f_n$ by $\Inn(c_n)\circ f_n$ with $c_n\in\Delta_n$, we replace $v_n$ by $v_nc_n^{-1}$ at
the same time, which leaves $\Phi_n=\Inn(v_n)\wh f_n$ unchanged. With these choices,
\begin{equation}
f_nq_n=q'_nh_n,
\qquad
\wh q'_n\Phi=\Phi_n\wh q_n.
\label{eq:drilling-square}
\end{equation}
The first equality is \eqref{eq:based-drilling-square}, and the second is the quotient identity from
\cref{lem:unit-fillings}.

Only finitely many isomorphisms $\Gamma\to\Delta$ exist up to conjugation in $\Delta$: after fixing
one such isomorphism, composition is a bijection from $\Out(\Delta)$ to these classes, and
$\Out(\Delta)$ is the isometry group of the finite volume hyperbolic manifold
\cite{Mostow1968,Prasad1973}, which is finite. So the isomorphisms $h_n$ contain an infinite
subsequence belonging to one conjugacy class. Choose a representative $f:\Gamma\to\Delta$ of this
class and write the indices of the subsequence as $n_j$. Then $n_j\to\infty$, and for each $j$ there
exists $a_{n_j}\in\Delta$ such that $h_{n_j}=\Inn(a_{n_j})\circ f$, with $f$ independent of $j$. In
what follows, we restrict to this unbounded subsequence and keep writing $n$ for its filling
indices. Thus
\[
h_n=\Inn(a_n)\circ f
\]
for every index under consideration. For the characteristic quotient
$p_m:\wh\Delta\to\Delta/K_m(\Delta)$, choose $n$ in this subsequence large enough that $p_m$ factors
through the filling, say $p_m=\bar p_{m,n}\wh q'_n$. Using the two commutation identities in
\eqref{eq:drilling-square}, together with $\Phi_n=\Inn(v_n)\circ\wh f_n$ and $h_n=\Inn(a_n)\circ f$,
we obtain the following equality of homomorphisms from $\wh\Gamma$ to $\Delta/K_m(\Delta)$:
\begin{equation}
p_m\Phi=
\Inn\bigl(\bar p_{m,n}(v_n)\,p_m(a_n)\bigr)\,p_m\wh f.
\label{eq:whole-quotient}
\end{equation}
By \eqref{eq:whole-quotient}, $\Phi$ and $\wh f$ induce maps on every standard characteristic
quotient of $\Delta$ which differ by an inner automorphism, and \cref{lem:compact-conjugators}
gives $u\in\wh\Delta$ with $\Phi=\Inn(u)\circ\wh f$, which is \eqref{eq:cusped-rigidity}.
\end{proof}

\begin{corollary}
\label{cor:all-integrality}
Let $\Gamma$ and $\Delta$ be fundamental groups of orientable finite volume hyperbolic
$3$-manifolds, and let $\Phi:\wh\Gamma\xrightarrow{\cong}\wh\Delta$. Under the canonical
identifications $H^1_{\mathrm{cts}}(\wh G,\hZ) \cong H^1(G,\Z)\otimes_{\Z}\hZ$, we have
\[
\Phi^*=B\otimes_{\Z}\id_{\hZ}
\qquad\text{for some}\qquad
B:H^1(\Delta,\Z)\xrightarrow{\cong}H^1(\Gamma,\Z).
\]

In the cusped case, let $P\leq\Gamma$ and $Q\leq\Delta$ be cusp subgroups. Then for every
$u\in\wh\Delta$,
\[
\Phi(\wh P)=u\wh Q u^{-1}
\quad\Longrightarrow\quad
\Inn(u^{-1})\circ\Phi|_{\wh P}=\wh A
\quad\text{for some }A:P\xrightarrow{\cong}Q.
\]
\end{corollary}

\begin{proof}
First, the two manifolds are either both closed or both cusped. To see this, note that goodness
identifies their group cohomology with the continuous cohomology of their profinite completions
\cite{AschenbrennerFriedlWilton2015}, so $\Phi$ induces an isomorphism
\[
H^3(\Gamma;\F_p)\cong H^3(\Delta;\F_p)
\]
for every prime $p$. For a closed orientable hyperbolic $3$-manifold this group is $\F_p$, by
asphericity and Poincar\'e duality, while for a cusped manifold a compact horospherical truncation
is aspherical with nonempty boundary, so the group vanishes by Poincar\'e--Lefschetz duality. Thus
the two cases cannot be mixed. The closed case now follows from \cref{thm:closed-integral}. In the
cusped case, \eqref{eq:cusped-rigidity} implies
\[
\Phi=\Inn(v)\circ\wh f,
\qquad
v\in\wh\Delta,
\qquad
f:\Gamma\xrightarrow{\cong}\Delta.
\]
Hence we can take $B=f^*:H^1(\Delta,\Z)\xrightarrow{\cong}H^1(\Gamma,\Z)$. For the peripheral
conclusion, fix cusp subgroups $P\leq\Gamma$ and $Q\leq\Delta$, and let $u\in\wh\Delta$ satisfy
$\Phi(\wh P)=u\wh Q u^{-1}$. By Mostow--Prasad rigidity \cite{Mostow1968,Prasad1973}, there exists
an isometry $F:\HH^3/\Gamma\longrightarrow\HH^3/\Delta$ whose induced homomorphism on fundamental
groups is $f$, after choosing a suitable path between basepoints. Since $F$ sends cusp ends to cusp
ends, there exist a cusp subgroup $Q'\leq\Delta$ and $a'\in\Delta$ such that $f(P)=a'Q'(a')^{-1}$.
Using $\Phi=\Inn(v)\circ\wh f$ we obtain
\[
(va')\wh{Q'}(va')^{-1}=u\wh Q u^{-1}.
\]
By \cite[Lemma~4.5]{WiltonZalesskii2017}, as used in the proof of \cref{lem:peripheral-scalar},
closures of representatives of distinct cusp conjugacy classes form a malnormal family. Thus, if
$Q'$ and $Q$ were not conjugate in $\Delta$, we would have
\[
\wh{Q'}\cap g\wh Qg^{-1}=\{1\}
\qquad\text{for every }g\in\wh\Delta.
\]
But we have just seen that $\wh{Q'}=g\wh Qg^{-1}$ for $g=(va')^{-1}u$, a contradiction. So
$Q'=bQb^{-1}$ for some $b\in\Delta$. Taking $a=a'b$, we obtain
\[
f(P)=a'Q'(a')^{-1}=aQa^{-1}.
\]
Consequently,
\[
u\wh Q u^{-1}
=\Phi(\wh P)
=(va)\wh Q(va)^{-1}.
\]
Set $c=(va)^{-1}u$. Then
\[
c\wh Qc^{-1}=\wh Q,
\qquad
c\in N_{\wh\Delta}(\wh Q)=\wh Q.
\]
Define the integral isomorphism
\[
A=\Inn(a^{-1})\circ f|_P:
P\xrightarrow{\cong}Q.
\]
Since $u=vac$ and $\wh Q$ is abelian, for every $x\in\wh P$ we obtain
\[
\begin{aligned}
u^{-1}\Phi(x)u
&=c^{-1}a^{-1}\wh f(x)ac\\
&=c^{-1}\wh A(x)c\\
&=\wh A(x).
\end{aligned}
\]
Therefore $\Inn(u^{-1})\circ\Phi|_{\wh P}=\wh A$, as required.
\end{proof}

\begin{proof}[Proof of \Cref{thm:lattice-rigidity}]
By Selberg's lemma, choose torsion-free finite index normal subgroups $N_1\normalf\Gamma$ and
$N_2\normalf\Delta$, and choose an integer $m$ large enough that
\[
H_1=K_m(\Gamma)\leq N_1,
\qquad
H_2=K_m(\Delta)\leq N_2.
\]
Then the $H_i$ are torsion-free characteristic finite index subgroups and by \cref{lem:char-functor}
we have $\Phi(\wh H_1)=\wh H_2$. Finite index subgroups inherit their full profinite topologies, so
the restriction is an isomorphism $\wh H_1\to\wh H_2$.

The finite volume hyperbolic manifolds $M_i=\HH^3/H_i$ are either both closed or both cusped, by the
cohomological argument in the proof of \cref{cor:all-integrality}. Applying
\cref{cor:closed-rigidity} in the closed case and \cref{cor:cusped-rigidity} in the cusped case, we
obtain an isomorphism $h:H_1\to H_2$ inducing the restricted map up to a profinite inner
automorphism. Applying \cref{lem:extensions}(ii) to the finite extensions $H_1\normalf\Gamma$ and
$H_2\normalf\Delta$, which is possible by \cref{lem:normalizer-rigidity}, we obtain
\[
\Phi=\Inn(u)\circ\wh f
\]
for an isomorphism $f:\Gamma\to\Delta$.

Since Mostow--Prasad rigidity for finite volume hyperbolic orbifolds realizes $f$, up to a discrete
inner automorphism, by an isometry, possibly reversing orientation \cite{Mostow1968,Prasad1973}, we
obtain the realization of $\Phi$ described in the theorem and the isomorphism of the discrete
groups, as well as the surjectivity of $\Out(\Gamma)\to\Out(\wh\Gamma)$.

Finally, for injectivity, let $\alpha\in\Aut(\Gamma)$ and suppose $\wh\alpha=\Inn(u)$ for
$u\in\wh\Gamma$. Choose a torsion-free characteristic finite index subgroup $H\normalf\Gamma$. Since
$H$ has finite index in $\Gamma$, the natural map induces an isomorphism
$\Gamma/H\cong\wh\Gamma/\wh H$. Choose $\gamma\in\Gamma$ representing the coset $u\wh H$, and set
$v=\gamma^{-1}u\in\wh H$. The equality $\wh\alpha=\Inn(u)$ means $uhu^{-1}=\alpha(h)$ for $h\in H$.
Since $H$ is characteristic in $\Gamma$, we have $\alpha(H)=H$, and therefore $vHv^{-1}=H$. Thus $v$
normalizes the discrete subgroup $H$ of $\wh H$, and by \cref{lem:normalizer-rigidity},
$v\in N_{\wh H}(H)=H$. Consequently
\[
u=\gamma v\in\Gamma,
\qquad
\alpha(x)=uxu^{-1}\quad(x\in\Gamma).
\]
Hence $\alpha\in\Inn(\Gamma)$, proving injectivity of $\Out(\Gamma)\longrightarrow\Out(\wh\Gamma)$.
\end{proof}


\section{Compact core realization}\label{sec:nonlattice}
In this section we assume that equivariant isomorphisms of residual complexes are given at a
cofinal family of characteristic quotients. For the conclusion about the groups it is enough to have
a graph whose inclusion is surjective on fundamental groups. If we keep the full simplex face
residual complex, we also obtain a PL homeomorphism of the chosen compact cores. These conclusions
do not use \cref{thm:lattice-rigidity}.
\begin{corollary}\label{cor:graph-kleinian}
For $i=1,2$, let $\Gamma_i<\PSL_2(\C)$ be finitely generated torsion-free Kleinian groups, and set
$M_i=\HH^3/\Gamma_i$. We choose compact cores $C_i\subset M_i$, finite triangulations of $C_i$, and
connected subgraphs $Y_i\subseteq C_i^{(1)}$ whose inclusions induce surjections on fundamental
groups. For the universal covers $p_i:\widetilde C_i\to C_i$, we set $X_i=p_i^{-1}(Y_i)$, which are
connected. Suppose $\Phi:\wh\Gamma_1\xrightarrow{\cong}\wh\Gamma_2$ and, for an increasing unbounded
sequence $(m_j)$, there are $\Phi_{m_j}$-equivariant isomorphisms
\[
\Dgraph\bigl(K_{m_j}(\Gamma_1)\backslash X_1\bigr)
\xrightarrow{\cong}
\Dgraph\bigl(K_{m_j}(\Gamma_2)\backslash X_2\bigr).
\]
Then there are $f:\Gamma_1\xrightarrow{\cong}\Gamma_2$ and $u\in\wh\Gamma_2$ such that
\begin{equation}\label{eq:graph-kleinian-realization}
\Phi=\Inn(u)\circ\wh f.
\end{equation}
The $\Gamma_i$ may have infinite covolume and need be neither convex cocompact nor geometrically
finite.
\end{corollary}
\begin{proof}
Finitely generated Kleinian groups are linear in characteristic zero and residually finite by
Mal'cev's theorem \cite{Nica2013}. Torsion-freeness makes $M_i$ an orientable hyperbolic
$3$-manifold. Scott's compact core theorem provides a compact core $C_i\subset M_i$ whose inclusion
induces an isomorphism $\pi_1(C_i)\xrightarrow{\cong}\pi_1(M_i)=\Gamma_i$ \cite{ScottCore1973}. The
compact $3$-manifold $C_i$ admits a finite triangulation by Moise's theorem \cite{Moise1952}. We
then use \cref{thm:graph-core-realization}.
\end{proof}

\begin{corollary}\label{cor:graph-cayley}
Let $\Gamma_i<\PSL_2(\C)$ be finitely generated Kleinian groups, with torsion allowed. We choose
finite symmetric generating sets $T_i\subseteq\Gamma_i\setminus\{1\}$. Their directed Cayley graphs
have
\[
\begin{gathered}
V(X_i)=\Gamma_i,\qquad A(X_i)=\Gamma_i\times T_i,\\
s(g,a)=g,\qquad t(g,a)=ga,\qquad\cl{(g,a)}=(ga,a^{-1}).
\end{gathered}
\]
in the notation of \cref{sec:graph-occurrence}: $V(X_i)$ and $A(X_i)$ are the vertex and directed
edge sets, $s$ and $t$ the endpoint maps, and the bar denotes edge reversal. Let the groups act by
left multiplication, and let $\Phi:\wh\Gamma_1\xrightarrow{\cong}\wh\Gamma_2$. If the residual
complexes of the finite quotients admit a cofinal family as in \cref{def:graph-cofinal-matches},
then \eqref{eq:graph-kleinian-realization} holds for a discrete isomorphism $f$ and some
$u\in\wh\Gamma_2$.
\end{corollary}
\begin{proof}
Residual finiteness follows from Mal'cev's theorem. The directed graphs are connected and locally
finite. The left actions are cocompact and free on vertices and directed edges. We apply
\cref{cor:graph-free}. Both orientations upstairs remain distinct even for generators of order two,
so reversal has no fixed directed edge in $X_i$, while the induced reversal in a finite quotient may
have fixed points, as allowed in \cref{def:graph-diagram}.
\end{proof}

\begin{corollary}\label{cor:graph-compact-core-pl}
For the groups and triangulated compact cores of \cref{cor:graph-kleinian}, suppose the universal
covers of the full core triangulations satisfy the conditions of \cref{thm:graph-pl-realization},
including \eqref{eq:graph-simplex-finite}. Then \eqref{eq:graph-kleinian-realization} holds with $f$
induced, up to basepoint conjugacy, by a PL homeomorphism $C_1\to C_2$. If finite labelled families
of boundary components or other subcomplexes are distinguished, assume also that the finite maps
preserve the quotient images of their full inverse images by permutations in a fixed finite
permutation group. The compatible maps and the descended homeomorphism then preserve these families
according to a common permutation.
\end{corollary}
\begin{proof}
We apply \cref{thm:graph-pl-realization} to the chosen compact core triangulations.
\end{proof}

\subsection*{Boundary subgroups and compact core pairs}

\begin{definition}\label{def:graph-boundary}
Let $C$ be a compact connected triangulated $3$-manifold with boundary subcomplex
$\partial C=S_1\sqcup\cdots\sqcup S_r$, where $S_j$ are its boundary components. We choose a
connected graph $Y\subseteq C^{(1)}$ containing $(\partial C)^{(1)}$ such that its inclusion induces
a surjection on fundamental groups. For the universal cover $p:\widetilde C\to C$, we set
\[
X=p^{-1}(Y),\qquad Z_j=p^{-1}\bigl(S_j^{(1)}\bigr)\subseteq X,
\]
and take $(Z_1,\ldots,Z_r)$ as the distinguished family. Such a graph exists: start with
$(\partial C)^{(1)}$, join its components by finitely many paths, and adjoin edge loops whose images
generate $\pi_1(C)$. If $\partial C$ is empty, use any graph from \cref{lem:graph-full-lift} and the
empty distinguished family. The subgraphs $Z_j$ are full inverse images, with all of their
components, and are in general disconnected.
\end{definition}

\begin{theorem}
\label{thm:graph-boundary-realization}
For $i=1,2$, let $\Gamma_i<\PSL_2(\C)$ be finitely generated torsion-free Kleinian groups with
triangulated compact cores $C_i$ and distinguished boundary subgraphs as in
\cref{def:graph-boundary}. We write $\partial C_i=S_{i,1}\sqcup\cdots\sqcup S_{i,r}$, and let
$P_{i,j}\leq\Gamma_i$ be the image of $\pi_1(S_{i,j})$, using fixed base paths. Let
$\Phi:\wh\Gamma_1\xrightarrow{\cong}\wh\Gamma_2$. Suppose that at an increasing unbounded sequence
$(m_k)$ there are $\Phi_{m_k}$-equivariant directed edge residual complex isomorphisms that map
the quotient images of the boundary subgraphs to one another according to permutations
$\sigma_k\in\mathfrak S_r$, unrelated for different $k$.

Then compatible maps can be selected, all with the same permutation $\sigma\in\mathfrak S_r$, and
there are $f:\Gamma_1\xrightarrow{\cong}\Gamma_2$ and $u\in\wh\Gamma_2$ such that
\begin{equation}\label{eq:graph-boundary-realization}
\Phi=\Inn(u)\circ\wh f,\qquad
f(P_{1,j})=d_jP_{2,\sigma(j)}d_j^{-1}\quad(1\leq j\leq r)
\end{equation}
for suitable $d_j\in\Gamma_2$, whether or not the boundary components are incompressible.
\end{theorem}
\begin{proof}
We include the boundary permutation as part of the definition of the finite isomorphisms. By
\cref{thm:graph-compatible-selection,thm:graph-realization}, a compatible limit involves a single
permutation and, after normalization, restricts to an $f$-equivariant original graph isomorphism
$F':X_1\to X_2$ with $F'(Z_{1,j})=Z_{2,\sigma(j)}$.

A component $\widetilde Z_{1,j}$ of $Z_{1,j}$ has stabilizer conjugate to the image of
$\pi_1(S_{1,j}^{(1)})$ in $\Gamma_1$ by covering spaces. The inclusion
$S_{1,j}^{(1)}\hookrightarrow S_{1,j}$ is surjective on fundamental groups, so this image is
$P_{1,j}$. Its image under $F'$ is a component of $Z_{2,\sigma(j)}$, and equivariance gives
\[
f\bigl(\Stab_{\Gamma_1}(\widetilde Z_{1,j})\bigr)
=\Stab_{\Gamma_2}\bigl(F'(\widetilde Z_{1,j})\bigr).
\]
The right side is a conjugate of $P_{2,\sigma(j)}$. Returning to the fixed subgroup representatives
gives the elements $d_j$ and \eqref{eq:graph-boundary-realization}.
\end{proof}

\begin{corollary}\label{cor:graph-core-pairs}
Under the group and compact core conditions of \cref{thm:graph-boundary-realization}, replace the
condition on directed edges by the condition on all simplex faces of
\cref{thm:graph-pl-realization}, with the boundary subcomplexes distinguished. The normalized
realization then descends to a PL homeomorphism
\[
(C_1,\partial C_1)\xrightarrow{\cong}(C_2,\partial C_2)
\]
mapping $S_{1,j}$ onto $S_{2,\sigma(j)}$ and inducing $f$ up to basepoint conjugacy.
\end{corollary}
\begin{proof}
Use the conclusion concerning distinguished subcomplexes of \cref{thm:graph-pl-realization}.
\end{proof}

\begin{corollary}\label{cor:graph-boundary-filling}
Assume \cref{thm:graph-boundary-realization}. Suppose the selected boundary components $S_{i,j}$ are
closed orientable surfaces of genus at least two that are incompressible, in the sense that the
inclusion induces an injection $\pi_1(S_{i,j})\to\pi_1(C_i)$, and that each $P_{i,j}$ inherits its
full profinite topology from $\Gamma_i$. We write $p_i:\widetilde C_i\to C_i$ for the universal
covering map and choose ordered nonseparating filling pairs $(\alpha_{i,j},\beta_{i,j})$, with
cyclic subgroups $A_{i,j},B_{i,j}\leq P_{i,j}$. After a common subdivision, represent the two curves
by embedded circle subgraphs $Q^\alpha_{i,j},Q^\beta_{i,j}\subseteq S_{i,j}^{(1)}$. In addition to
the boundary subgraphs, distinguish their full lifts
\[
p_i^{-1}(Q^\alpha_{i,j}),\qquad p_i^{-1}(Q^\beta_{i,j})
\]
and require the finite maps to preserve the two curve labels in order over the matched boundary
components.

We choose $f,u,d_j$ in \eqref{eq:graph-boundary-realization} using matched lifted boundary
components, and set
\begin{equation}\label{eq:graph-boundary-normalization}
h_j=\Inn(d_j^{-1})\circ f|_{P_{1,j}},\qquad c_j=ud_j.
\end{equation}
Then $h_j:P_{1,j}\xrightarrow{\cong}P_{2,\sigma(j)}$ sends the conjugacy classes of
$A_{1,j},B_{1,j}$ to the given classes, and
\begin{equation}\label{eq:graph-surface-restriction}
\varphi_j=\Inn(c_j^{-1})\circ\Phi|_{\wh{P_{1,j}}}
=\wh h_j:\wh{P_{1,j}}\xrightarrow{\cong}\wh{P_{2,\sigma(j)}}.
\end{equation}
\Cref{thm:finite-cone-main,thm:core-main}, together with \cref{cor:two-curve}, determine the
completed vertex groups of the splittings along the curves and both full profinite Bass--Serre
trees. They also give an equivariant isomorphism between the completed cubical residual complexes
of the two dual square complexes. After a single left translation in $\wh{P_{2,\sigma(j)}}$, this
isomorphism identifies the original dual square complexes and the dense surface groups.
\end{corollary}
\begin{proof}
We choose components $\widetilde Z_{1,j}$ and $\widetilde Z_{2,\sigma(j)}$ whose stabilizers are
$P_{1,j}$ and $P_{2,\sigma(j)}$, respectively. Let $F'$ be the normalized $f$-equivariant graph
isomorphism, and let $d_j\in\Gamma_2$ be as in the proof of \cref{thm:graph-boundary-realization},
so that $F'(\widetilde Z_{1,j})=d_j\widetilde Z_{2,\sigma(j)}$ and
$f(P_{1,j})=d_jP_{2,\sigma(j)}d_j^{-1}$. For $h_j$ defined by
\eqref{eq:graph-boundary-normalization}, the map
$d_j^{-1}F'|_{\widetilde Z_{1,j}}: \widetilde Z_{1,j}\longrightarrow\widetilde Z_{2,\sigma(j)}$ is
$h_j$-equivariant, since for $x\in P_{1,j}$ and $z\in\widetilde Z_{1,j}$ we have
$d_j^{-1}F'(xz)=h_j(x)\bigl(d_j^{-1}F'(z)\bigr)$. By incompressibility and covering space theory,
the components of the distinguished curve lifts inside each chosen boundary component have
stabilizers conjugate, within its surface group, to the corresponding subgroups $A_{i,j}$ and
$B_{i,j}$. Choose a component of the $\alpha_{1,j}$-lifts whose stabilizer is $A_{1,j}$. Its image
under $d_j^{-1}F'$ is a component of the $\alpha_{2,\sigma(j)}$-lifts, because the curve labels are
preserved in order. Equivariance identifies their stabilizers. The same argument applies to the
$\beta$-lifts. Consequently, there exist $s_j,t_j\in P_{2,\sigma(j)}$ such that
$h_j(A_{1,j})=s_jA_{2,\sigma(j)}s_j^{-1}$, and $h_j(B_{1,j})=t_jB_{2,\sigma(j)}t_j^{-1}$.

Since $P_{i,j}$ inherits its full profinite topology, $\wh{P_{i,j}}$ is identified with the closure
of $P_{i,j}$ in $\wh\Gamma_i$, and completing $f|_{P_{1,j}}=\Inn(d_j)\circ h_j$ gives
$\wh f|_{\wh{P_{1,j}}}=\Inn(d_j)\circ\wh h_j$. Thus \eqref{eq:graph-boundary-realization} and
\eqref{eq:graph-boundary-normalization} imply $\Phi|_{\wh{P_{1,j}}}=\Inn(ud_j)\circ\wh h_j$. As
$c_j=ud_j$, this says $\Inn(c_j^{-1})\circ \Phi|_{\wh{P_{1,j}}}=\wh h_j$, which is
\eqref{eq:graph-surface-restriction}.

The remaining assertions follow from \cref{thm:finite-cone-main}, applied to the completed cyclic
subgroups of the two curves, from \cref{thm:core-main}, applied to the ordered filling pairs, and
from \cref{thm:discrete-star} as used in \cref{cor:two-curve}.
\end{proof}

\subsection*{Schottky obstruction}
Finally, we give an elementary fact which is an easy consequence of \cite{RibesZalesskii2010}.

\begin{proposition}\label{prop:graph-schottky}
For every $r\geq2$ there is a classical Schottky group $\Gamma<\PSL_2(\C)$ of rank $r$ and an
automorphism $\Psi\in\Aut(\wh\Gamma)$ for which there are no $w\in\wh\Gamma$ and $f\in\Aut(\Gamma)$
such that $\Psi=\Inn(w)\circ\wh f$.
\end{proposition}
\begin{proof}
Let $\Gamma$ be any classical Schottky group of rank $r\geq2$, and choose a free basis
$x_1,\ldots,x_r$. Then $\wh\Gamma$ is the free profinite group on this basis
\cite{RibesZalesskii2010}.
Choose $\lambda\in\hZ^\times\setminus\{1,-1\}.$
By the universal property, the assignments
\[
\Psi_\lambda(x_1)=x_1^\lambda,
\qquad
\Psi_\lambda(x_j)=x_j
\quad(2\leq j\leq r)
\]
define a continuous endomorphism of $\wh\Gamma$. Continuous homomorphisms preserve profinite powers,
so the map $\Psi_{\lambda^{-1}}$ is its inverse and $\Psi_\lambda$ is an automorphism.

On continuous abelianization $\wh\Gamma^{\mathrm{ab}}\cong\hZ^r$, this automorphism induces
\[
(\Psi_\lambda)_{\mathrm{ab}}
=\operatorname{diag}(\lambda,1,\ldots,1).
\]
Suppose that $\Psi_\lambda=\Inn(w)\circ\wh f$ for some $w\in\wh\Gamma$ and $f\in\Aut(\Gamma)$. Inner
automorphisms act trivially on abelianization, whereas $f$ induces a matrix
$f_{\mathrm{ab}}\in\GL_r(\Z)$. Taking determinants therefore gives
\[
\lambda=\det(f_{\mathrm{ab}})\in\{1,-1\},
\]
a contradiction.
\end{proof}

\begin{corollary}
\label{cor:graph-schottky-outer}
Let $\Gamma$ be as in \cref{prop:graph-schottky}, acting cocompactly on a connected locally finite
directed multigraph $X$ with finite vertex stabilizers, and fix a finite labelled distinguished
family $\mathcal S$. Then
\[
\Out(\Gamma\acts X,\mathcal S)\xrightarrow{\cong}
\Out\bigl(\wh\Gamma\acts\wh{\Dgraph}_\Gamma(X),\wh{\mathcal S}\bigr),
\]
although $\Out(\Gamma)\to\Out(\wh\Gamma)$ is not surjective.
\end{corollary}
\begin{proof}
The action statement is \cref{thm:graph-action-outer}. The automorphism of
\cref{prop:graph-schottky} shows that $\Out(\Gamma)\to\Out(\wh\Gamma)$ is not surjective.
\end{proof}

\begin{samepage}
\begin{center}
\textbf{Acknowledgement}
\end{center}
I thank Benson Farb and Dave Gabai for their unwavering support over the years and deeply insightful mathematical discussions, and Peter Shalen
for his constant encouragement and for his contributions to the special issue that we are co-editing. I am
grateful to Ying Zhou and Yaowu Zhou  for their help throughout this work.
\par\medskip

\noindent\texttt{henry.yhou@gmail.com}
\end{samepage}


\end{document}